\documentclass[11pt]{article}
\usepackage{amsmath,latexsym,amssymb,amsfonts,amsbsy, amsthm}
\usepackage{graphicx}
\usepackage{empheq}
\usepackage{color} 
\usepackage{ulem}
\def\inte#1{
\displaystyle\mathop{#1\kern0pt}^\circ }

\def\virgp{\raise 2pt\hbox{,}}
\def\cdotpv{\raise 2pt\hbox{;}}

\def\eqdefa{\buildrel\hbox{\footnotesize def}\over =}

\def\C{\mathop{\bf C\kern 0pt}\nolimits}
\def\DD{\mathop{\bf D\kern 0pt}\nolimits}
\def\K{\mathop{\bf K\kern 0pt}\nolimits}
\def\N{\mathop{\bf N\kern 0pt}\nolimits}
\def\Q{\mathop{\bf Q\kern 0pt}\nolimits}
\def\R{\mathop{\bf R\kern 0pt}\nolimits}
\def\SS{\mathop{\bf S\kern 0pt}\nolimits}
\def\ZZ{\mathop{\bf Z\kern 0pt}\nolimits}
\def\TT{\mathop{\bf T\kern 0pt}\nolimits}

\def\na{\nabla}

\newcommand{\beq}{\begin{equation}}
\newcommand{\eeq}{\end{equation}}
\newcommand{\ben}{\begin{eqnarray}}
\newcommand{\een}{\end{eqnarray}}
\newcommand{\beno}{\begin{eqnarray*}}
\newcommand{\eeno}{\end{eqnarray*}}
\newtheorem{defi}{Definition}[section]

\newtheorem{thm}[defi]{Theorem}
\newtheorem{lem}[defi]{Lemma}
\newtheorem{rmk}[defi]{Remark}

\newtheorem{prop}[defi]{Proposition}
\renewcommand{\theequation}{\thesection.\arabic{equation}}

\begin{document}

\title{Linear instability of $Z$-pinch in plasma (II): Viscous case}
\author{Dongfen Bian \footnote{ School of Mathematics and Statistics,
 Beijing Institute of Technology, Beijing 100081, China; Division of Applied Mathematics, Brown University, Providence, Rhode Island 02912. Email: {\tt dongfen\_bian@brown.edu/biandongfen@bit.edu.cn}.} \and
Yan Guo\footnote{Division of Applied Mathematics, Brown University, Providence, Rhode Island 02912. Email: {\tt yan\_guo@brown.edu}.} 
\and Ian Tice \footnote{Department of Mathematical Sciences, Carnegie Mellon University,  Pittsburgh, PA 15213. Email: {\tt iantice@andrew.cmu.edu}.} 
}

\maketitle

\begin{abstract} 
	The z-pinch is a classical steady state for the MHD model, 
	where a confined plasma
	fluid is separated by vacuum, in the presence of a magnetic field which is generated
	by a prescribed current along the $z$ direction. We develop a scaled variational framework to
	study its stability in the presence of viscosity effect, and demonstrate that any such z-pinch is always unstable. We also
	establish the existence of a largest growing mode, which dominates the linear growth of
	the linear MHD system.
\end{abstract}


\vskip 0.2cm


\renewcommand{\theequation}{\thesection.\arabic{equation}}
\setcounter{equation}{0}
\setcounter{tocdepth}{1}
\tableofcontents
\newpage
\section{Introduction}

\subsection{Formulation of the problem in Eulerian coordinates}

In the present paper, we are concerned with the plasma-vacuum MHD system, where the plasma is confined inside a rigid wall and isolated from it by a region of low enough density to be treated as a ``vacuum". This model
describes confined plasmas in a closed vessel, but separated from the wall by a vacuum region.  We consider the cylindrical domain $\big\{(x_1,x_2,z)\in \mathbb{R}^2\times 2\pi \mathbb{T}|x_1^2+x_2^2\leq r_w^2\big\}$, which is meant to model the container holding the plasma and the container divides into two disjoint pieces, $\Omega(t)=\big\{\rho(t)>0\big\}$ and $\Omega^v(t)=\big\{\rho(t)=0\big\}$, with the free boundary $\Sigma_{t, pv} \eqdefa \overline{\Omega(t)} \cap \overline{\Omega^v(t)}$  and the perfectly conducting wall $\Sigma_w$ on the outside $r_w$.
For smooth solutions, the compressible MHD system in the plasma region can be written in Euler coordinates as
\begin{equation}\label{mhd-plasma-1-viscosity}
\begin{cases}
&\partial_t \rho + (u\cdot \nabla) \rho + \rho \nabla\cdot u=0\quad \mbox{in} \quad \Omega(t),\\
&\rho(\partial_t u + (u\cdot \nabla) u) + \nabla (p+\frac{1}{2}|B|^2)- \nabla \cdot \mathbb{S}(u)=(B\cdot \nabla) B \quad \mbox{in} \quad \Omega(t),\\
&\partial_t B -\nabla \times (u \times B)=0 \quad \mbox{in} \quad \Omega(t),\\
&\nabla \cdot B=0\quad \mbox{in} \quad \Omega(t),
\end{cases}
\end{equation}
where the vector-field $u = (u_1, u_2,u_3)$ denotes the Eulerian plasma velocity field, $\rho$ denotes the density of the fluid, $B=(B_1, B_2,B_3)$ is magnetic field, and $p$ denotes the pressure function. The above system \eqref{mhd-plasma-1-viscosity} is called compressible MHD equations which describe the motion of a perfectly conducting fluid interacting with a magnetic field.
Here, the open, bounded subset $\Omega(t) \subset \mathbb{R}^3$ denotes the changing volume occupied by
	the plasma with  $\rho(t)>0$ in $\Omega(t)$.    The strain tensor $\mathbb{D}(u)$ is defined as twice the symmetric part of the gradient of the velocity $u$, that is, $\mathbb{D}(u)=\nabla u+\nabla u^{T}$, with $\nabla \cdot u$  the rate of expansion of the plasma. The deviatoric (trace-free) part of the strain tensor $\mathbb{D}(u)$ is then $\mathbb{D}^0(u)=\mathbb{D}(u)-\frac{2}{3} \mbox{div}\, u\, \mathbb{I}$, where $\mathbb{I}$ is the  identity matrix. The viscous stress tensor in fluid is then given by $\mathbb{S}(u)=\varepsilon\mathbb{D}^0(u)+\delta  \mbox{div}\, u\, \mathbb{I}$, where dynamic viscosity $\varepsilon$ and bulk viscosity $\delta$ are constants. 
We have here considered the polytropic gases, the constitutive relation, which is also called the
equation of state, and is given by $p=A\,\rho^{\gamma}$, where $A$ is an entropy constant and $\gamma>1$ is the adiabatic gas exponent.

From the mass conservation equation in \eqref{mhd-plasma-1-viscosity} and pressure satisfying $\gamma $ law, one can get that 
\begin{equation}\label{pressure}
\partial_t p + u\cdot \nabla p + \gamma p \nabla\cdot u=0.
\end{equation}

In the vacuum domain $\Omega^v(t)$, we have the div-curl system
\begin{equation}\label{div-curl-1}
  \begin{cases}
 & \nabla \cdot \widehat{B}=0 \quad \mbox{in} \quad \Omega^v(t),\\
 & \nabla \times \widehat{B}=0 \quad \mbox{in} \quad \Omega^v(t)
  \end{cases}
\end{equation}
which describes the vacuum magnetic field $\widehat{B}$. Here, we consider so-called pre-Maxwell dynamics. That is, as usual in nonrelativistic
MHD, we neglect the displacement current $\frac{1}{c^2} \partial_t \widehat{E}$, where $c$ is the speed of the light
and $\widehat{E}$ is the electric field. In general, quantities with a hat $\widehat{\cdot}$ denote vacuum variables.

We assume that the plasma region $\Omega(t)$ is isolated from the fixed perfectly conducting wall $\Sigma_w$ by a vacuum region $\Omega^v(t)$, which makes the plasma surface free to move. Hence, this model is a free boundary problem of the combined plasma-vacuum system.
To solve this system, we need to prescribe appropriate boundary conditions.
On the perfectly conducting wall $\Sigma_w$, the normal component of the magnetic field must vanish:
\begin{equation}\label{wall-bdry-vacuum-1}
 n\cdot \widehat{B}|_{\Sigma_w}=0,
\end{equation}
where $n$ is the outer unit normal to the boundary of $ \Sigma_w$.

We prescribe the following jump conditions on the free boundary to connect the magnetic fields across the surface. These arise
from the divergence $B$ equation, and the momentum
equations:
\begin{equation}\label{bdry-plasma-vacuum-1}
\begin{cases}
 & n \cdot B=n \cdot \widehat{B} \quad \mbox{on} \quad \Sigma_{t, pv},\\
 & \bigg[\bigg[p+\frac{1}{2}|B|^2\bigg]\bigg]\, n-\mathbb{S}(u)n  =0 \quad \mbox{on} \quad \Sigma_{t, pv},
\end{cases}
\end{equation}
where $[[q]]_{\Sigma_{t, pv}}$ denotes $\widehat{q}-q$ on the free boundary $\Sigma_{t, pv}$, and $n $ is the outer normal to the boundary of $ \Omega(t)$.

In conclusion, denote $\mathcal{V}(\Sigma_{t, pv})$ as the normal velocity of the free surface $\Sigma_{t, pv}$, the plasma-vacuum compressible MHD system can be written in Eulerian coordinates as
\begin{equation}\label{mhd-plasma-E-1-viscosity}
\begin{cases}
&\partial_t \rho + \nabla\cdot( \rho \, u)=0\quad \mbox{in} \quad \Omega(t),\\
&\rho(\partial_t u + (u\cdot \nabla) u) + \nabla (p+\frac{1}{2}|B|^2)- \nabla\cdot\mathbb{S}(u)=(B\cdot  \nabla) B \quad \mbox{in} \quad \Omega(t),\\
&\partial_t B - \nabla \times (u \times B)=0, \quad  \nabla \cdot B=0 \quad \mbox{in} \quad \Omega(t),\\
&\nabla \cdot B=0\quad \mbox{in} \quad \Omega(t),\\
& \nabla \cdot \widehat{B}=0, \quad \nabla \, \times \, \widehat{B}=0 \quad \mbox{in} \quad \Omega^v(t),\\
&\mathcal{V}(\Sigma_{t, pv})=u\cdot n \quad \mbox{on} \quad \Sigma_{t, pv},\\
& n\cdot B=n\cdot \widehat{B}\quad \mbox{on} \quad \Sigma_{t, pv},\\
& \Big(p+\frac{1}{2}|B|^2-\frac{1}{2}|\widehat{B}|^2\Big)\, n -\mathbb{S}(u) n=0 \quad \mbox{on} \quad \Sigma_{t, pv},\\
&n\cdot \widehat{B}|_{\Sigma_w}=0,\, \rho|_{t=0}=\rho_0, \, u|_{t=0}=u_0, \, B|_{t=0}=B_0.
\end{cases}
\end{equation}
\subsection{Background}

 The  $z$-pinch instability  in plasma for the compressible MHD system  with vacuum and free boundary is an interesting and long-time open problem since the pinch experiments of the 1960s and 1970s, see \cite{Mikhailovshii, Schmit} and the references therein. There are many numerical simulations \cite{Freidberg,Goedbloed-poedts}. Up to now, there is not rigirous mathematical proof for viscous compressible MHD equations \eqref{mhd-plasma-E-1-viscosity}. 

Note that  Guo-Tice \cite{Guo-Tice-inviscid} and \cite{Guo-TIce-viscous} proved the 
 linear	Rayleigh-Taylor instability for inviscid and viscous compressible fluids by introducing a new variational method. Later on, using variational method,  many authors considered the effects of magnetic field in the fluid equations. Jiang-Jiang \cite{jiang-jiang-ARMA} considered the magnetic inhibition theory
 	in non-resistive incompressible MHD fluids. Jiang-Jiang \cite{jiang-jiang-CV} considered the nonlinear stability and instability in the Rayleigh-Taylor problem of compressible MHD equations without  vacuum and established the stability/instability criteria for the stratified compressible magnetic Rayleigh-Taylor problem in Lagrangian coordinates. Jiang-Jiang \cite{jiang-jiang-PD}  investigated the stability and instability of the Parker problem for the three-dimensional compressible isentropic viscous magnetohy-drodynamic system with zero resistivity in the presence of a modified gravitational force in a vertical strip domain in which the velocity of the fluid is non-slip on the boundary.   Wang-Xin \cite{wang-xin} proved the global well-posedness of the inviscid and resistive problem with surface tension around a non-horizontal uniform magnetic field for two-dimensional incompressible MHD equations. Wang \cite{wang} got sharp nonlinear stability criterion of viscous incompressible non-resistive MHD internal waves  in 3D. Gui \cite{Gui} considered the Cauchy problem of the two-dimensional incomplressible magnetohydrodynamics system with inhomogeneous density and electrical conductivity and has showed the global well-posedness for a generic family of the variations of the initial data and an inhomogeneous electrical conductivity.  All these results do not contain vacuum. For presenting vacuum, under the Taylor sign condition of the total pressure on the free surface, Gu-Wang \cite{Gu-Wang}  proved the local well-posedness of the ideal incompressible MHD equations in Sobolev spaces. Recently, Bian-Guo-Tice \cite{bian-guo-tice-inviscid} have established any $z$-pinch linear instability for the ideal compressible MHD equations \eqref{mhd-plasma-E-1-viscosity} with $\varepsilon=\delta
 	=0$ and containing vaccum. In this paper, we will rigirously prove $z$-pinch linear instability for system \eqref{mhd-plasma-E-1-viscosity} with vacuum.

\section{Steady state and main results  }\label{sect-2}
\subsection{Derivation of the MHD system  in Lagrangian coordinates}\label{susect-free-surface-1}
In this subsection, we mainly  introduce the Lagrangian coordinates in which the free boundary becomes fixed.

First, we assume the equilibrium domains are given by
	\begin{equation*}
	\overline{\Omega}=\{(r, \theta, z)| r<r_0, \, \theta\in[0,2\pi],z\in 2\pi \mathbb{T}\}\,\, \mbox{and}\,\,
	\overline{\Omega}^v=\{(r,\theta, z)|r_0<r<r_w, \,  \theta\in[0,2\pi],z\in 2\pi \mathbb{T}\}.
	\end{equation*} 
	Here, the constant $r_0$ is the interface boundary and the constant $r_w$ is the perfectly conducting wall position.
	This is meant to be a simplified model of the toroidal geometry employed in tokamaks.

Now we introduce the Lagrangian coordinates.

{\bf 1. The flow map}

Let $h(t, \mathcal{X})$ be a position of the gas particle $\mathcal{X}$  in the equilibrium domain $\overline{\Omega}$ at time $t$ so that
\begin{equation}\label{def-flowmap-1}
\begin{cases}
&\frac{d}{dt}h(t, \mathcal{X})=u(t, h(t, \mathcal{X})), \quad t>0, \, \mathcal{X}\in \overline{\Omega},\\
&h|_{t=0}=\mathcal{X}+g_0(\mathcal{X}), \quad \mathcal{X}\in \overline{\Omega}.
\end{cases}
\end{equation}

Then the displacement $g(t, \mathcal{X})\eqdefa h(t, \mathcal{X})-\mathcal{X}$ satisfies
\begin{equation}\label{def-flowmap-2}
\begin{cases}
&\frac{d}{dt}g(t, \mathcal{X})=u(t, \mathcal{X}+g(t,  \mathcal{X})),\\
&g|_{t=0}=g_0.
\end{cases}
\end{equation}
We define the Lagrangian quantities in the plasma as follows (where $\mathcal{X}=(x, y, z)\in \overline{\Omega}$):
\begin{equation*}
\begin{split}
&f(t, \mathcal{X})\eqdefa \rho(t, h(t, \mathcal{X})),\quad v(t, \mathcal{X})\eqdefa u(t, h(t, \mathcal{X})),\quad q(t, \mathcal{X})\eqdefa p(t, h(t, \mathcal{X})), \\
& b(t, \mathcal{X})\eqdefa B(t, h(t, \mathcal{X})), \quad \mathcal{A}\eqdefa (Dh)^{-1}, \quad J \eqdefa \mbox{det}(Dh).
\end{split}
\end{equation*}
According to definitions of the flow map $h$ and the displacement $g$,   for $(i,j,k)\in \{1,2,3\}$ one can get the following identities
\begin{equation}\label{flow-map-identity-1}
\mathcal{A}_{i}^k \partial_{k} h^j=\mathcal{A}_{k}^j \partial_{i} h^k=\delta_i^j, \quad \partial_k(J\mathcal{A}_{i}^k)=0,\quad\partial_{i} h^j=\delta_i^j+\partial_{i} g^j, \quad \mathcal{A}_{i}^j=\delta_i^j-\mathcal{A}_{i}^k \partial_kg^j,
\end{equation}
where the Einstein notation is used and will be used in the whole paper.
If the displacement $g$ is sufficiently small in an appropriate Sobolev space, then the flow mapping $h$ is a
diffeomorphism from $\Omega_0$ to $\Omega(t)$, which  allows us to switch back and forth from Lagrangian to Eulerian coordinates.

{\bf 2. Derivatives of $J$ and $\mathcal{A}$ in Lagrangian coordinates}

We write the derivatives of $J$ and $\mathcal{A}$ in Lagrangian coordinates as follows:
\begin{equation}\label{identity-Lagrangian-1}
\begin{split}
&\partial_t J =J \mathcal{A}_{i}^j \partial_j v^i, \quad \partial_{\ell} J =J \mathcal{A}_{i}^j \partial_j\partial_{\ell} g^i, \quad\partial_t \mathcal{A}_{i}^j=-\mathcal{A}_{k}^j\mathcal{A}_{i}^{\ell} \partial_{\ell}v^k,\\
&\partial_{\ell} \mathcal{A}_{i}^j=-\mathcal{A}_{k}^j\mathcal{A}_{i}^{n} \partial_{n}\partial_{\ell}g^k, \quad \partial_{i} v^j=\partial_{i}h^k \mathcal{A}_{k}^\ell \partial_{\ell}v^j= \mathcal{A}_{i}^\ell \partial_{\ell}v^j+\partial_{i}g^k \mathcal{A}_{k}^\ell \partial_{\ell}v^j.
\end{split}
\end{equation}

{\bf 3. Plasma equations in Lagrangian coordinates}

Denote $(\nabla_{\mathcal{A}})_i=\mathcal{A}_i^j \partial_j$. Then we can write the plasma equations in Lagrangian coordinates as follows
\begin{equation}\label{mhd-fluid-2}
\begin{cases}
& \partial_t g=v\quad \mbox{in}\quad  \overline{\Omega}, \\
& f \partial_t v + \nabla_{\mathcal{A}}\Big(q+\frac{1}{2}|b|^2\Big)-\nabla_{\mathcal{A}} \cdot \mathbb{S}_{\mathcal{A}}(v)=(b\cdot \nabla_{\mathcal{A}}) b \quad \mbox{in}\quad  \overline{\Omega},\\
&\partial_t f+ f \nabla_{\mathcal{A}} \cdot  v=0\quad \mbox{in}\quad  \overline{\Omega},\\
&\partial_t b+ b \nabla_{\mathcal{A}}\cdot v=(b\cdot \nabla_{\mathcal{A}}) v\quad \mbox{in}\quad  \overline{\Omega},\\
&\nabla_{\mathcal{A}} \cdot b=0\quad \mbox{in}\quad  \overline{\Omega},\\
&n\cdot b=n\cdot \widehat{  b}\quad\mbox{on} \quad \Sigma_{0, pv},\\
& \Big(q+\frac{1}{2}|b|^2-\frac{1}{2}|\widehat{b}|^2\Big)\mathcal{N}-\mathbb{S}_{\mathcal{A}}(v)\mathcal{N}=0 \quad\mbox{on} \quad \Sigma_{0, pv},
\end{cases}
\end{equation}
where $\mathbb{S}_{\mathcal{A}}(v)=\varepsilon\Big(\nabla_{\mathcal{A}} v+\nabla_{\mathcal{A}} v^T-\frac{2}{3}\mbox{div}_{\mathcal{A}} \, v \, \mathbb{I}\Big)+\delta\mbox{div}_{\mathcal{A}} \, v \, \mathbb{I}$ and the exterior magnetic field $\widehat{  b}$ satisfies the vacuum equations (5.107) in lagrangian coordinates which can be recalled from Appendix of \cite{bian-guo-tice-inviscid}.

Since $\partial_t J =J \mathcal{A}_{i}^j \partial_j v^i=J\,\nabla_{\mathcal{A}} \cdot  v$ and $J(0)=\det(Dh_0)=\det(I+Dg_0)$, with $I$ the identity matix, we find from the equation of $f$ in \eqref{mhd-fluid-2} that
$f\,J= \rho_0(h_0)\det(I+Dg_0)$,
where $\rho_0$ is given initial density function. Taking $\rho_0$ such that
$\rho_0(h_0)\det(I+Dg_0)=\overline{\rho}$,
we get
\begin{equation}\label{f-q-1}
\begin{split}
f= J^{-1}\,\overline{\rho},\quad q=AJ^{-\gamma}\,\overline{\rho}^{\gamma}.
\end{split}
\end{equation}

On the other hand, we multiply  the magnetic field equation of \eqref{mhd-fluid-2} by  $J \mathcal{A}^T$ to get
\begin{equation*}
\begin{split}
J \mathcal{A}_j^i  \partial_t b^j+ J\mathcal{A}_j^i  b^j \mathcal{A}_k^h \partial_hv^k=J \mathcal{A}_j^i b^h \mathcal{A}_h^k\partial_k v^j,
\end{split}
\end{equation*}
which along with \eqref{identity-Lagrangian-1} implies
\begin{equation*}
\begin{split}
&\partial_t(\mathcal{A}_j^i J b^j) =J \mathcal{A}_j^i  \partial_t b^j+\mathcal{A}_j^i  b^j\partial_t J +J b^j\partial_t \mathcal{A}_j^i =J \mathcal{A}_j^i  \partial_t b^j+ J\mathcal{A}_j^i  b^j\mathcal{A}_k^h \partial_hv^k-J b^j\mathcal{A}_{k}^i\mathcal{A}_{j}^{h} \partial_{h}v^k=0.
\end{split}
\end{equation*}
Therefore, we have
\begin{equation}\label{b-identity-1}
\begin{split}
J b^j \mathcal{A}_j^i=J(0) b^j_0 \mathcal{A}_j^i(0)=\det(I+Dg_0)B^j_0(h_0) \mathcal{A}_j^i(0),
\end{split}
\end{equation}
where $B_0$ is given initial magnetic field. Taking $B_0$ such that
$\det(I+Dg_0)B^j_0(h_0) \mathcal{A}_j^i(0)=\overline{B}^i$,
we obtain from \eqref{b-identity-1} that
\begin{equation*}\label{b-identity-2}
\begin{split}
b^k =J^{-1}\overline{B}^i \partial_ih^k=J^{-1}\overline{B}^k+J^{-1}\overline{B}^i \partial_ig^k.
\end{split}
\end{equation*}

\subsection{The equilibrium for the $z$-pinch plasma}
We recall from \cite{bian-guo-tice-inviscid} a $z$-pinch steady state of 
 $\overline{ u}=0$, $  B=(0,B_\theta(r),0)$, $p=p(r)$, $\widehat{  B}=(0,\widehat{B}_\theta(r),0)$.
	Assume the equilibrium domains are given by
	\begin{equation*}
	\overline{\Omega}=\{(r, \theta, z)| r<r_0, \, \theta\in[0,2\pi],z\in 2\pi  \mathbb{T}\}\,\, \mbox{and} \,\, \overline{\Omega}^v=\{(r,\theta, z)|r_0<r<r_w, \,  \theta\in[0,2\pi],z\in 2\pi \mathbb{T}\}.
	\end{equation*}
Here, the constant $r_0$ is the interface boundary and the consatnt $r_w$ is the perfectly conducting wall position.
This is meant to be a simplified model of the toroidal geometry employed in tokamaks.

Then by Lemma 2.1, Lemma 2.4, Lemma 2.5, Lemma 2.6 and Corollary 2.7  in Bian-Guo-Tice \cite{bian-guo-tice-inviscid}, we can get the following Proposition describing the steady solution.
\begin{prop}
		\label{steady-lem}
	Assume that the function $p(r)$ satisfies $p(r)\geq 0$ and $p(r)=0$ if and only if $r=r_0$, and \begin{equation}\label{new-condition}
	-\int_0^rs^2p'(s)ds \geq 0 \,\, \mbox{ for all} \quad 0\leq r\leq r_0, \quad p(r)\in C^{2,1}([0,r_0]).
	\end{equation}
	Then the cylindrically symmetric steady solution $\overline{ u}=0$, $ B=B(r)$, $\mathbb{J}_z=\mathbb{J}_z(r)$, $\widehat{  B}=\widehat{B}(r)$ with a  function $p(r)$ taking the form of 
		\begin{equation}\label{z-pinch-1-1}
		\begin{cases}
		&B_r= 0, \quad B_z= 0, \quad  B_\theta(r)=\Big(-\frac{2}{r^2}\int_0^rs^2 p'(s)ds\Big)^{\frac{1}{2}},
		\\
		&\mathbb{J}_z(r)
		=\frac{1}{2}\Big(-\frac{2}{r^2}\int_0^rs^2 p'(s)ds\Big)^{-\frac{1}{2}}\Big(\frac{4}{r^3}\int_0^rs^2p'(s)ds-2p'(r)\Big)\\
		&\quad\quad\quad+\frac{1}{r}\Big(-\frac{2}{r^2}\int_0^rs^2p'(s)ds\Big)^{\frac{1}{2}}\quad \mbox{in} \quad \overline{\Omega},\\
		& \widehat{B}_r= 0, \quad \widehat{B}_z= 0, \quad \widehat{B}_{\theta}(r) = B_{\theta}(r_0)\frac{r_0}{r} \quad \mbox{in} \quad \overline{\Omega}^v,
		\end{cases}
		\end{equation}
		solves the  equilibrium equations in plasma domain,
		\begin{equation}\label{steady-equ-plasma}
		\nabla p=\mathbb{J}\times B, \qquad\nabla \cdot B=0,\qquad
		\mathbb{J}=\nabla\times B,
		\end{equation}
		and the system \eqref{div-curl-1} in the vacuum region. 
		We can define the equilibrium density
		\begin{equation*}
		\rho(r) = \left(\frac{p(r)}{A} \right)^{1/\gamma}.
		\end{equation*}
Moreover, 
			we have 
			\begin{equation}\label{cond-jz}
			\mathbb{J}_z\in C^{1,1}([0,r_0]),\quad
			B_\theta\in C^{1,1}([0,r_0]),
			\end{equation}
			and the steady solution satisfies the following properties
	
		(i)	There exists $r_* \in (0,r_0)$ such that
			\begin{equation}\label{cond}
			p'(r_*) + \frac{2 \gamma p(r_*) {B^2_\theta(r_*)} }{r_*(\gamma p(r_*) + {B^2_\theta(r_*)} )} <0.
			\end{equation}
		
		(ii)	For $|m|\geq 2$, suppose that $\mathbb{J}_z: [0,r_0] \to [0,\infty)$ is non-increasing, then we have 
			\begin{equation*}
			2 p'(r) + m^2 \frac{{B^2_\theta(r)}}{r} \ge 0 \,\, \, \,\mbox{for all }\,\, r\in [0,r_0].
			\end{equation*}
			
		(iii)  For $|m|\geq 2$, we define $\alpha = (m^2-2)/2 \ge 1$.  Suppose that $\beta > \alpha\ge 0$ and $\mathbb{J}_z$ vanishes to order $\beta$ at the origin $r=0$ in the sense that 
		$ |\mathbb{J}_z(r)|\leq Cr^\beta$, and further suppose that  $B_\theta \neq 0$ in $(0,r_0]$, i.e. $B_\theta$ has a sign. Then
		there exists $r^* \in (0,r_0)$ such that  
		\begin{equation*}
		2p'(r^*) + m^2 \frac{B^2_\theta(r^*)}{r^*} <0.
		\end{equation*} 
		
	(iv) If $\mathbb{J}_z \ge 0$ and $\mathbb{J}_z$ is compactly supported in $(0,r_0)$, then for each $m \in \mathbb{Z} \backslash \{0\}$,  there exists $r^* \in (0,r_0)$ such that 
	\begin{equation*}
	2p'(r^*) + m^2 \frac{B^2_\theta(r^*)}{r^*} <0.
	\end{equation*}
\end{prop}

We remark that in Proposition \ref{steady-lem}, the property (i)  implies unconditional instability of any $z$-pinch equilibrium for $m=0$. The property (ii)  implies absence of instability for $|m|\geq 2$ for a general class of $z$-pinch equilibria. The property (iii) implies the instability for $|m|\geq 2$ for some class of $z$-pinch equilibria.
And the property (iv) implies the instability for any $m \in \mathbb{Z} \backslash \{0\}$ for some class of $z$-pinch equilibria.

From  Proposition \ref{steady-lem}, we know that at the plasma-vacuum interface,  the steady solution $B(r)$  in  cylindrical $ r$, $\theta$, $z$-coordinates satisfies naturally 
\begin{equation*}
n_0\cdot B=n_0 \cdot \widehat{B}=0, \quad \mbox{on} \quad \Sigma_{0, pv},
\end{equation*} due to  $n_0=e_r$, $B=(0, B_\theta(r),0)$ and $\widehat{B}=(0, B_\theta(r_0)\frac{r}{r_0},0)$.

	Now we introduce the admissibility of the pressure $p$, which will be used in the following sections.	
\begin{defi}\label{admissible}
	We say that $p$ is admissible if $p(r) \geq 0$ for all $r \in [0,r_0]$ and $p(r)=0$ if and only if $r=r_0$, $p'(r)\leq 0$ for $r$ near $r_0$, that is, $p'(s)\leq 0$ for $s\in (r_0-\epsilon, r_0]$ with small constant $\epsilon>0$,  and $p(r)$ satisfies  \eqref{new-condition} and
	\begin{equation}\label{minimizer-assumption}
	\quad \lim_{r\rightarrow r_0}\frac{p(r)}{p'(r)}=0.
	\end{equation}
\end{defi}
	
\subsection{Lagrangian formulation and main results}
From the Lagrangian formulation in Appendix of \cite{bian-guo-tice-inviscid}, the linearized viscous compressible MHD system in a perturbation formulation around the steady solution in \eqref{z-pinch-1-1} takes the following form
\begin{equation}\label{linear-perturbation-and-boundary-viscosity}
\begin{cases}
&\partial_tg=v\quad \mbox{in}\quad \overline{\Omega},\\
& \rho\partial_{tt} g =\nabla(g\cdot\nabla p+\gamma p\nabla\cdot g)+(\nabla \times B)\times [\nabla \times (g \times B)]\\
&\quad+\{\nabla \times [\nabla \times (g \times B)]\}\times B+ \mbox{div} \Big (\varepsilon\Big(\nabla v+\nabla v^T-\frac{2}{3}\mbox{div}\,  v \, \mathbb{I}\Big)+\delta \mbox{div} \,v\, \mathbb{I}\Big),\,\, \mbox{in}\,\, \overline{\Omega},\\
&  \nabla \cdot \widehat{Q}=0,\quad \mbox{in}\quad \overline{\Omega}^v,\\
& \nabla \times \widehat{Q}=0,\quad \mbox{in}\quad \overline{\Omega}^v,\\
& n \cdot \nabla \times (g\times \widehat{B})=n\cdot \widehat{Q}, \quad \mbox{on} \quad \Sigma_{0, pv},\\
&\Big[\Big(-\gamma p\nabla\cdot g+B\cdot Q+g\cdot\nabla \big (\frac{1}{2}|B|^2\big)\Big)\, \mathbb{I}-\varepsilon\big(\nabla v+\nabla v^T\big)-\big(\delta-\frac{2}{3}\varepsilon\big)\mbox{div} \, v \, \mathbb{I}\Big]n\\
&\quad=\Big[\Big(\widehat{ B}\cdot\widehat{Q}+g\cdot\nabla\big(\frac{1}{2}|\widehat{ B}|^2\big)\Big)\, \mathbb{I}\Big]n, \quad \mbox{on} \quad \Sigma_{0, pv},\\
&  n\cdot\widehat{Q}|_{\Sigma_w}=0,\, g|_{t=0}=g_0,
\end{cases}
\end{equation}
with $Q=\nabla\times (g\times B)$.

{\bf Notations: }
Define the energy
\begin{equation}\label{energy}
E[g, \widehat{Q}]=
E^p[g]+E^s[g]+E^v[\widehat{Q}],
\end{equation}
where $E^p[{g}]$, $E^s[g]$ and $E^v[\widehat{ Q}]$ are three real numbers, satisfying  
	\begin{equation}\label{fluid-e}
\begin{split}
E^p[g]&=\frac{1}{2}\int_{\overline{\Omega}}\Big[\gamma p |\nabla \cdot g|^2+|Q|^2+(g^*\cdot\nabla p)(\nabla\cdot g)+(\nabla \times B)\cdot (g^* \times Q)\Big]dx\\
&\quad+\frac{1}{4}\mu\int_{\overline{\Omega}} \Big[\varepsilon\Big|\nabla g+\nabla g^T-\frac{2}{3}\mbox{div} \, g \, \mathbb{I}\Big|^2+2\delta |\mbox{div}\, g|^2\Big]dx,
\end{split}
\end{equation}
\begin{equation}\label{surface-e}
E^s[g]=\frac{1}{2}\int_{\Sigma_{0, pv}} |n\cdot g|^2n\cdot [[\nabla (p+\frac{1}{2}|B|^2)]]dx,
\end{equation}
\begin{equation}\label{vacuum-e-viscosity}
E^v[\widehat{Q}]=\frac{1}{2}\int_{\overline{\Omega}^v}|\widehat{Q}|^2dx.
\end{equation}
In order to introduce the energy $E$, using Lemma \ref{vec-xi-grow-lem}, 
	formally we can get that
\begin{equation}
\begin{split}
\frac{d}{dt}\|\sqrt{\rho}g_{t}\|^2_{L^2}
=-\frac{d}{dt}E[g, \widehat{Q}].
\end{split}
\end{equation}
Our goal is to prove that there exists functions $g$ and $\widehat{ Q}$ such that $\inf_{\int\rho g^2dV=1} E[g, \widehat{Q}]<0$. 
	
The main purpose of this paper is to construct the growing mode solution of the form 
	\begin{equation}\label{grow-mode}
	\begin{split}
	&g(r,\theta, z, t)=(g_{r,mk}(r),g_{\theta,mk}(r),g_{z,mk}(r))e^{\mu t+i(m\theta+kz)}, \\&
	\widehat{Q}(r,\theta, z)=(i\widehat{ Q}_{r,mk}(r),\widehat{ Q}_{\theta,mk}(r),\widehat{Q}_{z,mk}(r))e^{\mu t+i(m\theta+kz)},
	\end{split}
	\end{equation} 
	where $$\mu >0,$$
	$(r,\theta,z)$ are the cylindrical coordinates,  $m,k\in \mathbb{Z}$, 
	and the subscripts $m$ and $k$ will be dropped for notational simplicity. Here, $g_\theta$ and $g_z$ are pure imaginary functions, $g_r$, $\widehat{Q}_r$, $\widehat{Q}_\theta$ and $\widehat{Q}_z$ are real-valued functions,
then we define new  three real-valued funtions
\begin{equation}\label{defini-xi-eta-zeta}
\xi=e_r\cdot g=g_r,\quad \eta=-ie_z\cdot g=-ig_z,
\quad \zeta=ie_\theta\cdot g=ig_\theta.
\end{equation}
which together with \eqref{grow-mode}, gives that
	\begin{equation}\label{varable-in-cyl}
	\begin{split}
	&Q=\nabla\times (g\times B)=\frac{im}{r}B_{\theta}\xi e_r-[(B_{\theta}\xi)'-k B_{\theta}\eta]e_{\theta}-\frac{m}{r}\eta B_{\theta}e_z,
	\\&g\cdot \nabla p+\gamma p\na\cdot g=p'\xi+\gamma p\na\cdot g,\quad \na\cdot g =\frac1r(r\xi)'-k\eta+\frac{m}{r}\zeta,
	\end{split}
	\end{equation}
	where the factor $e^{\mu t+i(m\theta +kz)}$ is dropped for notational simplicity.

Denote the Fourier decomposition $$E=\sum_{m,k \in Z}E_{m,k}.$$
In terms of $\xi$, $\eta$ and $\zeta$, from the expressions in \eqref{varable-in-cyl},  the boundary conditions in \eqref{linear-perturbation-and-boundary-viscosity} are transformed to
\begin{equation}\label{bound-q-wall}
\widehat{ Q}_r=0,\quad \mbox{at} \quad r=r_w, 
\end{equation}
\begin{equation}\label{boundary-cyl-1}
m\widehat{ B}_\theta\xi=r\widehat{ Q}_r, \quad \mbox{at} \quad r=r_0, 
\end{equation}
\begin{equation}\label{boundary-cyl-2}
\begin{split}
&\Big[B_{\theta}^2\xi-B_{\theta}^2\xi' r+kB_\theta^2\eta r  -\widehat{ B}_{\theta}\widehat{Q}_{\theta}r\Big]n\\
&\quad-\varepsilon\mu\Big(2 \xi'r,-i\zeta'r+i m\xi +i\zeta,i\eta'r+i\xi k r\Big)^T\\
&\quad-\mu(\delta-\frac{2}{3}\varepsilon)\Big[\xi'r+\xi+m\zeta -k\eta r\Big]n=0, \quad \mbox{at} \quad r=r_0.
\end{split}
\end{equation} 
We impose the boundary conditions \eqref{bound-q-wall} and \eqref{boundary-cyl-1} as constraint for variational problem setup and the boundary condition \eqref{boundary-cyl-2} follows the minimizer solution. When $m=0$, we know that $\widehat{  Q}_r=0$ on the boundary $r=r_0$, which implies that $\widehat{  Q}_r$ is separated from interior variational problem. Obviously, it holds that $\widehat{  Q}=0$.
Therefore, for the case $m=0$ and any $k\in \mathbb{Z}$, 
the energy functional \eqref{energy} reduces to
\begin{equation}\label{sausage-in-v}
\begin{split}
E_{0,k}(\xi,\eta,\zeta)&=2\pi^2\int_0^{r_0}\Big\{\frac{2p'\xi^2}{r}+B_{\theta}^2\Big[k\eta-\frac{1}{r}((r\xi)'-2\xi)\Big]^2\\
&\quad+\gamma p\Big[\frac{1}{r}(r\xi)'-k\eta\Big]^2\Big\}rdr+2\pi^2\int_0^{r_0}\tilde{\varepsilon}\Big[\frac{2}{9}\Big(-2\xi'+\frac{\xi}{r}-k\eta\Big)^2\\
&\quad+\frac{2}{9}\Big(\xi'-\frac{2\xi}{r}-k\eta\Big)^2+\frac{2}{9}\Big(\xi'+\frac{\xi}{r}+2k\eta\Big)^2+\Big(-\zeta'+\frac{\zeta}{r}\Big)^2\\
&\quad+(\eta'+k\xi)^2+k^2\zeta^2\Big]rdr+2\pi^2\int_0^{r_0}\tilde{\delta}\Big(\xi'+\frac{\xi}{r}-k\eta\Big)^2rdr.
\end{split}
\end{equation}
For the case $m\neq 0$ and any $k$, the solution $\widehat{ Q}_r$ and $\xi$ are related by the boundary conditions \eqref{boundary-cyl-1}, so we can not set $\widehat{  Q}_r=0$, therefore the energy funtional takes the form of  
	\begin{equation}\label{va-m-1-b}
	\begin{split}
	&E_{m,k}(\xi,\eta,\zeta,\widehat{ Q}_r)=2\pi^2\int_0^{r_0}\Big\{(m^2+k^2r^2)\Big[\frac{B_\theta}{r}\eta+\frac{-kB_\theta(r\xi)'+2kB_{\theta}\xi}{m^2+k^2r^2}\Big]^2\\
	&\qquad+\gamma p\Big[\frac{1}{r}(r\xi)'-k\eta+\frac{m\zeta}{r}\Big]^2\Big\}rdr+2\pi^2\int_0^{r_0}\tilde{\varepsilon}\Big[\frac{2}{9}\Big(-2\xi'+\frac{\xi}{r}+\frac{m}{r}\zeta-k\eta\Big)^2\\
	&\qquad+\frac{2}{9}\Big(\xi'-\frac{2\xi}{r}-\frac{2m}{r}\zeta-k\eta\Big)^2+\frac{2}{9}\Big(\xi'+\frac{\xi}{r}+\frac{m}{r}\zeta+2k\eta\Big)^2+\Big(-\zeta'+\frac{\zeta}{r}+\frac{m}{r}\xi\Big)^2\\
	&\qquad+(\eta'+k\xi)^2+\Big(\frac{m}{r}\eta-k\zeta\Big)^2\Big]rdr
	+2\pi^2\int_0^{r_0}\tilde{\delta}\Big(\xi'+\frac{\xi}{r}+\frac{m}{r}\zeta-k\eta\Big)^2rdr\\
	&\qquad+2\pi^2\int_0^{r_0}\frac{m^2B_\theta^2}{r(m^2+k^2r^2)}(\xi-r\xi')^2dr+\pi 	L\int_0^{r_0}\Big[2p'+\frac{m^2B_{\theta}^2}{r}\Big]\xi^2dr\\
	&\qquad+2\pi^2\int_{r_0}^{r_w}\Big[|\widehat{Q}_r|^2+\frac{1}{m^2+k^2r^2}|(r\widehat{Q}_r)'|^2\Big]r
	dr.
	\end{split}
	\end{equation}	
We define the weighted $L^2$ norm and the viscosity seminorm by
\begin{equation}\label{norm-1}
\|(\xi,\eta,\zeta)\|^2_1=\int_0^{r_0} \rho (|\xi|^2+|\eta|^2+|\zeta|^2)rdr,
\end{equation}
\begin{equation}\label{norm-2}
\begin{split}
\|(\xi,\eta,\zeta)\|^2_2&=\int_0^{r_0}\varepsilon\Big[\frac{2}{9}\Big(-2\xi'+\frac{\xi}{r}
+\frac{m}{r}\zeta-k\eta\Big)^2+\frac{2}{9}\Big(\xi'-\frac{2\xi}{r}-\frac{2m}{r}\zeta-k\eta\Big)^2\\
&\quad+\frac{2}{9}\Big(\xi'+\frac{\xi}{r}+\frac{m}{r}\zeta+2k\eta\Big)^2+\Big(-\zeta'+\frac{\zeta}{r}+\frac{m}{r}\xi\Big)^2+(\eta'+k\xi)^2\\
&\quad+\Big(\frac{m}{r}\eta-k\zeta\Big)^2\Big]rdr
+\int_0^{r_0}\delta\Big(\xi'+\frac{\xi}{r}+\frac{m}{r}\zeta-k\eta\Big)^2rdr.
\end{split}
\end{equation}	
When using the notation $g$, the norms $\|(\xi,\eta,\zeta)\|^2_1$ and $\|(\xi,\eta,\zeta)\|^2_2$ are equivalent to 
\begin{equation}\label{norm-1-xi}
\|g\|_1^2=\|\sqrt{\rho}g\|^2_{L^2}=\int_{\overline{\Omega}}\rho g^2dx,
\end{equation}
\begin{equation}\label{norm-2-xi}
\|g\|^2_2=\int_0^t\int_{\overline{\Omega}} \Big[\frac{\varepsilon}{2}\Big|\nabla g+\nabla g^T-\frac{2}{3}\mbox{div}\,  g \, \mathbb{I}\Big|^2+\delta |\mbox{div}\, g|^2\Big]dxds.
\end{equation}
We denote  $\partial_{tt}f=f_{tt}$. 
Assume the steady solution $(\overline{ u},B,p)$ in plasma satisfies  $\overline{u}=0$, $B=(0,B_{\theta}(r),0)$, $p=p(r)$ and the steady solution $\widehat{ B}$ in the vacuum region satisfies $\widehat{ B}=(0,\widehat{ B}_\theta(r),0) $ with $\widehat{B}_\theta(r)=B_{\theta}(r_0)\frac{r_0}{r}$, which are stated in \eqref{z-pinch-1-1}. 
Our main results are as follows.
	\begin{thm} Assume \eqref{minimizer-assumption} and Definition \ref{admissible}/the admissibility of the pressure $p$ holds. Then we have
\item[1)] For the modes $m=0$ and any $k\in\mathbb{Z}$, $\lambda_{0,k}=\inf_{(\xi,\eta)\in \mathcal{A}_1}E_{0,k}<0$, with the set $\mathcal{A}_1$ defined in \eqref{set a-b}, there is always growing mode of $z$-pinch instability.
\item[2)] For the modes $m\neq 0$ and any $k\in \mathbb{Z}$, 
if there exists $r^*$ such that $2p'(r^*)+\frac{m^2B_\theta^2(r^*)}{r^*}<0$,
then $\mu_{m,k}=\sqrt{-\lambda_{m,k}}$ is the growing mode to the linearized PDE \eqref{linear-perturbation-and-boundary-viscosity}, see also \eqref{spectral-formulation} and \eqref{euler-l-q}.
\item[3)] 
\begin{equation}\label{sup-Lambda}
	0<\sup_{B}\mu=\Lambda<\infty,\,\,\mbox{where}\,\, B=\{(m,k)\in \mathbb{Z}\times \mathbb{Z}|\lambda_{m,k}<0\}.
	\end{equation}
\item[4)] Moreover, assume 
	the initial data $\sqrt{\rho}(g_{t}(0), g_{tt}(0)) \in L^2(0,r_0)$,  $\widehat{ Q}_{t}(0)\in L^2(r_0,r_w)$ and $\|g_t(0)\|_2$ is bounded, which is defined in \eqref{norm-2-xi}, and 
	assume $g$ be a $H^2$ solution to \eqref{linear-perturbation-and-boundary-viscosity}. Then we have
	\begin{equation*}
	\begin{split}
	&\|g_t\|_1^2+\|g_t\|_2^2+\|\partial_{tt}g(t)\|_1^2\\
	&\quad\leq Ce^{2\Lambda t} \Big(\|g_t(0)\|_1^2+\|g_t(0)\|_2^2
	+\|\partial_{tt}g(0)\|_1^2+\|\widehat{ Q}_{t}(0)\|_{L^2}^2\Big).
	\end{split}
	\end{equation*}
\end{thm}
\begin{rmk}
The main goal of this paper is to study the growing modes to the linearized system \eqref{linear-perturbation-and-boundary-viscosity}, and we will leave the construction of $H^2$ solutions for general initial conditions for the future study.
\end{rmk}
\begin{rmk}
	In plasma literature, instability with $|m|=0$ is called a Sausage instability, and instability with $|m|=1$ is called a Kink instability.
\end{rmk}
We construct a largest growing mode which dominates the linear dynamics around a z-pinch steady state in the presence of viscosity effect.
The main difficulties and innovations are as follows. Beside the degeneracy both at the origin and the boundary,  and the singularity of $\frac{1}{r}$ at the origin,  we have several new difficulties in this paper.

 In the presence of viscosity, there is no natural variational framework for constructing growing mode solutions to the linearized problem.
We use the modified variational method by studying a family of modified variational problems to produce the growing modes. 
More precisely, we artificially remove the linear dependence on $\mu$, and we define two modified viscosities $\tilde{\varepsilon}=s\varepsilon$ and $\tilde{\delta}=s\delta$, instead of $\mu\varepsilon$ and $\mu\delta$, where $s>0$ is an arbitrary parameter, then we can introduce a family of modified problem \eqref{spectral-formulation}
along with boundary conditions 
\begin{equation*}
\widehat{ Q}_r=0, \quad \mbox{at} \quad r=r_w, 
\end{equation*}
\begin{equation*}
m\widehat{  B}_\theta \xi=r\widehat{ Q}_r, \quad \mbox{at} \quad r=r_0, 
\end{equation*}
\begin{equation*}
\begin{split}
&\Big[B_{\theta}^2\xi-B_{\theta}^2\xi' r+kB_\theta^2\eta r-\widehat{ B}_{\theta}\widehat{Q}_{\theta}r\Big]n\\
&\quad-\tilde{\varepsilon}\Big(2\xi'r,-i\zeta'r+im\xi +i\zeta, i\eta'r+ik\xi r\Big)^T\\
&\quad-(\tilde{\delta}-\frac{2}{3}\tilde{\varepsilon})\Big[\xi'r+\xi+m\zeta -k\eta r\Big]n=0, \quad \mbox{at} \quad r=r_0.
\end{split}
\end{equation*}
A solution to the modified problem with $\mu = s$ corresponds to a solution to the original problem \eqref{spectral-formulation-orig}.
Modifying the problem in this way restores the variational structure and allows us to apply a constrained minimization to the viscous analogue of the energy $E_{0,k}$ defined in \eqref{sausage-in-v} and $E_{m,k}$ defined in \eqref{va-m-1-b} to find a solution to \eqref{spectral-formulation} with $\mu=\mu(s)>0$ when $s > 0$ is sufficiently small. 
When $m=0$ and any $k\in\mathbb{Z}$, we define
\begin{equation*}
\lambda(s)=\inf_{(\xi,\eta,\zeta)\in X_k}\frac{E_{0,k}(\xi,\eta,\zeta;s)}{\mathcal{J}(\xi,\eta,\zeta)},
\end{equation*}
and when $m\neq 0$ and any $k\in\mathbb{Z}$, we define
\begin{equation*}
\lambda(s)=\inf_{((\xi,\eta,\zeta),\widehat{ Q}_r)\in \mathcal{A}_2}E_{m,k}(\xi,\eta,\zeta,\widehat{ Q}_r;s),
\end{equation*}
with
\begin{equation*}
\mathcal{A}_2=\left\{
\begin{aligned}
\big((\xi,\eta,\zeta),\widehat{ Q}_r\big
)\in Y_{m,k}\times H^1(r_0,r_w)| \mathcal{J}(\xi,\eta,\zeta)=1,\,\\ m\widehat{  B}_\theta \xi=r\widehat{ Q}_r\, \,\mbox{at}\, \,r=r_0\,\, \mbox{and}\,\, \widehat{ Q}_r=0\, \,\mbox{at}\, \,r=r_w
\end{aligned}
\right\}.
\end{equation*}
We then further exploit the variational structure to show that $\lambda$  is a  continuous function and is strictly increasing in $s$. Using this, we show in Proposition \ref{prop3.7} when $m=0$ and Proposition \ref{prop3.14} when $m\neq 0$ that the parameter $s$ can be uniquely chosen so that 
$$s=\mu(s)=\sqrt{-\lambda},\,\, \mbox{for any} \,\, 
(m,k)\in \mathbb{Z}\times \mathbb{Z},$$
which implies that we have found a solution to the original problem \eqref{spectral-formulation-orig}.

In contrast to the inviscid case, the ill-posedness is now excluded. Thanks to the viscous effect, we are able to show that there exists biggest growing mode for any $m$ and $k$, and prove $\lim_{m,k\rightarrow \infty} \mu_{m,k}=0$ by a contradiction argument, see Proposition  \ref{growing-mode-b2}. More precisely, suppose  that 
	\begin{equation*}
	\limsup_{m \rightarrow \infty\, \mbox{or}\,k\rightarrow\infty}s_{m,k}>0.
	\end{equation*}
	Then we have
	\begin{equation*}
	\begin{split}
	s_{m,k}\leq \frac{C_0+C}{D_{m,k}},\,\,	\mbox{for any}\,\, (m, k)\in \mathbb{Z}\times \mathbb{Z}.
	\end{split}
	\end{equation*}
	From \eqref{norm-2-xi} and new coordinates \eqref{grow-mode}, the dissipation rate is defined naturally as 
		\begin{equation}\label{dissipation-rate}
	\begin{split}
	D_{m,k}&=2\pi^2\int_0^{r_0}\varepsilon\Big[\frac{2}{9}\Big(-2\xi'+\frac{\xi}{r}+\frac{m}{r}\zeta-k\eta\Big)^2+\frac{2}{9}\Big(\xi'-\frac{2\xi}{r}-\frac{2m}{r}\zeta-k\eta\Big)^2
	\\
	&\quad+\frac{2}{9}\Big(\xi'+\frac{\xi}{r}+\frac{m}{r}\zeta+2k\eta\Big)^2+\Big(-\zeta'+\frac{\zeta}{r}+\frac{m}{r}\xi\Big)^2+(\eta'+k\xi)^2\\
	&\quad+\Big(\frac{m}{r}\eta-k\zeta\Big)^2\Big]rdr
	+2\pi^2\int_0^{r_0}\delta\Big(\xi'+\frac{\xi}{r}+\frac{m}{r}\zeta-k\eta\Big)^2rdr.
	\end{split}
	\end{equation}
The key is to establish 
$$D_{m,k}\gtrsim k^2+m^2,$$
which involves delicate analysis for weighted estimates 
for $\int_0^{r_0}\frac{\xi^2}{r}dr$, $\int_0^{r_0}\frac{\eta^2}{r}dr$ and 	$\int_0^{r_0}\frac{\zeta^2}{r}dr$. Precise algebraic manipulations have to be carried out to analyze the lower bound of the viscous dissipation, for each case of $m=0$, $|m|=1$ and $|m|\geq 2$ separately in Section 4.

\section{A family of modified variational problems }
\subsection{Growing mode ansatz and Cylindrical coordinates}
In order to see the property of the force operator  
	\begin{equation}\label{force-operator}
	\begin{split}
	F(g)&=\nabla(g\cdot\nabla p+\gamma p\nabla\cdot g)+(\nabla \times B)\times [\nabla \times (g \times B)]+\{\nabla \times [\nabla \times (g \times B)]\}\times B\\
	&\quad+ \mbox{div} \Big (\varepsilon\Big(\nabla g_t+\nabla g_t^T-\frac{2}{3}\mbox{div}  g_t \, \mathbb{I}\Big)+\delta \mbox{div} g_t\, \mathbb{I}\Big),
	\end{split}
	\end{equation} 
we consider two displacement vector fields $g$ and $h$ defined over the plasma volume $V$, their associated magnetic field perturbations
\begin{equation*}
Q=\nabla\times (g\times B), \quad R=\nabla\times (h\times B),
\end{equation*}
and the vacuum perturbations $\widehat{ Q}$ and  $\widehat{ R}$ defined over the vacuum volume $\widehat{ V}$ are their extensions,
that is, to `extend' the function $g$ into the vacuum by means of the magnetic field variable $\widehat{ Q}$, and likewise to `extend' $h$ by means of $\widehat{ R}$.
Then by Chapter 6 in \cite{Goedbloed-poedts} and Lemma 5.3 in \cite{bian-guo-tice-inviscid}, we have the following lemma.
\begin{lem}\label{joint}
Assume $g$ is a $H^2$ solution of the system \eqref{linear-perturbation-and-boundary-viscosity}, then we get a meaning expression for the potential energy of interface plasma by identifying $g$, $h$,  $\widehat{ Q}$ and $\widehat{ R}$, in the  quadratic form 
	\begin{equation}\label{sym}
	\begin{split}
	\int_{\overline{\Omega} }h \cdot F(g)\,dx&=-\int_{\overline{\Omega}}\Big[\gamma p \nabla \cdot g \nabla \cdot h+Q\cdot R+\frac{1}{2}\nabla p\cdot(g\nabla\cdot h+h\nabla\cdot g)\\
	&\quad+\frac{1}{2}\nabla \times B\cdot (g \times R+h\times Q)\Big]dx -\int_{\overline{\Omega}^v}\widehat{Q}\cdot \widehat{ R}\,dx\\
	&\quad-\int_{\overline{\Omega}} \bigg (\varepsilon\bigg(\nabla g_t+\nabla g_t^T-\frac{2}{3}\mbox{div}\,  g_t \, \mathbb{I}\bigg)+\delta \mbox{div}\, g_t\, \mathbb{I}\bigg):\nabla h \,dx\\
	&\quad-\int_{\Sigma_{0, pv} }n\cdot g \, n \cdot h\, n\cdot \Big[\Big[\nabla (p+\frac{1}{2}|B|^2)\Big]\Big]\,dx,
	\end{split}
	\end{equation}
which is symmetric in $g$ and $h$, and their extensions $\widehat{ Q}$ and $\widehat{ R}$, if $\varepsilon=\delta=0$. 
\end{lem}
\begin{proof}
	Note that
	\begin{equation*}
	\begin{split}
	&\mbox{div} \bigg (\varepsilon\bigg(\nabla g_t+\nabla g_t^T-\frac{2}{3}\mbox{div} \, g_t \, \mathbb{I}\bigg)+\delta \mbox{div}\, g_t\, \mathbb{I}\bigg)\cdot h\\
	&= \mbox{div} \bigg[\bigg (\varepsilon\bigg(\nabla g_t+\nabla g_t^T-\frac{2}{3}\mbox{div} \, g_t \, \mathbb{I}\bigg)+\delta \mbox{div}\, g_t\, \mathbb{I}\bigg)h\bigg]\\
	&\quad-\bigg (\varepsilon\bigg(\nabla g_t+\nabla g_t^T-\frac{2}{3}\mbox{div}\,  g_t \, \mathbb{I}\bigg)+\delta \mbox{div}\, g_t\, \mathbb{I}\bigg): \nabla h.
	\end{split}
	\end{equation*}
Using the estimates in Lemma 5.3 in Bian-Guo-Tice \cite{bian-guo-tice-inviscid},
we get  that
	\begin{equation}\label{sym-form}
	\begin{split}
	h\cdot F(g)&=-\gamma p\nabla\cdot g\nabla\cdot h-Q\cdot R
	-\frac{1}{2}\nabla p\cdot (g\nabla\cdot h+h\nabla\cdot g)\\
	&\quad-\frac{1}{2}\nabla\times B\cdot (g\times R+h\times Q)\\
	&\quad+\nabla\cdot [ h(g\cdot \nabla p)]-\nabla\cdot (hB\cdot Q)+\nabla \cdot(h\gamma p\nabla\cdot g)\\
	&\quad+\frac{1}{2}\nabla\cdot [(\nabla\times B) B\cdot (g\times h)]+\nabla\cdot (Bh\cdot Q)\\
	&\quad-\frac{1}{2}\nabla\cdot [(\nabla\times B\times B)\cdot(gh-hg)]\\
	&\quad+ \mbox{div} \bigg[\bigg (\varepsilon\bigg(\nabla g_t+\nabla g_t^T-\frac{2}{3}\mbox{div}\,  g_t \, \mathbb{I}\bigg)+\delta \mbox{div} \,g_t\, \mathbb{I}\bigg)h\bigg]\\
	&\quad-\bigg (\varepsilon\bigg(\nabla g_t+\nabla g_t^T-\frac{2}{3}\mbox{div} \, g_t \, \mathbb{I}\bigg)+\delta \mbox{div} \,g_t\, \mathbb{I}\bigg): \nabla h.
	\end{split}
	\end{equation}
	Integrating \eqref{sym-form} gives that
	\begin{equation}\label{sym-form-int}
	\begin{split}
	\int_{\overline{\Omega}} h\cdot F(g) \,dx&=-\int_{\overline{\Omega}}\Big[\gamma p\nabla\cdot g\nabla\cdot h+Q\cdot R+\frac{1}{2}\nabla p\cdot (g\nabla\cdot h+h\nabla\cdot g)\\
	&\quad+\frac{1}{2}\nabla\times B\cdot (g\times R+h\times Q)\Big]\,dx\\
	&\quad+\int_{\Sigma_{0, pv}} n\cdot h\Big(g\cdot \nabla p-B\cdot Q+\gamma p\nabla\cdot g\Big)\,dx
	\\
	&\quad+\int_{\Sigma_{0, pv}} n\cdot h \, Tr\bigg[\bigg (\varepsilon\bigg(\nabla g_t+\nabla g_t^T-\frac{2}{3}\mbox{div}\,  g_t \, \mathbb{I}\bigg)+\delta \mbox{div}\, g_t\, \mathbb{I}\bigg)\bigg]\,dx\\
	&\quad-\int_{\overline{\Omega}} \bigg (\varepsilon\bigg(\nabla g_t+\nabla g_t^T-\frac{2}{3}\mbox{div}\,  g_t \, \mathbb{I}\bigg)+\delta \mbox{div}\, g_t\, \mathbb{I}\bigg): \nabla h\, dx.
	\end{split}
	\end{equation}
	There are no contributions from the eighth, ninth and tenth terms of \eqref{sym-form}  to the surface integral, since $n\cdot B=0$ and $n\cdot \nabla\times B=0$ on the plasma surface, whereas $\nabla \times B\times B$ is parallel to $n$.
From the second interface condition of \eqref{linear-perturbation-and-boundary-viscosity}, the surface integral takes the form of 
	\begin{equation}\label{tranf-sur}
	\begin{split}
	&\int_{\Sigma_{0, pv}} n\cdot h\Big(g\cdot \nabla p-B\cdot Q+\gamma p\nabla\cdot g\Big)\,dx\\
	&\qquad+\int_{\Sigma_{0, pv}} n\cdot h \, Tr\bigg[\bigg (\varepsilon\bigg(\nabla g_t+\nabla g_t^T-\frac{2}{3}\mbox{div}\,  g_t \, \mathbb{I}\bigg)+\delta \mbox{div}\, g_t\, \mathbb{I}\bigg)\bigg]\,dx\\
	&\quad=-\int_{\Sigma_{0, pv}} n\cdot hg\cdot \Big[\Big[\nabla (p+\frac{1}{2}|B|^2)\Big]\Big]dx-\int_{\Sigma_{0, pv}} n\cdot h\widehat{   B}\cdot \widehat{ Q}\,dx\\
	&\quad =-\int_{\Sigma_{0, pv}} n\cdot g \, n \cdot h\, n\cdot \Big[\Big[\nabla (p+\frac{1}{2}|B|^2)\Big]\Big]\,dx-\int_{\Sigma_{0, pv}} n\cdot h\widehat{   B}\cdot \widehat{ Q}\,dx\\
	&\quad =-\int_{\Sigma_{0, pv}} n\cdot g \, n \cdot h\, n\cdot \Big[\Big[\nabla (p+\frac{1}{2}|B|^2)\Big]\Big]\,dx-\int_{\overline{\Omega}^v}\widehat{ Q}\cdot \widehat{ R}\,dx.
	\end{split}
	\end{equation}
	Here, we have used the facts
$-\int_{\Sigma_{0, pv}} n\cdot h\widehat{  B}\cdot \widehat{ Q}\,dx=-\int_{\overline{\Omega}^v} \widehat{ Q}\cdot \widehat{ R}\, dx$ (Lemma 5.3 in Bian-Guo-Tice \cite{bian-guo-tice-inviscid}), the equilibrium jump condition $\Big[\Big[ p+\frac{1}{2}|B|^2\Big]\Big]=0$, which implies that the tangential derivative of the jump vanishes as well
	$\vec{t}\cdot \Big[\Big[\nabla (p+\frac{1}{2}|B|^2)\Big]\Big]=0,$ where $\vec{t}$ is an arbitrary unit vector tangential to the surface.
This finishes the proof.
\end{proof}

	We now record several computations in cylindrical  coordinates.

\begin{lem}\label{lem-e}
Let $\tilde{\varepsilon}=\mu\varepsilon$ and $\tilde{\delta}=\mu\delta$. We decompose $E_{m,k}$ as follows
		\begin{equation}\label{e_mk}
	E_{m,k}(\xi,\eta,\zeta,\widehat{ Q}_r)=E_{m,k}^p+E^s_{m,k}+E_{m,k}^v,
		\end{equation}
where the fluid energy takes the form of
		\begin{equation}\label{fluid-en-cyl-viscosity}
		\begin{split}
		&E_{m,k}^p=2\pi^2\int_0^{r_0}\Big\{(m^2+k^2r^2)\Big[\frac{B_\theta}{r}\eta+\frac{-kB_\theta(r\xi)'+2kB_{\theta}\xi}{m^2+k^2r^2}\Big]^2\\
		&\qquad+\gamma p\Big[\frac{1}{r}(r\xi)'-k\eta+\frac{m\zeta}{r}\Big]^2\Big\}rdr+2\pi^2\int_0^{r_0}\tilde{\varepsilon}\Big[\frac{2}{9}\Big(-2\xi'+\frac{\xi}{r}+\frac{m}{r}\zeta-k\eta\Big)^2\\
		&\qquad+\frac{2}{9}\Big(\xi'-\frac{2\xi}{r}-\frac{2m}{r}\zeta-k\eta\Big)^2+\frac{2}{9}\Big(\xi'+\frac{\xi}{r}+\frac{m}{r}\zeta+2k\eta\Big)^2+\Big(-\zeta'+\frac{\zeta}{r}+\frac{m}{r}\xi\Big)^2\\
		&\qquad+(\eta'+k\xi)^2+\Big(\frac{m}{r}\eta-k\zeta\Big)^2\Big]rdr
		+2\pi^2\int_0^{r_0}\tilde{\delta}\Big(\xi'+\frac{\xi}{r}+\frac{m}{r}\zeta-k\eta\Big)^2rdr\\
		&\qquad+2\pi^2\int_0^{r_0}\frac{m^2B_\theta^2}{r(m^2+k^2r^2)}(\xi-r\xi')^2dr+2\pi^2 	\int_0^{r_0}\Big[2p'+\frac{m^2B_{\theta}^2}{r}\Big]\xi^2dr,
		\end{split}
		\end{equation}
			the surface energy vanishes
		\begin{equation}\label{surface-en-cyl-viscosity}
		E^s_{m,k}=-2\pi^2[\widehat{B}_{\theta}^2-B_{\theta}^2]_{r=r_0}\xi^2(r_0)=0,
		\end{equation}
	and when $m\neq 0$ and any $k$,   the vacuum energy takes  the form of 
			\begin{equation}\label{vacuum-e-cyl-viscosity}
			\begin{split}
			E^v_{m,k}&=2\pi^2\int_{r_0}^{r_w}\Big[|\widehat{Q}_r|^2+\frac{1}{m^2+k^2r^2}|(r\widehat{Q}_r)'|^2\Big]rdr.
			\end{split}
			\end{equation}		
\end{lem}
\begin{proof}
Recall $(\xi,\eta,
	\zeta)$	 in \eqref{defini-xi-eta-zeta} and \eqref{grow-mode}.
The proofs of \eqref{surface-en-cyl-viscosity} and \eqref{vacuum-e-cyl-viscosity} can be recalled from the Lemma 3.1 in \cite{bian-guo-tice-inviscid}, so we only need to prove  \eqref{fluid-en-cyl-viscosity}.
		Inserting the expressions of \eqref{varable-in-cyl} into \eqref{fluid-e} and using $\xi_\theta^*=i\zeta$, we can get  \eqref{fluid-en-cyl-viscosity}. 
		In fact,
		\begin{equation*}
		(g^*\cdot\nabla  p)(\nabla \cdot g )=\xi  p'\Big(\frac1r(rg_r)'-k\eta+\frac{m}{r}\zeta\Big), 
		\end{equation*}
		\begin{equation*}
		\begin{split}
		(\nabla \times B)\cdot (g^* \times Q)&=\Big(B_\theta'+\frac{B_\theta}{r}\Big)\Big[-g_r\Big((B_{\theta}g_r)'-k B_{\theta}\eta\Big)-\frac{im}{r}B_{\theta}g_rg_\theta^*\Big]\\
		&=\Big(B_\theta'+\frac{B_\theta}{r}\Big)\Big[-g_r\Big((B_{\theta}g_r)'-k B_{\theta}\eta\Big)+\frac{m}{r}\zeta B_{\theta}g_r\Big],
		\end{split}
		\end{equation*}
		\begin{equation*}\begin{split}
		\nabla g=	\left(\begin{array}{ccc}
		\partial_rg_r&\partial_{r}g_{\theta}&\partial_rg_z\\
		\frac{img_r}{r}- \frac{g_\theta}{r}&\frac{img_\theta}{r}+\frac{g_r}{r}&\frac{im g_z}{r}\\
		ikg_r&ikg_\theta&ikg_z
		\end{array}
		\right)=\left(\begin{array}{ccc}
		\xi'&-i\zeta'&i\eta'\\
		\frac{im}{r}\xi+\frac{i\zeta}{r}&\frac{m}{r}\zeta+\frac{\xi}{r}&-\frac{m}{r}\eta\\
		ik\xi&k\zeta&-k\eta
		\end{array}
		\right), 
		\end{split}
		\end{equation*}
		\begin{equation*}\begin{split}
		\nabla g+\nabla g^T=\left(\begin{array}{ccc}
		2\xi'&-i\zeta'+\frac{im}{r}\xi+\frac{i\zeta}{r}&i\eta'+ik\xi\\
		-i\zeta'+\frac{im}{r}\xi+\frac{i\zeta}{r}&\frac{2m}{r}\zeta+\frac{2\xi}{r}&-\frac{m}{r}\eta+k\zeta\\
		i\eta'+ik\xi&-\frac{m}{r}\eta+k\zeta&-2k\eta
		\end{array}
		\right), 
		\end{split}
		\end{equation*}
		\begin{equation*}
		\begin{split}
		\mbox{div}  g \, \mathbb{I}=\left(\begin{array}{ccc}
		\xi'+\frac{\xi}{r}-k\eta+\frac{m}{r}\zeta&0&0\\
		0&\xi'+\frac{\xi}{r}-k\eta+\frac{m}{r}\zeta&0\\
		0&0&	\xi'+\frac{\xi}{r}-k\eta+\frac{m}{r}\zeta
		\end{array}
		\right).
		\end{split}
		\end{equation*}
		By $p'=-B_\theta B'_\theta-\frac{B^2_\theta}{r}$, we have
		\begin{equation*}
		\begin{split}
		&\int_0^{r_0}\Big(\frac{m^2B_\theta^2}{r^2(m^2+k^2r^2)}[(r\xi)']^2+\beta_0(r\xi)^2\Big)rdr-\Big[\frac{2m^2B_{\theta}^2}{m^2+k^2r^2}\xi^2\Big]_{r=r_0}\\
		&=\int_0^{r_0}\bigg[\frac{m^2B^2_\theta}{r(m^2+k^2r^2)}(\xi-r\xi')^2+2p'\xi^2+\frac{m^2B^2_\theta\xi^2}{r}\bigg]dr,
		\end{split}
		\end{equation*}
		with
	$$\beta_0=\frac{1}{r}\bigg[\frac{m^2B_\theta^2}{r^3}+\frac{2m^2B_{\theta}(\frac{B_\theta}{r})'}{r(m^2+k^2r^2)}-\frac{4k^2m^2B^2_{\theta}}{r(m^2+k^2r^2)^2}+\frac{2k^2p'}{m^2+k^2r^2}\bigg].$$
	 Combining the above identities with \eqref{fluid-e} and  \eqref{varable-in-cyl} yields \eqref{fluid-en-cyl-viscosity}.	 	
\end{proof}
Using  $g(r,\theta, z, t)=(g_{r}(r,t),g_{\theta}(r,t),g_{z}(r,t))e^{i(m\theta+kz)}$, we can prove that the second equation in \eqref{linear-perturbation-and-boundary-viscosity} is reduced to the following system.
\begin{lem}\label{spectral-lem-without-factor}
Assume $g(r,\theta, z, t)=(g_{r}(r,t),g_{\theta}(r,t),g_{z}(r,t))e^{i(m\theta+kz)}$ solves  the second equation in \eqref{linear-perturbation-and-boundary-viscosity}, then
	\begin{equation}\label{spectral-formulation-orig}
	\begin{split}
	&\left(
	\begin{array}{ccc}
	\frac{d}{dr}\frac{\gamma p+B_{\theta}^2}{r}\frac{d}{dr}r-\frac{m^2}{r^2}B_{\theta}^2
	-r(\frac{B_{\theta}^2}{r^2})'&-\frac{d}{dr}k(\gamma p+B_{\theta}^2)-\frac{2kB_{\theta}^2}{r}&\frac{d}{dr}\frac{m}{r}\gamma p\\
	\frac{k(\gamma p+B_{\theta}^2)}{r}\frac{d}{dr}r-\frac{2kB_{\theta}^2}{r}&
	-k^2(\gamma p+B_{\theta}^2)-\frac{m^2}{r^2}B_{\theta}^2&\frac{mk}{r}\gamma p\\
	-\frac{m\gamma p}{r^2}\frac{d}{dr}r&\frac{mk}{r}\gamma p&-\frac{m^2}{r^2}\gamma p
	\end{array}
	\right)
	\left(
	\begin{array}{lll}
	\xi   \\
	\eta \\
	\zeta\\
	\end{array}
	\right)\\
	&\quad	+\left(
	\begin{array}{ccc}
	d_{11}&d_{12}&d_{13}\\
	d_{21}&
	d_{22}&d_{23}\\
	d_{31}&d_{32}&d_{33}
	\end{array}
	\right)
	\left(
	\begin{array}{lll}
	\xi_t   \\
	\eta_t \\
	\zeta_t\\
	\end{array}
	\right)	
	=\rho \left(
	\begin{array}{lll}
	\xi_{tt}   \\
	\eta_{tt} \\
	\zeta_{tt}\\
	\end{array}
	\right),
	\end{split}
	\end{equation}
with 
\begin{equation*}
\begin{split}
&d_{11}=	(\frac{4\varepsilon}{3}+\delta)\frac{d^2}{dr^2}+(\frac{4\varepsilon}{3}+\delta)\frac{d}{dr}\frac{1}{r}-\varepsilon\Big(\frac{m^2}{r^2}+k^2\Big),\quad d_{12}=-(\frac{\varepsilon}{3}+\delta)k\frac{d}{dr},\\
&d_{13}=(\frac{\varepsilon}{3}+\delta)\frac{d}{dr}\frac{m}{r}
-\frac{2\varepsilon m}{r^2},\quad d_{21}=(\frac{\varepsilon}{3}+\delta)k\frac{d}{dr}+(\frac{\varepsilon}{3}+\delta)\frac{k}{r},\\
&d_{22}=	\varepsilon\frac{d^2}{dr^2}+\frac{\varepsilon}{r}\frac{d}{dr}
-\varepsilon \frac{m^2}{r^2}-(\frac{4\varepsilon}{3}+\delta)k^2,\quad
d_{23}=(\frac{\varepsilon}{3}+\delta)\frac{mk}{r},\\
&d_{31}=-(\frac{\varepsilon}{3}+\delta)\frac{m}{r}\frac{d}{dr}-(\frac{7\varepsilon}{3}+\delta)\frac{m}{r^2},\quad d_{32}=(\frac{\varepsilon}{3}+\delta)\frac{mk}{r},\\
&d_{33}=\varepsilon\frac{d^2}{dr^2}+\varepsilon\frac{d}{dr}\frac{1}{r}-(\frac{4\varepsilon}{3}+\delta)\frac{m^2}{r^2}-\varepsilon k^2.
\end{split}
\end{equation*}
\end{lem}
\begin{proof}
	Inserting the expression \eqref{varable-in-cyl} into the second equation in \eqref{linear-perturbation-and-boundary-viscosity}, by  $g(r,\theta, z, t)=(g_{r}(r,t),g_{\theta}(r,t),g_{z}(r,t))e^{i(m\theta+kz)}$ and the definitions of $\xi$, $\eta$ and $\zeta$ in \eqref{defini-xi-eta-zeta}, we can easily get that the second equation in \eqref{linear-perturbation-and-boundary-viscosity} reduces to \eqref{spectral-formulation-orig}.
\end{proof}	
We now use the idea of modifying the viscosity parameters in Guo-Tice \cite{Guo-TIce-viscous}. Taking the normal mode \eqref{grow-mode} in Lemma \ref{lem-e}, in order to restore the ability to use varational methods, we artificially remove the linear dependence on $\mu$. We define two modified  viscosities $\tilde{\varepsilon}=s\varepsilon$ and $\tilde{\delta}=s\delta$, instead of $\mu\varepsilon$ and $\mu\delta$, where $s>0$ is an arbitrary parameter, we can get modified problem as follows.
\begin{lem}\label{spectral-lem-with-factor}
Assume  $g(r,\theta, z, t)=(g_{r}(r,t),g_{\theta}(r,t),g_{z}(r,t))e^{\mu t+i(m\theta+kz)}$
solves the second equation in \eqref{linear-perturbation-and-boundary-viscosity}, then 
	\begin{equation}\label{spectral-formulation}
		\begin{split}
	&\left(
	\begin{array}{ccc}
	\frac{d}{dr}\frac{\gamma p+B_{\theta}^2}{r}\frac{d}{dr}r-\frac{m^2}{r^2}B_{\theta}^2
	-r(\frac{B_{\theta}^2}{r^2})'&-\frac{d}{dr}k(\gamma p+B_{\theta}^2)-\frac{2kB_{\theta}^2}{r}&\frac{d}{dr}\frac{m}{r}\gamma p\\
	\frac{k(\gamma p+B_{\theta}^2)}{r}\frac{d}{dr}r-\frac{2kB_{\theta}^2}{r}&
	-k^2(\gamma p+B_{\theta}^2)-\frac{m^2}{r^2}B_{\theta}^2&\frac{mk}{r}\gamma p\\
	-\frac{m\gamma p}{r^2}\frac{d}{dr}r&\frac{mk}{r}\gamma p&-\frac{m^2}{r^2}\gamma p
	\end{array}
	\right)
	\left(
	\begin{array}{lll}
	\xi   \\
	\eta \\
	\zeta\\
	\end{array}
	\right)\\
	&\quad	+\left(
	\begin{array}{ccc}
a_{11}&a_{12}&a_{13}\\
	a_{21}&
a_{22}&a_{23}\\
	a_{31}&a_{32}&a_{33}
	\end{array}
	\right)
	\left(
	\begin{array}{lll}
	\xi   \\
	\eta \\
	\zeta\\
	\end{array}
	\right)	
	=\rho \mu^2 \left(
	\begin{array}{lll}
	\xi   \\
	\eta \\
	\zeta\\
	\end{array}
	\right),
	\end{split}
	\end{equation}
with 
\begin{equation*}
\begin{split}
&a_{11}=	(\frac{4\tilde{\varepsilon}}{3}+\tilde{\delta})\frac{d^2}{dr^2}+(\frac{4\tilde{\varepsilon}}{3}+\tilde{\delta})\frac{d}{dr}\frac{1}{r}-\tilde{\varepsilon}\Big(\frac{m^2}{r^2}+k^2\Big),\quad a_{12}=-(\frac{\tilde{\varepsilon}}{3}+\tilde{\delta})k\frac{d}{dr},\\
&a_{13}=(\frac{\tilde{\varepsilon}}{3}+\tilde{\delta})\frac{d}{dr}\frac{m}{r}
-\frac{2\tilde{\varepsilon}m}{r^2},\quad a_{21}=(\frac{\tilde{\varepsilon}}{3}+\tilde{\delta})k\frac{d}{dr}+(\frac{\tilde{\varepsilon}}{3}+\tilde{\delta})\frac{k}{r},\\
&a_{22}=	\tilde{\varepsilon}\frac{d^2}{dr^2}+\frac{\tilde{\varepsilon}}{r}\frac{d}{dr}
-\tilde{\varepsilon}\frac{m^2}{r^2}-(\frac{4\tilde{\varepsilon}}{3}+\tilde{\delta})k^2,\quad
a_{23}=(\frac{\tilde{\varepsilon}}{3}+\tilde{\delta})\frac{mk}{r},\\
&a_{31}=-(\frac{\tilde{\varepsilon}}{3}+\tilde{\delta})\frac{m}{r}\frac{d}{dr}-(\frac{7\tilde{\varepsilon}}{3}+\tilde{\delta})\frac{m}{r^2},\quad a_{32}=(\frac{\tilde{\varepsilon}}{3}+\tilde{\delta})\frac{mk}{r},\\
&a_{33}=\tilde{\varepsilon}\frac{d^2}{dr^2}+\tilde{\varepsilon}\frac{d}{dr}\frac{1}{r}-(\frac{4\tilde{\varepsilon}}{3}+\tilde{\delta})\frac{m^2}{r^2}-\tilde{\varepsilon}k^2.
\end{split}
\end{equation*}
\end{lem}
\begin{proof}	
	Inserting the expression \eqref{varable-in-cyl} into the second equation in \eqref{linear-perturbation-and-boundary-viscosity}, by  $g(r,\theta, z, t)=(g_{r}(r,t),g_{\theta}(r,t),g_{z}(r,t))e^{\mu t+i(m\theta+kz)}$ and the definitions of $\xi$, $\eta$ and $\zeta$ in \eqref{defini-xi-eta-zeta}, we can easily get that the second equation in \eqref{linear-perturbation-and-boundary-viscosity} reduces to \eqref{spectral-formulation}.
\end{proof}	
In order to study the stability to use variational methods in vacuum domain, we use the following second-order ODE about $\widehat{ Q}_r$ for $m\neq 0$ and any $k$.
\begin{lem}[Lemma 3.6 in \cite{bian-guo-tice-inviscid}]
The  vacuum equations \eqref{linear-perturbation-and-boundary-viscosity}$_3$ and \eqref{linear-perturbation-and-boundary-viscosity}$_4$  can be reduced to the second order differential equation
\begin{equation}\label{euler-l-q}
\bigg[\frac{r}{m^2+k^2r^2}(r\widehat{Q}_r)'\bigg]'-\widehat{Q}_r=0,
\end{equation}
with the other two components $\widehat{ Q}_{\theta}=-\frac{m}{m^2+k^2r^2}(r\widehat{ Q}_r)'$ and
$\widehat{ Q}_{z}=-\frac{kr}{m^2+k^2r^2}(r\widehat{ Q}_r)'$. 
\end{lem}

\subsection{Modified variational problem when $m=0$}
In this subsection, we will introduce the definition of the function space $X_k$ and its properties, then  give the modified variational analysis for the case $m=0$ and any $k\in\mathbb{Z}$.
We first introduce the Definition of the function space $X_k$ for any $k\in\mathbb{Z}$ and its properties.
\begin{defi}\label{defi-X}
	The weighted Sobolev space $X_k$ is defined as the completion of $\Big\{(\xi,\eta,\zeta) \in C^{\infty}([0,r_0])\times C^{\infty}([0,r_0])\times C^{\infty}([0,r_0]) \Big|\xi(0)=\zeta(0)=0\Big\}$, with respect to the norm
\begin{equation}\label{defi-Xk}
\begin{split}
\|(\xi,\eta,\zeta)\|^2_{X_k}&=\int_0^{r_0}\Big[\Big(-2\xi'+\frac{\xi}{r}-k\eta\Big)^2+\Big(\xi'-\frac{2\xi}{r}-k\eta\Big)^2\\
&\quad+\Big(\xi'+\frac{\xi}{r}+2k\eta\Big)^2+\Big(-\zeta'+\frac{\zeta}{r}\Big)^2+(\eta'+k\xi)^2+k^2\zeta^2\Big]rdr
\\
&\quad+\int_0^{r_0}\Big(\xi'+\frac{\xi}{r}-k\eta\Big)^2rdr+\int_0^{r_0}\rho(|\xi|^2+|\eta|^2+|\zeta|^2)rdr.
\end{split}
\end{equation}
\end{defi}
Next, we consider the basic estimate of $\xi$ on
	 the interval $(0,\frac{r_0}{6})$.
\begin{lem} \label{xi-bound}
	For any $r\in (0,\frac{r_0}{6})$,  it holds that
\begin{equation}\label{xi-bound-1}
\begin{split}
\xi(r)
\leq\sqrt{\frac{6}{r_0}}\Big(\int_{\frac{r_0}{6}}^{\frac{r_0}{3}}|\xi(a)|^2da\Big)^{\frac{1}{2}}
+\Big|\Big(\int_{r}^{\frac{r_0}{3}}|\xi'|^2sds\Big)^{\frac{1}{2}}\Big|
\Big|\ln \frac{r_0}{3}-\ln r\Big|^{\frac{1}{2}}.
\end{split}
\end{equation}
\end{lem}
\begin{proof}
Notice that
$|\xi(r)|=\Big|\xi(a)-\int_r^a\xi'dr\Big|\leq |\xi(a)|+\Big|\Big(\int_{r}^{a}|\xi'|^2sds\Big)^{\frac{1}{2}}\Big(\int_r^a\frac{1}{s}ds\Big)^{\frac{1}{2}}\Big|.$
Integrating the above inequality about $a$ on the interval $(\frac{r_0}{6},\frac{r_0}{3})$, we have
\begin{equation*}
\begin{split}
\frac{r_0}{6}|\xi(r)|\leq\int_{\frac{r_0}{6}}^{\frac{r_0}{3}}|\xi(a)|da
+\frac{r_0}{6}\Big|\Big(\int_{r}^{\frac{r_0}{3}}|\xi'|^2sds\Big)^{\frac{1}{2}}\Big|
\Big|\ln {\frac{r_0}{3}}-\ln r\Big|^{\frac{1}{2}},
\end{split}
\end{equation*}
which gives \eqref{xi-bound-1}.
\end{proof}
Then we introduce the estimate of $\xi(r)$ for any  $r\in (\frac{r_0}{3},r_0)$.
\begin{lem} \label{xi-bound-m-general}
	For any $r\in (\frac{r_0}{3},r_0)$, it holds that
	\begin{equation}\label{xi-bound-1-m-general}
	\begin{split}
	\xi(r) \leq\sqrt{\frac{6}{r_0}}\Big(\int_{\frac{r_0}{3}}^{\frac{r_0}{2}}|\xi(b)|^2db\Big)^{\frac{1}{2}}
	+\Big(\int_{\frac{r_0}{3}}^{r}|\xi'|^2|r_0-s|ds\Big)^{\frac{1}{2}} \Big|\ln\frac{2r_0}{3}-\ln(r_0-r)\Big|^\frac{1}{2}.
	\end{split}
	\end{equation}
\end{lem}
\begin{proof}
	Notice that
	\begin{equation*}
	\xi(r)=\xi(b)+\int_b^r\xi'dr\leq \xi(b)+\Big(\int_{b}^{r}|\xi'|^2|r_0-s|ds\Big)^{\frac{1}{2}}\Big(\int_b^r\frac{1}{r_0-s}ds\Big)^{\frac{1}{2}}.
	\end{equation*}
	Integrating the above inequality about $b$ on the interval $(\frac{r_0}{3},\frac{r_0}{2})$, we have
	\begin{equation*}
	\begin{split}
	\frac{r_0}{6}\xi(r)\leq\int_{\frac{r_0}{3}}^{\frac{r_0}{2}}|\xi(b)|db
	+\frac{r_0}{6}\Big(\int_{\frac{r_0}{3}}^{r}|\xi'|^2|r_0-s|ds\Big)^{\frac{1}{2}}
	\Big(\int_{\frac{r_0}{3}}^r\frac{1}{r_0-s}ds\Big)^{\frac{1}{2}},
	\end{split}
	\end{equation*}
	which is divided by $\frac{r_0}{6}$, gives \eqref{xi-bound-1-m-general}.
\end{proof}

From the Definition \ref{defi-X}, Lemma \ref{xi-bound} and Lemma \ref{xi-bound-m-general}, we can show the following compactness results.
\begin{prop}\label{embeddding-b}
Let $\pi_i$ for $ = 1, 2, 3$ denote the projection operator onto the $i$-th factor,
there holds that $\pi_i: X_k\rightarrow Z_i$ is a bounded, linear, compact map for $ i = 1, 2,3$,
	where the spaces 
	\begin{equation}\label{defi-Z}
Z_1=\{\xi \in L^2(0,r_0)\}, \quad  Z_2=\{\eta\in L^2(0,r_0)\}, \quad Z_3=\{\zeta\in L^2(0,r_0)\}.
	\end{equation} 
	We denote them by 
	\begin{equation}\label{defi-subset-z1}
X_k\subset\subset Z_1,
	\end{equation} 	\begin{equation}\label{defi-subset-z2}
	X_k\subset\subset Z_2,
	\end{equation} \begin{equation}\label{defi-subset-z3}
	X_k\subset\subset Z_3.
	\end{equation} 
\end{prop}
\begin{proof}
 Assume that $\|(\xi_n,\eta_n,\zeta_n)\|_{X_k}\leq C$ for $n\in \mathbb{N}$. When $m=0$, from Proposition \ref{new-prop-bian-guo}, we know that
\begin{equation*}
\int_0^{r_0}{\xi'_n}^2rdr+ \int_0^{r_0}{\eta'_n}^2rdr+ \int_0^{r_0}{\zeta'_n}^2rdr\leq C.
\end{equation*}
Fix any $\kappa>0$.
We claim that there exists a subsequence $\{\xi_{n_i}\}$ so that 
\begin{equation}\label{claim}
\sup_{i,j}\|\xi_{n_i}-\xi_{n_j}\|_{Z_1}\leq \kappa.
\end{equation}
 Therefore, from Lemma \ref{xi-bound}, choosing $0<s_0<\frac{r_0}{3}$ small enough such that $C\Big(\frac{6}{r_0}
	+ \Big|\ln \frac{r_0}{3}\Big|+1\Big)s_0\leq  \frac{\kappa}{6}$,  then we have
\begin{equation}\label{small-s0}
\begin{split}
\int_{0}^{s_0}\xi_n^2(r)dr
&\leq\frac{6C}{r_0}s_0
+C\int_{0}^{s_0}\Big|\Big(\int_{r}^{\frac{r_0}{3}}|\xi'_n|^2sds\Big)\Big|
\Big|\ln \frac{r_0}{3}-\ln r\Big|dr
\\
&\leq \frac{6C}{r_0}s_0
+C \Big|\ln \frac{r_0}{3}\Big| s_0+C|s_0\ln s_0|\\
&\leq C\Big(\frac{6}{r_0}
+ \Big|\ln \frac{r_0}{3}\Big|+1\Big)s_0\leq  \frac{\kappa}{6}.
\end{split}
\end{equation}
On the other hand, choosing $\frac{r_0}{3}<s_1<r_0$ and $s_1$ close enough to $r_0$, such that $ C\Big(\frac{6}{r_0}
	+\Big|\ln \frac{2r_0}{3}\Big|+1\Big)(r_0-s_1)\leq  \frac{\kappa}{6}$,
by Lemma \ref{xi-bound-m-general}, we can show that
\begin{equation}\label{third-inter}
\begin{split}
\int_{s_1}^{r_0}\xi_n^2(r)dr
&\leq\frac{6C}{r_0}(r_0-s_1)
+C\int_{s_1}^{r_0}\Big(\int_{\frac{r_0}{3}}^{r}|\xi_n'|^2|r_0-s|ds\Big)
\Big|\ln\frac{2r_0}{3}-\ln(r_0-r)\Big|dr\\
&
\leq\frac{6C}{r_0}(r_0-s_1)+C\int_{s_1}^{r_0}\Big|\ln \frac{2r_0}{3}-\ln(r_0-r)\Big|dr
\\
&\leq\frac{6C}{r_0}(r_0-s_1)
+C\Big|\ln \frac{2r_0}{3}\Big|(r_0-s_1)+C(r_0-s_1)\\
&\quad+C|\ln(r_0-s_1)(r_0-s_1)|\\
&\leq C\Big(\frac{6}{r_0}
+\Big|\ln \frac{2r_0}{3}\Big|+1\Big)(r_0-s_1)\leq  \frac{\kappa}{6},
\end{split}
\end{equation}
where we have used the facts  
$$\int_{\frac{r_0}{3}}^{r_0}|r_0-r|\big|\xi'\big|^2dr\leq \int_{\frac{r_0}{3}}^{r_0}\big|\xi'\big|^2rdr+\int_{\frac{r_0}{3}}^{r_0}\big|\xi'\big|^2r_0dr\leq C.$$		
Since the subinterval $(s_0,s_1)$ avoids the singularity of $\frac{1}{r}$ and the degenerate of the density $\rho$ on the boundary $r=r_0$, the function $\xi_n$ is uniformly bounded in $H^1(s_0,s_1)$. By the compact embedding $H^1(s_0,s_1)\subset\subset C^0(s_0,s_1)$, one can extract a subsequence $\{\xi_{n_i}\}$ that converges in $L^\infty(s_0,s_1)$. So for $i, j$ large enough, it holds that 
$\sup_{i,j}\|\xi_{n_i}-\xi_{n_j}\|^2_{L^\infty(s_0,s_1)}\leq \frac{\kappa}{3(s_1-s_0)}$.
Then  along the above subsequence one can get 
from  \eqref{small-s0} and \eqref{third-inter} that 
\begin{equation}
\begin{split}
&\|\xi_{n_i}-\xi_{n_j}\|^2_{L^2(0,r_0)}=\int_0^{r_0}|\xi_{n_i}-\xi_{n_j}|^2dr=\int_0^{s_0}+\int_{s_0}^{s_1}+\int_{s_1}^{r_0}\\
&\quad\leq 2C\Big(\frac{6}{r_0}
+ \Big|\ln \frac{r_0}{3}\Big|+1\Big)s_0
+(s_1-s_0)\sup_{i,j}\|\eta_{n_i}-\eta_{n_j}\|^2_{L^\infty(s_0,s_1)}\\
&\qquad+2C\Big(\frac{6}{r_0}
+\Big|\ln \frac{2r_0}{3}\Big|+1\Big)(r_0-s_1)
\leq \kappa,
\end{split}
\end{equation}
which implies the claim \eqref{claim} and the compactness result \eqref{defi-subset-z1}.
Similarly as the above estimates, we can prove \eqref{defi-subset-z2} and \eqref{defi-subset-z3}. We finish the proof.
\end{proof} 
Now, we give the variational analysis about the case $m=0$ and any $k\in\mathbb{Z}$.
In order to understand $\mu$,
we consider the energy \eqref{sausage-in-v} and 
\begin{equation}\label{constraint}
\mathcal{J}(\xi,\eta,\zeta)=2\pi^2\int_0^{r_0} \rho (|\xi|^2+|\eta|^2+|\zeta|^2)rdr.
\end{equation}
From Definition \ref{defi-X} and Proposition \ref{embeddding-b}, we can get that	$E_{0,k}$ and $\mathcal{J}$ are both well defined on the space $X_k$. 
\begin{lem}
	$E_{0,k}$ and $\mathcal{J}$ are both well defined on the space $X_k$. 
\end{lem}
\begin{proof}
Applying Lemma \ref{xi-bound}, we get  for $0<s_0<\frac{r_0}{6}$ small enough that
\begin{equation*}
\begin{split}
\Big|\int_0^{s_0}p'\xi^2dr\Big|&\leq C	\int_{0}^{s_0}\xi^2dr\leq	C\int_{0}^{s_0}\Big(\int_{\frac{r_0}{6}}^{\frac{r_0}{3}}|\xi(b)|^2db\Big)dr
\\
&\quad+C\int_0^{s_0}\Big(\int_{r}^{\frac{r_0}{3}}|\xi'|^2sds\Big)
\Big|\ln \frac{r_0}{3}-\ln r\Big|dr\\
&
\leq C\mathcal{J}+C\Big(\int_{0}^{\frac{r_0}{3}}|\xi'|^2sds\Big)\int_{0}^{s_0}\Big|\ln \frac{r_0}{3}-\ln r\Big|dr
\\&\leq C\mathcal{J}
+C\int_{0}^{\frac{r_0}{3}}|\xi'|^2sds\leq C\mathcal{J}+C\|\big(\xi,\eta,\zeta\big)\|^2_{X_k}.
\end{split}
\end{equation*}	
On the other hand, from the proof of Proposition \ref{embeddding-b},  it follows that
\begin{equation*}
\Big|\int_{s_0}^{r_0}p'\xi^2dr\Big|\leq C\int_{s_0}^{r_0} \xi^2dr\leq C\|\big(\xi,\eta,\zeta\big)\|^2_{X_k}.
\end{equation*}
Hence, we get that
\begin{equation}\label{pressure-term-control}
\Big|\int_0^{r_0}p'\xi^2dr\Big|\leq  C\mathcal{J}+C\|\big(\xi,\eta,\zeta\big)\|^2_{X_k},
\end{equation}
which  implies from Proposition \ref{embeddding-b} 
that
\begin{equation*}
\begin{split}
|E_{0,k}(\xi,\eta,\zeta)|
&\leq C\mathcal{J}+C\|\big(\xi,\eta,\zeta\big)\|^2_{X_k}+C\int_0^{r_0}\Big\{B_{\theta}^2\big[k\eta-\frac{1}{r}((r\xi)'-2\xi)\big]^2\Big\}rdr\\
&\quad+C\int_0^{r_0} p\big|\frac{1}{r}(r\xi)'\big|^2rdr+ C\|\rho\|^{\gamma-1}_{L^\infty}\int_0^{r_0} \rho |\eta|^2rdr+C\int_0^{r_0}\Big[\Big(-2\xi'+\frac{\xi}{r}-k\eta\Big)^2\\
&\quad+\Big(\xi'-\frac{2\xi}{r}-k\eta\Big)^2+\Big(\xi'+\frac{\xi}{r}+2k\eta\Big)^2+\Big(-\zeta'+\frac{\zeta}{r}\Big)^2+(\eta'+k\xi)^2+k^2\zeta^2\Big]rdr
\\
&\quad+C\int_0^{r_0}\Big(\xi'+\frac{\xi}{r}-k\eta\Big)^2rdr\leq C	\|(\xi,\eta,\zeta)\|^2_{X_k}.
\end{split}
\end{equation*}
Hence, $E_{0,k}$ and $\mathcal{J}$ are both well-defined on the space $X_k$.
\end{proof}	
Define function $g(r)=\sup_{r\leq s\leq r_0}\frac{p(s)}{-p'(s)}$, then by Definition \ref{admissible}/admissibility of the presure $p$, we can get the following lemma.
\begin{lem}[Lemma 3.7 in \cite{bian-guo-tice-inviscid}]\label{presure-boundary}
	Assume $s_1$ near $r_0$,  then it holds that
	\begin{equation}\label{weight-p'}
	\begin{split}
	\int_{s_1}^{r_0}-p'(r)\xi^2dr
\leq 2 p(s_1)\xi^2(s_1)+4g(s_1)
	\int_{s_1}^{r_0}p\xi'^2dr,
	\end{split}
	\end{equation}
	with $g(s_1)\rightarrow 0$ as $s_1\rightarrow r_0$.
\end{lem}

Now we define \begin{equation}\label{con-mono}
\lambda(s)=\inf_{(\xi,\eta,\zeta)\in X_k}\frac{E_{0,k}(\xi,\eta,\zeta;s)}{\mathcal{J}(\xi,\eta,\zeta)}.
\end{equation}
Consider the set 
\begin{equation}\label{set a-b}
\mathcal{A}_1=\{(\xi,\eta,\zeta)\in X_k| \mathcal{J}(\xi,\eta,\zeta)=1\}.
\end{equation}
We want to show that the infimum of $E_{0,k}(\xi,\eta,\zeta)$ over the set $\mathcal{A}_1$ is achieved and is negative. And then we show that the minimizer solves \eqref{spectral-formulation}  with $m=0$ and $k\neq 0$ and the corresponding boundary conditions. First, we prove that the energy $E_{0,k}$ has a lower bound on the set $ \mathcal{A}_1$ and the coercivity estimate holds.
\begin{lem}
The energy $E_{0,k}(\xi,\eta,\zeta)$ has a lower bound on the set $ \mathcal{A}_1$ and any minimizing sequence $(\xi_n,\eta_n,\zeta_n)$ is bounded in $X_k$.
\end{lem}
\begin{proof}	
We can directly get from \eqref{sausage-in-v} that
\begin{equation}\label{con-mono-use-bian}
\begin{split}
E_{0,k}(\xi,\eta,\zeta)&\geq2\pi^2\int_0^{r_0}\Big\{\frac{2p'\xi^2}{r}+\gamma p\Big[\frac{1}{r}(r\xi)'-k\eta\Big]^2\Big\}rdr\\
&\quad+2\pi^2\int_0^{r_0}\tilde{\varepsilon}\Big[\frac{2}{9}\Big(-2\xi'+\frac{\xi}{r}-k\eta\Big)^2+\frac{2}{9}\Big(\xi'-\frac{2\xi}{r}-k\eta\Big)^2\Big]rdr
\\
&\quad+2\pi^2\int_0^{r_0}\tilde{\delta}\Big(\xi'+\frac{\xi}{r}-k\eta\Big)^2rdr\\
&\geq 
2\pi^2\min(\tilde{\varepsilon},\tilde{\delta})\int_0^{r_0}\Big[\frac{2}{9}\Big(-2\xi'+\frac{\xi}{r}-k\eta\Big)^2+\frac{2}{9}\Big(\xi'+\frac{\xi}{r}-k\eta\Big)^2\Big]rdr\\
&\quad+2\pi^2\min(\tilde{\varepsilon},\tilde{\delta})\int_0^{r_0}\Big[\frac{2}{9}\Big(\xi'-\frac{2\xi}{r}-k\eta\Big)^2+\frac{2}{9}\Big(\xi'+\frac{\xi}{r}-k\eta\Big)^2\Big]rdr\\
&\quad+4\pi^2\int_0^{r_0}p'\xi^2dr\\
&\geq2\pi^2\min(\tilde{\varepsilon},\tilde{\delta})\int_0^{r_0}{\xi'}^2rdr+2\pi^2\min(\tilde{\varepsilon},\tilde{\delta})\int_0^{r_0}\frac{\xi^2}{r}dr+4\pi^2\int_0^{r_0}p'\xi^2dr,
\end{split}
\end{equation}
for any $(\xi,\eta,\zeta)\in \mathcal{A}$.
Here, we have used the facts that $a^2+b^2\geq \frac{1}{2}(a-b)^2,$ with $a=-2\xi'+\frac{\xi}{r}-k\eta$, $b=\xi'+\frac{\xi}{r}-k\eta$; $a=\xi'-\frac{2\xi}{r}-k\eta$, $b=\xi'+\frac{\xi}{r}-k\eta$.
By Lemma \ref{xi-bound}, choosing $0<s_0<\frac{r_0}{6}$ small enough  such that
$Cs_0\leq \frac{1}{4}$, we obtain that
\begin{equation}\label{estimate-p'}
\begin{split}
\Big|\int_0^{s_0}p'\xi^2dr\Big|&\leq C	\int_{0}^{s_0}\xi^2dr\leq	C\int_{0}^{s_0}\Big(\int_{\frac{r_0}{6}}^{\frac{r_0}{3}}|\xi(b)|^2db\Big)dr
\\
&\quad+C\int_0^{s_0}\Big(\int_{r}^{\frac{r_0}{3}}|\xi'|^2sds\Big)
\Big|\ln \frac{r_0}{3}-\ln r\Big|dr\\
&
\leq C\mathcal{J}s_0+C\Big(\int_{0}^{\frac{r_0}{3}}|\xi'|^2sds\Big)\int_{0}^{s_0}\Big|\ln \frac{r_0}{3}-\ln r\Big|dr
\\&\leq C\mathcal{J}s_0
+Cs_0\int_{0}^{\frac{r_0}{3}}|\xi'|^2sds\\&\leq C\mathcal{J}s_0
+2Cs_0\pi^2\min(\tilde{\varepsilon},\tilde{\delta})\int_0^{r_0}{\xi'}^2rdr\\
&\leq C\mathcal{J}s_0
+\frac{1}{2}\pi^2\min(\tilde{\varepsilon},\tilde{\delta})\int_0^{r_0}{\xi'}^2rdr.
\end{split}
\end{equation}	
On the other hand, from the proof of Proposition \ref{embeddding-b}, choosing $s_0<s_1<r_0$ close to $r_0$ such that
 $C(r_0-s_1)\leq \frac{1}{4}$, we get that
\begin{equation}\label{estimates-p'-3}
\Big|\int_{s_0}^{s_1}p'\xi^2dr\Big|\leq C\int_{s_0}^{s_1} \xi^2dr\leq C\mathcal{J}
\end{equation}
and
\begin{equation}\label{lower-bound-large-s1}
\begin{split}
\Big|\int_{s_1}^{r_0}p'\xi^2dr\Big|&\leq C\int_{s_1}^{r_0} \xi^2dr
\leq C(r_0-s_1)\|\xi\|^2_{L^\infty(s_1,r_0)}\leq C(r_0-s_1)\|\xi\|^2_{H^1(s_1,r_0)}\\
&\leq C(r_0-s_1)\pi^2\min(\tilde{\varepsilon},\tilde{\delta})\Big(\int_{s_1}^{r_0}{\xi'}^2rdr+\int_{s_1}^{r_0}\frac{\xi^2}{r}dr\Big)\\
&\leq C(r_0-s_1)\pi^2\min(\tilde{\varepsilon},\tilde{\delta})\Big(\int_{0}^{r_0}{\xi'}^2rdr+\int_{0}^{r_0}\frac{\xi^2}{r}dr\Big)\\
&\leq \frac{1}{4}\pi^2\min(\tilde{\varepsilon},\tilde{\delta})\Big(\int_{0}^{r_0}{\xi'}^2rdr+\int_{0}^{r_0}\frac{\xi^2}{r}dr\Big),
\end{split}
\end{equation}
where we have used the facts  $H^1(s_1,r_0)\subset\subset L^\infty(s_1,r_0)$.
Therefore, we can prove
\begin{equation}
\begin{split}
E_{0,k}(\xi,\eta,\zeta)
&\geq\pi^2\min(\tilde{\varepsilon},\tilde{\delta})\int_0^{r_0}{\xi'}^2rdr+\frac{3}{2}\pi^2\min(\tilde{\varepsilon},\tilde{\delta})\int_0^{r_0}\frac{\xi^2}{r}dr-2C\mathcal{J}\\
&\geq -2C\mathcal{J}=-2C,
\end{split}
\end{equation}
which gives that the energy $E_{0,k}(\xi,\eta,\zeta)$ has a lower bound on the set $ \mathcal{A}_1$. Using the facts that	$\mathcal{J}=1$ and $E_{0,k}$ has a lower bound on set $\mathcal{A}_1$, we can choose a minimizing sequence such that along the minimizing sequence, 
	we have $M\leq E_{0,k}(\xi_n,\eta_n,\zeta_n)<M+1$, and for the minimizing sequence,
	we can show coercivity estimate: 
	\begin{equation}\label{coercivity}
\begin{split}
\|(\xi_n,\eta_n,\zeta_n)\|^2_{X_k}
\leq C\mathcal{J}+C(M+1)\leq C.
	\end{split}
	\end{equation}
\end{proof}
We now show that the infimum of $E_{0,k}$ over the set $\mathcal{A}_1$ is negative for the case $m=0$ and any $k\in \mathbb{Z}$. 
\begin{prop}\label{infimum-A-3}
For any fixed $k=k_0$, there exists constant $s_0>0$ depending on $\varepsilon$, $\delta$ and $k_0$ so that for sufficiently small $s\leq s_0$, it holds that $\lambda=\inf E_{0,k}<0$. 
\end{prop}
\begin{proof}
	Since both $E_{0,k}$ and $\mathcal{J}$ are homogeneous degree $2$, it suffices to show that 
	\begin{equation*}
	\inf_{(\xi,\eta,\zeta)\in X_k}\frac{E_{0,k}(\xi,\eta,\zeta)}{\mathcal{J}(\xi,\eta,\zeta)}<0.
	\end{equation*}
	But since $\mathcal{J}$ is positive definite, one may reduce to constructing  $(\xi,\eta,\zeta)\in X_k$ (see \eqref{defi-Xk}) such that $E_{0,k}(\xi,\eta,\zeta)<0$. 
	Notice that the first integral in the energy \eqref{sausage-in-v} can be rewritten as 
	\begin{equation*}
	\begin{split}
2\pi^2\int_0^{r_0}\Big\{\Big[\frac{2p'}{r}+\frac{4\gamma p B_{\theta}^2}{r^2(\gamma p+B_{\theta}^2)}\Big]\xi^2+(\gamma p+B_{\theta}^2)\Big[k\eta-\frac{1}{r}\Big((r\xi)'-\frac{2B_{\theta}^2}{\gamma p +B_{\theta}^2}\xi\Big)\Big]^2\Big\}rdr.
	\end{split}
	\end{equation*}
From the property (i) of Proposition \ref{steady-lem},  we can choose a smooth function $\xi^*\in C_c^{\infty}(0,r_0)$ such that $$2\pi^2\int_0^{r_0}\Big[\frac{2p'}{r}+\frac{4\gamma p B_{\theta}^2}{r^2(\gamma p+B_{\theta}^2)}\Big]{\xi^*}^2rdr< 0.$$
	Then, 
	we can assume that $k\eta^*=\frac{1}{r}\Big((r\xi^*)'-\frac{2B_{\theta}^2}{\gamma p +B_{\theta}^2}\xi^*\Big)$, such that the second term  in $E_{0,k}(\xi^*,\eta^*,\zeta^*)$ vanishes. Here, $\xi^*$ and $\eta^*$ are smooth function, belong to the space $X_k$. Then we choose  $\zeta^*=0$. 
	
From $\xi^*\in C_c^{\infty}(0,r_0)$  and $k\eta^*=\frac{1}{r}\Big((r\xi^*)'-\frac{2B_{\theta}^2}{\gamma p +B_{\theta}^2}\xi^*\Big)$,
for any finite fixed $k=k_0$,  it follows that
		\begin{equation*}
		\begin{split}
	&	2\pi^2\varepsilon\int_0^{r_0}\Big[\frac{2}{9}\Big(-2{\xi^*}'+\frac{\xi^*}{r}-k\eta^*\Big)^2
	+\frac{2}{9}\Big({\xi^*}'-\frac{2\xi^*}{r}-k\eta^*\Big)^2
	+\frac{2}{9}\Big({\xi^*}'+\frac{\xi^*}{r}+2k\eta^*\Big)^2\\
	&\quad+({\eta^*}'+k\xi^*)^2\Big]rdr+2\pi^2\delta\int_0^{r_0}\Big({\xi^*}'+\frac{\xi^*}{r}-k\eta^*\Big)^2rdr\leq C.
		\end{split}
		\end{equation*}	
		Therefore, the energy takes  as follows
	\begin{equation*}
	\begin{split}
	\widetilde{E}(\xi^*)&=E_{0,k}\Big(\xi^*, \frac{1}{kr}\Big((r\xi^*)'-\frac{2B_{\theta}^2}{\gamma p +B_{\theta}^2}\xi^*\Big),0\Big)=2\pi^2\int_0^{r_0}\Big[\frac{2p'}{r}+\frac{4\gamma p B_{\theta}^2}{r^2(\gamma p+B_{\theta}^2)}\Big]{\xi^*}^2rdr\\
	&\quad+2\pi^2\int_0^{r_0}\tilde{\varepsilon}\Big[\frac{2}{9}\Big(-2{\xi^*}'+\frac{\xi^*}{r}-k\eta^*\Big)^2+\frac{2}{9}\Big({\xi^*}'-\frac{2\xi^*}{r}-k\eta^*\Big)^2\\
	&\quad+\frac{2}{9}\Big({\xi^*}'+\frac{\xi^*}{r}+2k\eta^*\Big)^2+({\eta^*}'+k\xi^*)^2\Big]rdr+2\pi^2\int_0^{r_0}\tilde{\delta}\Big({\xi^*}'+\frac{\xi^*}{r}-k\eta^*\Big)^2rdr\\
	&\leq 2\pi^2\int_0^{r_0}\Big[\frac{2p'}{r}+\frac{4\gamma p B_{\theta}^2}{r^2(\gamma p+B_{\theta}^2)}\Big]{\xi^*}^2rdr+sC.
	\end{split}
	\end{equation*}	
	Then there exists $s_0>0$ such that $s\leq s_0$,  
	\begin{equation*}
		\widetilde{E}(\xi^*) \leq \pi^2\int_0^{r_0}\Big[\frac{2p'}{r}+\frac{4\gamma p B_{\theta}^2}{r^2(\gamma p+B_{\theta}^2)}\Big]{\xi^*}^2rdr<0,
	\end{equation*}
	which implies the result.
\end{proof}

 We now deduce  the existence of a minimizer of $E_{0,k}$ on the set $\mathcal{A}_1$.
\begin{prop}\label{infimum-A}
The energy $E_{0,k}$ achieves its infimum on the set $\mathcal{A}_1$.	
\end{prop}
\begin{proof}
	First note  that $E_{0,k}$ is bounded below on the set $\mathcal{A}_1$. Let $(\xi_n,\eta_n,\zeta_n) \in \mathcal{A}_1$ be a minimizing sequence. Then, we know that $(\xi_n,\eta_n,\zeta_n)$ is bounded in $X_k$, so up to the extraction of a subsequence $\psi_n=|B_{\theta}|\Big[k\eta_n-\frac{1}{r}((r\xi_n)'-2\xi_n)\Big]r^{\frac12}\rightharpoonup \psi=|B_{\theta}|\Big[k\eta-\frac{1}{r}((r\xi)'-2\xi)\Big]r^{\frac12}$ weakly in $L^2$, and $\xi_n\rightarrow \xi$, $\eta_n\rightarrow \eta$  and $\zeta_n\rightarrow \zeta$ strongly in $L^2$ from the compactness results in Proposition \ref{embeddding-b}.
	
		By weak lower semi-continuity, since $\psi_n\rightharpoonup \psi$ in the space $L^2(0,r_0)$,
		we have 
		$$\int_0^{r_0}B_{\theta}^2\Big[k\eta-\frac{1}{r}((r\xi)'-2\xi)\Big]^2rdr\leq\liminf_{n\to\infty}\int_0^{r_0}B_{\theta}^2\Big[k\eta_n-\frac{1}{r}((r\xi_n)'-2\xi_n)\Big]^2rdr.
		$$	
	Because of the quadratic structure of all the terms in the integrals defining $E_{0,k}$, similarly by weak lower semicontinuity and strong $L^2$ convergence, we get that 
$$	E_{0,k}(\xi,\eta,\zeta)\leq \liminf_{n\to\infty}E_{0,k}(\xi_n,\eta_n,\zeta_n)=\inf_{\mathcal{A}_1}E_{0,k}.$$

All that remains is to show that $(\xi,\eta,\zeta)\in \mathcal{A}_1 $. By the compactness results in Proposition \ref{embeddding-b}, when $m=0$ and any $k\in\mathbb{Z}$, we can prove up to subsequence
 $\lim_{n\rightarrow \infty}\mathcal{J}(\xi_n, \eta_n,\zeta_n)=\mathcal{J}(\xi,\eta,\zeta)=1$, which implies that  $(\xi,\eta,\zeta)\in \mathcal{A}_1 $.
\end{proof}
We now prove that the minimizer constructed in the previous result satisfies Euler-Lagrange equations equivalent to \eqref{spectral-formulation}  with $m=0$ and any $k\in\mathbb{Z}$.

\begin{prop}\label{infimum-A-2}
	Let $(\xi,\eta,\zeta)\in\mathcal{A}_1$ be the minimizer of $E_{0,k}$ constructed in Proposition \ref{infimum-A}. Then $(\xi,\eta,\zeta)$ are smooth when restricted to $(0,r_0)$ and satisfy 
	\begin{equation} \label{spectal formulation-2} 
		\begin{split}
	&\left(
	\begin{array}{ccc}
	\frac{d}{dr}\frac{\gamma p+B_{\theta}^2}{r}\frac{d}{dr}r
	-r(\frac{B_{\theta}^2}{r^2})'&-\frac{d}{dr}k(\gamma p+B_{\theta}^2)-\frac{2kB_{\theta}^2}{r}&0\\
	\frac{k(\gamma p+B_{\theta}^2)}{r}\frac{d}{dr}r-\frac{2kB_{\theta}^2}{r}&
	-k^2(\gamma p+B_{\theta}^2)&0\\
0&0&0
	\end{array}
	\right)
	\left(
	\begin{array}{lll}
	\xi   \\
	\eta \\
	\zeta\\
	\end{array}
	\right)+\\
	&	\left(
	\begin{array}{ccc}
	(\frac{4\tilde{\varepsilon}}{3}+\tilde{\delta})\frac{d^2}{dr^2}+(\frac{4\tilde{\varepsilon}}{3}+\tilde{\delta})\frac{d}{dr}\frac{1}{r}-\tilde{\varepsilon}k^2&-(\frac{\tilde{\varepsilon}}{3}+\tilde{\delta})k\frac{d}{dr}&0\\
	(\frac{\tilde{\varepsilon}}{3}+\tilde{\delta})k\frac{d}{dr}+(\frac{\tilde{\varepsilon}}{3}+\tilde{\delta})\frac{k}{r}&
	\tilde{\varepsilon}\frac{d^2}{dr^2}+\frac{\tilde{\varepsilon}}{r}\frac{d}{dr}-(\frac{4\tilde{\varepsilon}}{3}+\tilde{\delta})k^2&0\\
	0&0&\tilde{\varepsilon}\frac{d^2}{dr^2}+\tilde{\varepsilon}\frac{d}{dr}\frac{1}{r}-\tilde{\varepsilon}k^2
	\end{array}
	\right)
	\left(
	\begin{array}{lll}
	\xi   \\
	\eta \\
	\zeta\\
	\end{array}
	\right)	
	\\&\quad=-\rho \lambda \left(
	\begin{array}{lll}
	\xi   \\
	\eta \\
	\zeta\\
	\end{array}
	\right),
	\end{split}
	\end{equation}
	along with the interface boundary conditions	
	\begin{equation}\label{inter-two-com}
\begin{split}
&\Big[B_{\theta}^2\xi-B_{\theta}^2\xi' r+kB_\theta^2\eta r \Big]n\\
&\quad+\tilde{\varepsilon}\Big(-2\xi'r,i\zeta'r-i\zeta,-i\eta'r-ik\xi r\Big)^T\\
&\quad-\Big(\tilde{\delta}-\frac{2}{3}\tilde{\varepsilon}\Big)\Big[\xi'r+\xi-k\eta r\Big]n=0, \quad \mbox{at} \quad r=r_0.
\end{split}
\end{equation}
\end{prop}
\begin{proof}
	Since we want to use the structure of the energy and properties of functional space, we first change the spectral formula \eqref{spectal formulation-2} into the following equations by a simple computation
		\begin{equation}\label{new-equations}
		\begin{cases}
		&-\frac{d}{dr}\Big\{\gamma p\big[k\eta-\frac{1}{r}(r\xi)'\big] \Big\}-\frac{d}{dr}\Big\{B^2_{\theta}\big[k\eta-\frac{1}{r}\big((r\xi)'-2\xi\big)\big]\Big\}
		\\&\quad+	(\frac{4\tilde{\varepsilon}}{3}+\tilde{\delta})\frac{d^2}{dr^2}\xi+(\frac{4\tilde{\varepsilon}}{3}+\tilde{\delta})\frac{d}{dr}\frac{\xi}{r}-\tilde{\varepsilon}k^2\xi-(\frac{\tilde{\varepsilon}}{3}+\tilde{\delta})k\frac{d}{dr}\eta
		\\
		&\quad-\frac{2B^2_{\theta}}{r}\big[k\eta-\frac{1}{r}\big((r\xi)'-2\xi\big)\big]-\frac{2p'\xi}{r}=-\rho \lambda\xi,\\
		&-k(\gamma p+B^2_{\theta})\Big[k\eta-\frac{1}{r}(r\xi)'+\frac{2B^2_\theta}{r(\gamma p+B^2_{\theta})}\xi\Big]+(\frac{\tilde{\varepsilon}}{3}+\tilde{\delta})k\frac{d}{dr}\xi+(\frac{\tilde{\varepsilon}}{3}+\tilde{\delta})\frac{k\xi}{r}\\
		&\quad+	\tilde{\varepsilon}\frac{d^2}{dr^2}\eta+\frac{\tilde{\varepsilon}}{r}\frac{d}{dr}\eta-(\frac{4\tilde{\varepsilon}}{3}+\tilde{\delta})k^2\eta=-\rho \lambda\eta,\\
		&\tilde{\varepsilon}\frac{d^2}{dr^2}\zeta+\tilde{\varepsilon}\frac{d}{dr}\frac{\zeta}{r}-\tilde{\varepsilon}k^2\zeta=-\rho \lambda\zeta.
		\end{cases}
		\end{equation}
	Next, we prove the minimization $\xi$, $\eta$ and $\zeta$ satisfy the equations \eqref{new-equations} in weak sense on $(0,r_0)$.
	
	Fix $(\xi_0,\eta_0,\zeta_0)\in X_{k}$ (see \eqref{defi-Xk}). Define 
	$$j(t,\tau(t))=\mathcal{J}(\xi+t\xi_0+\tau(t)\xi,\eta+t\eta_0+\tau(t)\eta,\zeta+t\zeta_0+\tau(t)\zeta)$$
	and note that $j(0,0)=1$. Moreover, $j$ is smooth, 
	\begin{equation*}
	\begin{split}
	&\frac{\partial j}{\partial t}(0,0)=2\pi^2\int_{0}^{r_0}2\rho(\xi_0\xi +\eta_0\eta+\zeta_0\zeta)rdr, \\
	&\frac{\partial j}{\partial \tau}(0,0)=2\pi^2\int_{0}^{r_0}2\rho(\xi^2 +\eta^2+\zeta^2)rdr=2.
	\end{split}
	\end{equation*}
	So, by the inverse function theorem, we can solve for $\tau=\tau(t)$ in a neighborhood of $0$ as a $C^1$ function of $t$ so that $\tau(0)=0$ and $j(t,\tau(t))=1$. We may differentiate the last equation to find 
	\begin{equation*}
	\frac{\partial j}{\partial t}(0,0)+\frac{\partial j}{\partial \tau}(0,0)\tau'(0)=0,
	\end{equation*}
	which gives that 
	\begin{equation*}
	\tau'(0)=-\frac{1}{2}\frac{\partial j}{\partial t}(0,0)=-2\pi^2\int_{0}^{r_0}\rho(\xi_0\xi +\eta_0\eta+\zeta\zeta_0)rdr.
	\end{equation*}
	Since $(\xi, \eta,\zeta)$ is the minimizer over the set $\mathcal{A}_1$, we may make variations with respect to $(\xi_0,\eta_0,\zeta_0)$ to find that 
	\begin{equation*}
	0=\frac{d}{dt}\bigg|_{t=0}E_{0,k}(\xi+t\xi_0+\tau(t)\xi, \eta+t\eta_0+\tau(t)\eta,\zeta+t\zeta_0+\tau(t)\zeta),
	\end{equation*}
	which together with \eqref{con-mono} and \eqref{set a-b}, implies that
	\begin{equation*}
	\begin{split}
0&=4\pi^2\int_{0}^{r_0}2p'\xi\xi_0dr +4\pi^2\int_{0}^{r_0}B^2_{\theta}\Big[k\eta-\frac{1}{r}\big((r\xi)'-2\xi\big)\Big]
\Big\{-\frac{1}{r}\Big[\big(r\xi_0\big)'-2\xi_0\Big]\Big\}rdr\\
&\quad+4\pi^2\int_{0}^{r_0}\gamma p\Big[\frac{1}{r}(r\xi)'-k\eta\Big](r\xi_0)'dr+4\pi^2\int_{0}^{r_0}kB^2_{\theta}\Big[k\eta-\frac{1}{r}\big((r\xi)'-2\xi\big)\Big]
\eta_0rdr\\
&\quad+4\pi^2\int_{0}^{r_0}\gamma p\Big[\frac{1}{r}(r\xi)'-k\eta\Big](-k\eta_0)rdr+
4\pi^2 \int _0^{r_0}\Big[\frac{2\tilde{\varepsilon}}{3}\Big(-\xi'\xi_0-\xi\xi'_0+\frac{2\xi\xi_0}{r}\Big)\\
&\quad+\Big(\frac{4\tilde{\varepsilon}}{3}+\tilde{\delta}\Big)\xi'\xi'_0r+\tilde{\varepsilon}k^2\xi\xi_0r+\tilde{\delta}\Big(\xi'\xi_0+\xi\xi'_0+\frac{\xi\xi_0}{r}\Big) \Big]dr+4\pi^2 \int _0^{r_0}\Big[\Big(\frac{2\tilde{\varepsilon}}{3}-\tilde{\delta}\Big)k\eta\xi'_0r
\\
&\quad+\Big(\frac{2\tilde{\varepsilon}}{3}-\tilde{\delta}\Big)k\eta\xi_0
+\tilde{\varepsilon}\eta'k\xi_0r
\Big]dr  +4\pi^2 \int_0^{r_0}\Big[\tilde{\varepsilon}k^2\zeta\zeta_0r-\tilde{\varepsilon}\zeta'\zeta_0-\tilde{\varepsilon}\zeta\zeta'_0+\tilde{\varepsilon}\frac{\zeta\zeta_0}{r}+\tilde{\varepsilon}\zeta'\zeta'_0r\Big]dr\\
&\quad+4\pi^2\int_0^{r_0}\Big[\Big(\frac{2\tilde{\varepsilon}}{3}-\tilde{\delta}\Big)k\xi'\eta_0r
+\Big(\frac{2\tilde{\varepsilon}}{3}-\tilde{\delta}\Big)k\xi\eta_0+\tilde{\varepsilon}k\xi\eta'_0r\Big]dr\\
&\quad+4\pi^2\int_0^{r_0}\Big[\Big(\frac{4\tilde{\varepsilon}}{3}+\tilde{\delta}\Big)k^2\eta\eta_0r+\tilde{\varepsilon}\eta'\eta'_0r\Big]dr+2\tau'(0)\lambda.
	\end{split}
	\end{equation*}
	
	Since $\xi_0$, $\eta_0$ and $\zeta_0$ are independent, one has the triplet of equations
	\begin{equation}\label{weak-form}
	\begin{split}
	&\int_{0}^{r_0}2p'\xi\xi_0dr -\int_{0}^{r_0}B^2_{\theta}\Big[k\eta-\frac{1}{r}\big((r\xi)'-2\xi\big)\Big]
	\Big[\big(r\xi_0\big)'-2\xi_0\Big]dr\\
	&\quad+\int _0^{r_0}\Big[\frac{2\tilde{\varepsilon}}{3}\Big(-\xi'\xi_0-\xi\xi'_0+\frac{2\xi\xi_0}{r}\Big)+\Big(\frac{4\tilde{\varepsilon}}{3}+\tilde{\delta}\Big)\xi'\xi'_0r+\tilde{\varepsilon}k^2\xi\xi_0r\\
	&\quad+\tilde{\delta}\Big(\xi'\xi_0+\xi\xi'_0+\frac{\xi\xi_0}{r}\Big) \Big]dr+\int _0^{r_0}\Big[\Big(\frac{2\tilde{\varepsilon}}{3}-\tilde{\delta}\Big)k\eta\xi'_0r
	+\Big(\frac{2\tilde{\varepsilon}}{3}-\tilde{\delta}\Big)k\eta\xi_0+\tilde{\varepsilon}\eta'k\xi_0r\Big]dr \\
	&\quad+\int_{0}^{r_0}\gamma p\Big[\frac{1}{r}(r\xi)'-k\eta\Big](r\xi_0)'dr=\int_{0}^{r_0}\rho\lambda\xi_0\xi rdr,
	\end{split}
	\end{equation}
	\begin{equation}\label{eta-weak}
	\begin{split}
	&\int_{0}^{r_0}kB^2_{\theta}\Big[k\eta-\frac{1}{r}\big((r\xi)'-2\xi\big)\Big]
	\eta_0rdr+\int_{0}^{r_0}\gamma pk\Big[k\eta-\frac{1}{r}(r\xi)'\Big]\eta_0rdr\\
	&\quad+\int_0^{r_0}\Big[\Big(\frac{2\tilde{\varepsilon}}{3}-\tilde{\delta}\Big)k\xi'\eta_0r+\Big(\frac{2\tilde{\varepsilon}}{3}-\tilde{\delta}\Big)k\xi\eta_0+\tilde{\varepsilon}k\xi\eta'_0r\Big]dr+\int_0^{r_0}\Big[\Big(\frac{4\tilde{\varepsilon}}{3}+\tilde{\delta}\Big)k^2\eta\eta_0r\\
	&\quad+\tilde{\varepsilon}\eta'\eta'_0r\Big]dr = \int_{0}^{r_0}\rho\lambda\eta_0\eta rdr,
	\end{split}
	\end{equation}
	\begin{equation}\label{zeta-weak}
	\begin{split}
	&  
	\int_0^{r_0}\Big(\tilde{\varepsilon}k^2\zeta\zeta_0r-\tilde{\varepsilon}\zeta'\zeta_0-\tilde{\varepsilon}\zeta\zeta'_0+\tilde{\varepsilon}\frac{\zeta\zeta_0}{r}+\tilde{\varepsilon}\zeta'\zeta'_0r\Big)dr=\int_{0}^{r_0}\rho\lambda\zeta_0\zeta rdr.
	\end{split}
	\end{equation}
	So $\xi$, $\eta$ and $\zeta$ satisfy \eqref{new-equations} in a weak sense on $(0,r_0)$,
	if $(\xi_0,\eta_0,\zeta_0)$ are chosen compactly supported in $(0,r_0)$.
	Now we prove that the interface boundary conditions \eqref{inter-two-com} are satisfied. 
		From the equations \eqref{new-equations}, we get 
		\begin{equation}\label{p-xi}
		\begin{split}
		&-\frac{d}{dr}\Big\{(\gamma p+B^2_{\theta})\Big[k\eta-\frac{1}{r}\big((r\xi)'-2\xi\big)\Big]\Big\}+	(\frac{4\tilde{\varepsilon}}{3}+\tilde{\delta})\frac{d^2}{dr^2}\xi\\
		&\quad+(\frac{4\tilde{\varepsilon}}{3}+\tilde{\delta})\frac{d}{dr}\frac{\xi}{r}-(\frac{\tilde{\varepsilon}}{3}+\tilde{\delta})k\frac{d}{dr}\eta-\tilde{\varepsilon}k^2\xi+ \frac{2\gamma p\xi'}{r}+ \frac{2\gamma p'\xi}{r}-\frac{2\gamma p \xi }{r^2} \\
		&\quad-\frac{2B^2_{\theta}}{r}\Big[k\eta-\frac{1}{r}\big((r\xi)'-2\xi\big)\Big]-\frac{2p'\xi}{r}=-\rho \lambda\xi,\\
		&\tilde{\varepsilon}\frac{d^2}{dr^2}\eta+\frac{\tilde{\varepsilon}}{r}\frac{d}{dr}\eta+(\frac{\tilde{\varepsilon}}{3}+\tilde{\delta})k\frac{d}{dr}\xi+(\frac{\tilde{\varepsilon}}{3}+\tilde{\delta})\frac{k\xi}{r}-(\frac{4\tilde{\varepsilon}}{3}+\tilde{\delta})k^2\eta\\
		&\quad-k(\gamma p+B^2_{\theta})\Big[k\eta-\frac{1}{r}(r\xi)'+\frac{2B^2_\theta}{r(\gamma p+B^2_{\theta})}\xi\Big]=-\rho\lambda\eta,\\
		&\tilde{\varepsilon}\frac{d^2}{dr^2}\zeta+\tilde{\varepsilon}\frac{d}{dr}\frac{\zeta}{r}-\tilde{\varepsilon}k^2\zeta=-\rho \lambda\zeta.
			\end{split}
		\end{equation}
From $(\xi,\eta,\zeta)\in X_k$ and the compactness results in Proposition \ref{embeddding-b}, we deduce that  $(\xi, \eta, \zeta) \in H^1(\frac{r_0}{2},r_0)\times H^1(\frac{r_0}{2},r_0)\times  H^1(\frac{r_0}{2},r_0)$ when $m=0$ and any $k\in \mathbb{Z}$, which gives that 
		\begin{equation*}
		\begin{split}
		&\frac{d}{dr}\Big\{(\gamma p+B^2_{\theta})\Big[k\eta-\frac{1}{r}\Big((r\xi)'-2\xi\Big)\Big]+	\Big(\frac{4\tilde{\varepsilon}}{3}+\tilde{\delta}\Big)\Big(\xi'+\frac{\xi}{r}\Big)-	\Big(\frac{\tilde{\varepsilon}}{3}+\tilde{\delta}\Big)k\eta\Big\} \in L^2(\frac{r_0}{2}, r_0),\\
		&\tilde{\varepsilon}\frac{d^2}{dr^2}\eta+\frac{\tilde{\varepsilon}}{r}\frac{d}{dr}\eta\in L^2(\frac{r_0}{2}, r_0),
		\,\,\,\frac{d}{dr}\Big(\tilde{\varepsilon}\zeta'+\tilde{\varepsilon}\frac{\zeta}{r}\Big)\in L^2(0,r_0).
		\end{split}
		\end{equation*}			
		 Hence $(\gamma p+B^2_{\theta})\Big[k\eta-\frac{1}{r}\big((r\xi)'-2\xi\big)\Big]+	(\frac{4\tilde{\varepsilon}}{3}+\tilde{\delta})\Big(\xi'+\frac{\xi}{r}\Big)-	(\frac{\tilde{\varepsilon}}{3}+\tilde{\delta})k\eta$, $\tilde{\varepsilon}\eta'+\frac{\tilde{\varepsilon}}{r}\eta$ and 
 		$\tilde{\varepsilon}\zeta'+\tilde{\varepsilon}\frac{\zeta}{r}$ are well-defined at the endpoint $r=r_0$. 
	Make variations with respect to $(\xi_0,\eta_0,\zeta_0)\in C_c^{\infty}((0,r_0])\times C_c^{\infty}((0,r_0])\times C_c^{\infty}((0,r_0])$. Integrating the terms in \eqref{weak-form}-\eqref{zeta-weak}  by parts and using that $(\xi,\eta,\zeta)$ solve the equations \eqref{new-equations} on $(0,r_0)$, we get that 
	\begin{equation}\label{new-boundary}
	\begin{split}
	&-B^2_\theta\Big[k\eta-\frac{1}{r}\Big((r\xi)'-2\xi\Big)\Big](r\xi_0)\Big|_{r=r_0}+\gamma p\Big[\frac{1}{r}(r\xi)'-k\eta \Big](r\xi_0)\Big|_{r=r_0}\\
	&\Big[-\frac{2\tilde{\varepsilon}}{3}\xi\xi_0+\tilde{\delta}\xi\xi_0+\Big(\frac{4\tilde{\varepsilon}}{3}+\tilde{\delta}\Big)\xi'\xi_0r+\Big(\frac{2\tilde{\varepsilon}}{3}-\tilde{\delta}\Big)k\eta \xi_0r\Big] \Big|_{r=r_0}=0,\\
	&\Big(\tilde{\varepsilon}k\xi \eta_0 r+\tilde{\varepsilon}\eta'\eta_0 r\Big)\Big|_{r=r_0}=0,
	\,\,\, \Big(\tilde{\varepsilon}\zeta'\zeta_0 r-\tilde{\varepsilon}\zeta\zeta_0\Big)\Big|_{r=r_0}=0.
	\end{split}
	\end{equation}
 Since the test functions $\xi_0$, $\eta_0 $ and $\zeta_0$ may be chosen arbitrarily,  by the pressure $p=0$ on the boundary $r=r_0$, we get the interface boundary conditions
	\begin{equation}\label{new-boiundary-3}
	\begin{split}
	&-B^2_\theta\big[k\eta r-\xi'r+\xi \big]\Big|_{r=r_0}+\Big[\Big(\tilde{\delta}-\frac{2\tilde{\varepsilon}}{3}\Big)\big(\xi-k\eta r+\xi'r\big)+2\tilde{\varepsilon}\xi'r\Big] \Big|_{r=r_0}=0,\\
	&\Big(\tilde{\varepsilon}k\xi  r+\tilde{\varepsilon}\eta' r\Big)\Big|_{r=r_0}=0,
	\,\,\, \Big(\tilde{\varepsilon}\zeta'r-\tilde{\varepsilon}\zeta\Big)\Big|_{r=r_0}=0,
	\end{split}
	\end{equation}
which insures that \eqref{inter-two-com} holds. We finish the proof.
\end{proof}

Next, we establish the continuity and monotonicity properties of the eigenvalue $\lambda(s)$.
\begin{prop}
Let $\lambda: (0,\infty)\rightarrow \mathbb{R}$ be given by \eqref{con-mono}. Then the following hold:\\
\mbox{(i)} $\lambda \in C_{loc}^{0,1}((0,\infty))$, and $\lambda\in C^0((0,\infty))$.\\
\mbox{(ii)} There exists a positive constant $C_2=C_2(r_0,p,\varepsilon,\delta,k)$ so that
\begin{equation}\label{decomp}
\lambda(s)\geq -C+sC_2.
\end{equation}\\
\mbox{(iii)} $\lambda(s) $ is strictly increasing.
\end{prop}
 \begin{proof}
Fix a compact interval $Q=[a,b] \subset\subset (0,\infty)$, and fix $(\xi_0,\eta_0,\zeta_0)\in \mathcal{A}_1$. We can decompose $E_{0,k}$ as follows
\begin{equation}\label{decomp-2}
E_{0,k}(\xi,\eta,\zeta;s)=E_{0,k}^0(\xi,\eta,\zeta)+sE_{0,k}^1(\xi,\eta,\zeta),
\end{equation}
with \begin{equation}\label{e0}
\begin{split}
E_{0,k}^0(\xi,\eta,\zeta)&=2\pi^2\int_0^{r_0}\Big\{\Big[\frac{2p'}{r}+\frac{4\gamma p B_{\theta}^2}{r^2(\gamma p+B_{\theta}^2)}\Big]\xi^2\\
&\quad+(\gamma p+B_{\theta}^2)\Big[k\eta-\frac{1}{r}\Big((r\xi)'-\frac{2B_{\theta}^2}{\gamma p +B_{\theta}^2}\xi\Big)\Big]^2\Big\}rdr,
\end{split}
\end{equation}
 \begin{equation}\label{e1}
\begin{split}
E_{0,k}^1(\xi,\eta,\zeta)&=2\pi^2\int_0^{r_0}\varepsilon\Big[\frac{2}{9}\Big(-2\xi'+\frac{\xi}{r}-k\eta\Big)^2+\frac{2}{9}\Big(\xi'-\frac{2\xi}{r}-k\eta\Big)^2\\
&\quad+\frac{2}{9}\Big(\xi'+\frac{\xi}{r}+2k\eta\Big)^2+\Big(-\zeta'+\frac{\zeta}{r}\Big)^2+(\eta'+k\xi)^2+k^2\zeta^2\Big]rdr
\\
&\quad+2\pi^2\int_0^{r_0}\delta\Big(\xi'+\frac{\xi}{r}-k\eta\Big)^2rdr.
\end{split}
\end{equation}
The nonnegativity of $E_{0,k}^1$ implies that $E_{0,k}$ is non-decreasing in $s$ with $(\xi,\eta,\zeta)\in \mathcal{A}_1$ kept fixed. 
By Proposition \ref{infimum-A}, for each $s\in (0,\infty) $ we can find  $(\xi,\eta,\zeta)\in \mathcal{A}_1$  so that 
\begin{equation*}
E_{0,k}(\xi_s,\eta_s,\zeta_s;s)=\inf_{(\xi,\eta,\zeta)\in \mathcal{A}_1}E_{0,k}(\xi,\eta,\zeta;s)=\lambda(s).
\end{equation*}
From the nonnegativity of $E_{0,k}^1$, the minimality of $(\xi_s,\eta_s, \zeta_s)$ and for $0<s_0<r_0$
\begin{equation}\label{E^0_0-k}
\begin{split}
E^0_{0,k}(\xi,\eta,\zeta)\geq
\pi^2\int_{s_0}^{r_0}\gamma p\Big|\frac{1}{r}(r\xi)'\Big|^2rdr-3C\mathcal{J}\geq -3C\mathcal{J},
\end{split}
\end{equation} which can be proved from Lemma 3.10 of Bian-Guo-Tice \cite{bian-guo-tice-inviscid}, using Lemma \ref{presure-boundary}, we have
\begin{equation*}
E_{0,k}(\xi_0,\eta_0,\zeta_0;b)\geq E_{0,k}(\xi_0,\eta_0,\zeta_0;s)\geq E_{0,k}(\xi_s,\eta_s,\zeta_s;s)\geq sE_{0,k}^1(\xi_s,\eta_s,\zeta_s)-C,
\end{equation*}
for all $s\in Q$. This implies that there exists a constant $$0<K=K(a,b,\xi_0,\eta_0,\zeta_0, \pi,  p)<\infty$$ so that 
\begin{equation}\label{bound-e1}
	\sup_{s\in Q}E_{0,k}^1(\xi_s,\eta_s,\zeta_s)\leq K.
\end{equation}
Let $s_i\in Q$ for $i=1,2$. Using the minimality of $(\xi_{s_1},\eta_{s_1}, \zeta_{s_1}) $ compared to $(\xi_{s_2},\eta_{s_2}, \zeta_{s_2}) $, we have 
\begin{equation}\label{lam-s1}
	\lambda(s_1)=E_{0,k}(\xi_{s_1},\eta_{s_1},\zeta_{s_1};s_1)\leq E_{0,k}(\xi_{s_2},\eta_{s_2},\zeta_{s_2};s_1),
\end{equation}
which together with \eqref{decomp-2} gives that
\begin{equation}\label{lam-s2}
\begin{split}
E_{0,k}(\xi_{s_2},\eta_{s_2},\zeta_{s_2};s_1)&\leq E_{0,k}(\xi_{s_2},\eta_{s_2},\zeta_{s_2};s_2)+|s_1-s_2|E_{0,k}^1(\xi_{s_2},\eta_{s_2},\zeta_{s_2})\\
&=\lambda(s_2)+|s_1-s_2|E_{0,k}^1(\xi_{s_2},\eta_{s_2},\zeta_{s_2}).
\end{split}
\end{equation}
Combining \eqref{bound-e1}, \eqref{lam-s1} and \eqref{lam-s2}, we get that
$\lambda(s_1)\leq \lambda(s_2)+K|s_1-s_2|$,
which implies that (i) holds.

Now, we prove (ii). Note that \eqref{E^0_0-k} 
 and the nonnegativity of $E_{0,k}^1$
 imply that 
$$\lambda(s)\geq s\inf_{(\xi,\eta,\zeta)\in \mathcal{A}}E_{0,k}^1(\xi,\eta,\zeta)-C,$$
where we denote the constant $C_2=\inf_{(\xi,\eta,\zeta)\in \mathcal{A}}E_{0,k}^1(\xi,\eta,\zeta)$ and this canstant is positive. 

Finally, we show (iii). Notice that if $0<s_1<s_2<\infty$, then
the decomposition \eqref{decomp-2} ensures that
\begin{equation*}
\begin{split}
\lambda(s_1)&=E_{0,k}(\xi_{s_1},\eta_{s_1},\zeta_{s_1};s_1)\leq E_{0,k}(\xi_{s_2},\eta_{s_2},\zeta_{s_2};s_1)\\
&\leq E_{0,k}(\xi_{s_2},\eta_{s_2},\zeta_{s_2};s_2)=\lambda(s_2).
\end{split}
\end{equation*}So $\lambda $ is non-decreasing in $s$. Suppose by way of contradiction that
$\lambda(s_1)=\lambda(s_2)$. Then, the above inequality implies that 
$s_1E_{0,k}^1(\xi_{s_2},\eta_{s_2},\zeta_{s_2})=s_2E_{0,k}^1(\xi_{s_2},\eta_{s_2},\zeta_{s_2})$,
which gives that
$E_{0,k}^1(\xi_{s_2},\eta_{s_2},\zeta_{s_2})=0$. This in turn implies $\xi_{s_2}=\eta_{s_2}=\zeta_{s_2}=0$, which contradicts that $(\xi_{s_2},\eta_{s_2},\zeta_{s_2})\in \mathcal{A}_1$. Therefore, the equality can not be achieved, and $\lambda$ is strictly increasing in $s$.
\end{proof} 

\begin{rmk}\label{regularity-u}
Define the open set $\mathcal{S}=\lambda^{-1}(-\infty,0) \subset(0,\infty)$, then we calculate $\mu=\sqrt{-\lambda}>0$. The open set $\mathcal{S}$ is nonempty by Proposition \ref{infimum-A-3}.
\end{rmk}
From Proposition \ref{infimum-A-3} and lower bound \eqref{decomp}, we can show the uniqueness of $s$. 
\begin{prop}\label{prop3.7}
There exists a unique $s\in \mathcal{S}$ so that $\mu(s)=\sqrt{-\lambda}>0$ and 
	$s=\mu(s)$.
\end{prop}
\begin{proof}
	From Proposition \ref{infimum-A-3}, we know that there exist two constants $C_0>0$ and $C_1>0$ such that $\lambda(s)\leq -C_0+sC_1$. On the other hand, the lower bound \eqref{decomp} implies that $\lambda(s)\rightarrow +\infty$, as $s\rightarrow \infty$. Since $\lambda$ is continuous and strictly increasing, there exists $s^* \in (0,\infty)$ so that 
	$$\mathcal{S}=\lambda^{-1}((-\infty,0))=(0,s^*).$$ Since $\lambda<0$, on $\mathcal{S}$, we can define $\mu=\sqrt{-\lambda(s)}$. Now, we define the function $\Phi: (0,s^*)\rightarrow(0,\infty)$ according to $\Phi(s)=\frac{s}{\mu(s)}$. So $\Phi$ is continous and strictly increasing in $s$ from the continuity and monotonicity properties of $\lambda$. Moreover, we know that $\lim_{s\rightarrow 0}\Phi(s)=0$ and $\lim_{s\rightarrow s^*}\Phi(s)=+\infty$. By the intermediate value theorem, there exists $s\in (0,s^*)$ so that $\Phi(s)=1$, that is, $s=\mu(s)$. This $s$ is unique since $\Phi
	$ is strictly increasing. 
\end{proof}
Now, we consider the regularity of $\xi$, $\eta$ and $\zeta$.
\begin{prop}
Let $(\xi,\eta,\zeta)$ be the solutions to \eqref{spectal formulation-2} or \eqref{new-equations}, then there exists $A_n$ such that
\begin{equation*}
\|\xi^{(n)}r^{(2n-1)/2}\|_{L^2(0,r_0)}\leq A_n,\,\,
\|\eta^{(n)}r^{(2n-1)/2}\|_{L^2(0,r_0)}\leq A_n,
\,\,
\|\zeta^{(n)}r^{(2n-1)/2}\|_{L^2(0,r_0)}\leq A_n,
\end{equation*}
with $A_n$ depending on $r_0$, $k$, $\pi$ and the presure $p$.
\end{prop}
\begin{proof}
Applying $(\xi,\eta,\zeta)\in X_k$ and the compactness results in Proposition  \ref{embeddding-b},  when $m=0$ and any $k\in \mathbb{Z}$, we can get that $(\xi,\eta, \zeta)\in L^2(0,r_0)$ and $\Big(\sqrt{r}\xi',\sqrt{r}\eta',\sqrt{r}\zeta',\frac{\xi}{\sqrt{r}},\frac{\zeta}{\sqrt{r}}\Big)\in L^2(0,r_0)$. From the system \eqref{new-equations}, we have
\begin{equation*}
\|\xi''r^{3/2}\|_{L^2(0,r_0)}\leq C\Big\|\Big(\xi,\eta,\sqrt{r}\xi',\sqrt{r}\eta',\frac{\xi}{\sqrt{r}}\Big)\Big\|_{L^2(0,r_0)},
\end{equation*}
\begin{equation*}
\|\eta''r^{3/2}\|_{L^2(0,r_0)}\leq  C\Big\|\Big(\xi,\eta,\sqrt{r}\xi',\sqrt{r}\eta',\frac{\xi}{\sqrt{r}}\Big)\Big\|_{L^2(0,r_0)},
\end{equation*}
\begin{equation*}
\Big\|\zeta''r^{3/2}\|_{L^2(0,r_0)}\leq C\|\Big(\frac{\zeta}{\sqrt{r}},\sqrt{r}\zeta'\Big)\|_{L^2(0,r_0)}.
\end{equation*}
By induction on $n$. Suppose for some $n\geq 1$,
\begin{equation*}
\|\xi^{(n)}r^{(2n-1)/2}\|_{L^2(0,r_0)}\leq A_n,\,\,
\|\eta^{(n)}r^{(2n-1)/2}\|_{L^2(0,r_0)}\leq A_n,\,\,
\|\zeta^{(n)}r^{(2n-1)/2}\|_{L^2(0,r_0)}\leq A_n.
\end{equation*}
By Remark \ref{regularity-u}, then differentiating \eqref{new-equations}, we get that there exists a constant $C$ depending on the various parameters so that
\begin{equation*}
\|\xi^{(n+1)}r^{(2n+1)/2}\|_{L^2(0,r_0)}\leq CA_n\leq A_{n+1},
\end{equation*}
\begin{equation*}
\|\eta^{(n+1)}r^{(2n+1)/2}\|_{L^2(0,r_0)}\leq CA_n\leq A_{n+1},
\end{equation*}
\begin{equation*}
\|\zeta^{(n+1)}r^{(2n+1)/2}\|_{L^2(0,r_0)}\leq CA_n\leq  A_{n+1}.
\end{equation*}
Then, the bound holds for $n+1$, and so by induction the bound holds for all $n\geq 1$.
\end{proof}

\subsection{Modified variational problem when $m\neq 0$}
In this subsection, we will first introduce the definition of function space $Y_{m,k}$ and its properties, then prove the instability for any $m\neq 0$ and  any $k\in \mathbb{Z}$. 

First, let us introduce the definition of the space $Y_{m,k}$.
\begin{defi}\label{defi-y}
	The weighted Sobolev space $Y_{m,k}$ is defined as the completion of $\Big\{(\xi,\eta,\zeta)\in C^{\infty}([0,r_0])\times C^{\infty}([0,r_0])\times C^{\infty}([0,r_0])\Big|\xi(0)=\eta(0)=\zeta(0)=0\Big\}$, with respect to the norm
\begin{equation}\label{defi-ymk}
\begin{split}
\|\big(\xi,\eta,\zeta\big)\|^2_{Y_{m,k}}&
=\int_0^{r_0}\Big[\Big(-2\xi'+\frac{\xi}{r}+\frac{m}{r}\zeta-k\eta\Big)^2+\Big(\xi'-\frac{2\xi}{r}-\frac{2m}{r}\zeta-k\eta\Big)^2\\
&\quad+\Big(\xi'+\frac{\xi}{r}+\frac{m}{r}\zeta+2k\eta\Big)^2+\Big(-\zeta'+\frac{\zeta}{r}+\frac{m}{r}\xi\Big)^2+(\eta'+k\xi)^2\\
&\quad+\Big(\frac{m}{r}\eta-k\zeta\Big)^2\Big]rdr
+\int_0^{r_0}\Big(\xi'+\frac{\xi}{r}+\frac{m}{r}\zeta-k\eta\Big)^2rdr\\
&\quad+\int_0^{r_0}\rho(\xi^2+\eta^2+\zeta^2)rdr.
\end{split}
\end{equation} 
\end{defi}

Applying Definition \ref{defi-y}, Lemma \ref{xi-bound} and Lemma \ref{xi-bound-m-general}, similarly as Proposition \ref{embeddding-b}, we can get the following compactness results.
\begin{prop}\label{embeddding-m1}
Let $\pi_i$ for $ = 1, 2,3$ denote the projection operator onto the $i$-th factor. Then $\pi_i: Y_{m,k}\rightarrow Z_i$ is a bounded, linear, compact map for $ i = 1, 2,3$,
with the spaces $Z_1$, $Z_2$ and $Z_3$ in \eqref{defi-Z}. We denote them by 	
$Y_{m,k}\subset\subset Z_i$ for $i=1,2,3$.
\end{prop}

Now, we consider the case $m\neq 0$ and any $k\in \mathbb{Z}$. 
		We need to consider the energy \eqref{va-m-1-b} and
\eqref{constraint}.
	
First, we can show that  $E_{m,k}$ and $\mathcal{J}$ are well defined on the space 
$Y_{m,k} \times H^1(r_0,r_w)$.
\begin{lem}
$E_{m,k}$ and $\mathcal{J}$ are well defined on the space $Y_{m,k} \times H^1(r_0,r_w)$.	
\end{lem}
\begin{proof}
	By Lemma \ref{xi-bound} and Proposition \ref{embeddding-m1}, similar to the proof of \eqref{pressure-term-control}, we also have by Definition \ref{defi-y} that 
	\begin{equation}
	\bigg| \int_0^{r_0}p'\xi^2dr\bigg|\leq  C\mathcal{J}+C\|\big(\xi,\eta,\zeta\big)\|^2_{Y_{m,k}},
	\end{equation}
	which  implies that
	\begin{equation*}
\begin{split}
|E(\xi,\eta,\zeta,\widehat{ Q}_r)|&\leq C\mathcal{J}++C\|\big(\xi,\eta,\zeta\big)\|^2_{Y_{m,k}}+ C\int_0^{r_0}B_{\theta}^2\Big[\frac{\eta}{r}-\frac{k}{m^2+k^2r^2}\big((r\xi)'-2\xi\big)\Big]^2rdr\\
&\quad+C\int_0^{r_0}p\Big[\frac{1}{r}(r\xi)'+\frac{m\zeta}{r}\Big]^2rdr+C\int_0^{r_0}\frac{B^2_\theta}{r}\xi^2dr+C\int_0^{r_0}\frac{B^2_\theta}{r}(\xi-r\xi')^2dr\\
&\quad+C\int_0^{r_0}\Big[\Big(-2\xi'+\frac{\xi}{r}+\frac{m}{r}\zeta-k\eta\Big)^2+\Big(\xi'-\frac{2\xi}{r}-\frac{2m}{r}\zeta-k\eta\Big)^2\\
&\quad+\Big(\xi'+\frac{\xi}{r}+\frac{m}{r}\zeta+2k\eta\Big)^2+\Big(-\zeta'+\frac{\zeta}{r}+\frac{m}{r}\xi\Big)^2+(\eta'+k\xi)^2+(\frac{m}{r}\eta-k\zeta)^2\Big]rdr
\\
&\quad+C\int_0^{r_0}\Big(\xi'+\frac{\xi}{r}+\frac{m}{r}\zeta-k\eta\Big)^2rdr+ C\|\widehat{ Q}_r\|_{H^1(r_0,r_w)}+ C\|\rho\|^{\gamma-1}_{L^\infty}\int_0^{r_0} \rho |\eta|^2rdr\\
&\leq C\|\big(\xi,\eta,\zeta\big)\|^2_{Y_{m,k}}+ C\|\widehat{ Q}_r\|_{H^1(r_0,r_w)}.
\end{split}
\end{equation*}
Therefore, $E_{m,k}$ and $\mathcal{J}$ are well defined on the space $Y_{m,k}\times H^1(r_0,r_w)$.
\end{proof}
	Consider the set 
	\begin{equation}\label{set a-e-3-viscosity}
	\mathcal{A}_2=\left\{
	\begin{aligned}
	\big((\xi,\eta,\zeta),\widehat{ Q}_r\big
	)\in Y_{m,k}\times H^1(r_0,r_w)| \mathcal{J}(\xi,\eta,\zeta)=1,\,\\ m\widehat{  B}_\theta \xi=r\widehat{ Q}_r\, \,\mbox{at}\, \,r=r_0\,\, \mbox{and}\,\, \widehat{ Q}_r=0\, \,\mbox{at}\, \,r=r_w
	\end{aligned}
	\right\},
	\end{equation}
	where the functions $\xi$, $\eta$ and $\zeta$ are restricted to $(0,r_0)$, and the function $\widehat{ Q}_r$ is restricted to $(r_0,r_w)$.	
	Then we write  \begin{equation}\label{e-3-viscosity}
	\lambda(s):=\inf_{((\xi,\eta,\zeta),\widehat{ Q}_r)\in \mathcal{A}_2}E_{m,k}(\xi,\eta,\zeta,\widehat{ Q}_r;s). 
	\end{equation}
		We want to show that the infimum of $E_{m,k}(\xi,\eta,\zeta,\widehat{ Q}_r)$ over the set $\mathcal{A}_2$ is achieved and is negative and that the minimizer solves \eqref{spectral-formulation} and \eqref{euler-l-q} with the corresponding boundary conditions. First, we prove that the energy $E_{m,k}$ has a lower bound on the set $ \mathcal{A}_2$ and the coercivity estimate holds.
\begin{lem}\label{lower-bound-general-m}
The energy $E_{m,k}(\xi,\eta,\zeta,\widehat{ Q}_r)$ has a lower bound on the set $ \mathcal{A}_2$ and any minimizing sequence is bounded in $Y_{m,k}$.
\end{lem}
\begin{proof}
Similarly as \eqref{con-mono-use-bian}, we can get from  \eqref{va-m-1-b} that for any $((\xi,\eta,\zeta),\widehat{ Q}_r)\in \mathcal{A}_2$,
\begin{equation}\label{con-mono-use-general-m}
\begin{split}
E_{m,k}(\xi,\eta,\zeta,\widehat{ Q}_r)
\geq 2\pi^2\min(\tilde{\varepsilon},\tilde{\delta})\int_0^{r_0}{\xi'}^2rdr+4\pi^2\int_0^{r_0}p'\xi^2dr.
\end{split}
\end{equation}
By Lemma \ref{xi-bound}, choosing $0<s_0<\frac{r_0}{3}$ small enough such that
$Cs_0\leq \frac{1}{4}$, similarly as \eqref{estimate-p'}  we deduce that
\begin{equation*}
\Big|\int_0^{s_0}p'\xi^2dr\Big|
\leq C\mathcal{J}s_0
+\frac{1}{2}\pi^2\min(\tilde{\varepsilon},\tilde{\delta})\int_0^{r_0}{\xi'}^2rdr.
\end{equation*}	
On the other hand, from the Definition of $\mathcal{J}$ and Lemma \ref{xi-bound-m-general},   choosing  $\frac{r_0}{3}<s_1<r_0$ close enough to $r_0$ such that
 $C(r_0-s_1)\leq \frac{1}{4}$, we can prove 
\begin{equation*}
\Big|\int_{s_0}^{s_1}p'\xi^2dr\Big|\leq C\int_{s_0}^{s_1} \xi^2dr\leq C\mathcal{J},
\end{equation*}
and
\begin{equation}\label{xi-L^2-boundary}
\begin{split}
\int_{s_1}^{r_0}\xi^2(r)dr&\leq	C\int_{s_1}^{r_0}\Big(\int_{\frac{r_0}{3}}^{\frac{r_0}{2}}|\xi(b)|^2db\Big)dr
+C\int_{s_1}^{r_0}\Big(\int_{\frac{r_0}{3}}^{r}|\xi'|^2|r_0-s|ds\Big)\\
&\quad\times
\Big|\ln\frac{2r_0}{3}-\ln(r_0-r)\Big|dr\\
&
\leq C\mathcal{J}(r_0-s_1)+C\Big(\int_{\frac{r_0}{3}}^{r_0}|\xi'|^2|r_0-s|ds\Big)\int_{s_1}^{r_0}\Big|\ln \frac{2r_0}{3}-\ln(r_0-r)\Big|dr
\\
&\leq C\mathcal{J}(r_0-s_1)
+C(r_0-s_1)\int_{\frac{r_0}{3}}^{r_0}|\xi'|^2|r_0-s|ds\\
&\leq C\mathcal{J}(r_0-s_1)
+2C(r_0-s_1)\pi^2\min(\tilde{\varepsilon},\tilde{\delta})\int_0^{r_0}{\xi'}^2rdr\\
&\leq C\mathcal{J}(r_0-s_1)
+\frac{1}{2}\pi^2\min(\tilde{\varepsilon},\tilde{\delta})\int_0^{r_0}{\xi'}^2rdr,
\end{split}
\end{equation}	where we have used the facts  \begin{equation}
\begin{split}
\int_{\frac{r_0}{3}}^{r_0}\big|\xi'\big|^2|r_0-s|ds&\leq \int_{\frac{r_0}{3}}^{r_0}\big|\xi'\big|^2rdr+\int_{\frac{r_0}{3}}^{r_0}\big|\xi'\big|^2r_0dr
\leq 2C\pi^2\min(\tilde{\varepsilon},\tilde{\delta})\int_{\frac{r_0}{3}}^{r_0}{\xi'}^2rdr\\
&\leq 2C\pi^2\min(\tilde{\varepsilon},\tilde{\delta})\int_0^{r_0}{\xi'}^2rdr.
\end{split}
\end{equation}	
Hence, we can get
\begin{equation}
\begin{split}
E_{m,k}(\xi,\eta,\zeta,\widehat{ Q}_r)
&\geq\pi^2\min(\tilde{\varepsilon},\tilde{\delta})\int_0^{r_0}{\xi'}^2rdr-3C\mathcal{J}\geq -3C\mathcal{J},
\end{split}
\end{equation}	
	which implies that the energy $E_{m,k}(\xi,\eta,\zeta,\widehat{ Q}_r)$ has a lower bound on the set $ \mathcal{A}_2$.
  Now we prove the coercivity estimate.
 	Using the facts that	$\mathcal{J}=1$ and $E_{m,k}$ has a lower bound on the set $\mathcal{A}_2$, we can choose a minimizing sequence such that along the minimizing sequence, 
 	we have $M\leq E_{m,k}(\xi_n,\eta_n,\zeta_n,\widehat{  Q}_{rn})<M+1$, and 
 		for the minimizing sequence,
 	we can prove coercivity estimate: 
	\begin{equation*}\label{coercivity-m-not-zero}
 	\begin{split}
 	\|\big(\xi_n,\eta_n,\zeta_n\big)\|^2_{Y_{m,k}}+\|\widehat{ Q}_{rn}\|_{H^1(r_0,r_w)}\leq C\mathcal{J}+C(M+1)\leq C.
 	\end{split}
 	\end{equation*}
\end{proof}
From Lemma \ref{lower-bound-general-m} and Proposition \ref{embeddding-m1}, we can show that $E_{m,k}$ achieves its infimum on the set $\mathcal{A}_2$.
\begin{prop}\label{infimum-A-out-viscosity}
 $E_{m,k}$ achieves its infimum on the set $\mathcal{A}_2$.
	\end{prop}
	\begin{proof}
		First from Lemma \ref{lower-bound-general-m}, we have that $E_{m,k}$ is bounded below on the set $\mathcal{A}_2$. Let $((\xi_n,\eta_n,\zeta_n),\widehat{ Q}_{rn}) \in \mathcal{A}_2$ be a minimizing sequence. Then $(\xi_n,\eta_n,\zeta_n)$ are bounded in $Y_{m,k}$ (see \eqref{defi-ymk}),  and $ \widehat{ Q}_{rn}$ is bounded in $H^1(r_0,r_w)$, so up to the extraction of a subsequence $\psi_n=\sqrt{m^2+k^2r^2}|B_{\theta}|\Big[\frac{\eta_n}{r}-\frac{k}{m^2+k^2r^2}((r\xi_n)'-2\xi_n)\Big]r^{\frac12}\rightharpoonup \psi=\sqrt{m^2+k^2r^2}|B_{\theta}|\Big[\frac{\eta}{r}-\frac{k}{m^2+k^2r^2}((r\xi)'-2\xi)\Big]r^{\frac12}$ weakly in $L^2$, and $\xi_n\rightarrow \xi$, $\eta_n\rightarrow \eta$ strongly in $L^2$ from the compact embeddings in Proposition \ref{embeddding-m1}.
			By weak lower semi-continuity, since $\psi_n\rightharpoonup \psi$ in the space $L^2(0,r_0)$,
			we get
			\begin{equation*}
			\begin{split}
			&\int_0^{r_0}(m^2+k^2r^2)B^2_{\theta}\Big[\frac{\eta}{r}-\frac{k}{m^2+k^2r^2}((r\xi)'-2\xi)\Big]^2rdr\\
			&\quad\leq\liminf_{n\to\infty}\int_0^{r_0}(m^2+k^2r^2)B^2_{\theta}\Big[\frac{\eta_n}{r}-\frac{k}{m^2+k^2r^2}((r\xi_n)'-2\xi_n)\Big]^2rdr.
			\end{split}
			\end{equation*}
			
		Because of the quadratic structure of all the terms in the integrals defining $E_{m,k}$, similarly dealing with the other positive terms by weak lower semi-continuity, strong $L^2(0,r_0)$ convergence in Proposition \ref{embeddding-m1} and $\widehat{ Q}_{nr}\rightarrow \widehat{ Q}_r$ strongly in $L^2(r_0,r_w)$, we deduce
		\begin{equation*}
		E_{m,k}(\xi,\eta,\zeta,\widehat{ Q}_{r})\leq \liminf_{n\to\infty}E_{m,k}(\xi_n,\eta_n,\zeta_n,\widehat{ Q}_{rn})=\inf_{\mathcal{A}_2}E_{m,k}.
		\end{equation*}

	All that remains is to show that $((\xi,\eta,\zeta),\widehat{ Q}_r)\in \mathcal{A} $.
	The fact that	$\xi_n\rightarrow \xi$, $\eta_n\rightarrow \eta$, $\zeta_n\rightarrow \zeta$ strongly in $L^2$ from the compact embeddings in Proposition \ref{embeddding-m1}, implies that $\mathcal{J}(\xi, \eta,\zeta)=1$, so that $((\xi,\eta,\zeta),\widehat{ Q}_r)\in \mathcal{A}_2 $.	We complete the proof.
	\end{proof}

We know that $\tilde{\varepsilon}=s\varepsilon$ and $\tilde{\delta}=s\delta$ are smooth, and bounded from above and below by positive quantities for fixed $s>0$.
Now, we prove the infimum of $E_{m,k}$ over the set $\mathcal{A}_2$ is negative, if there exists $r^*$ such that $2p'(r^*)+\frac{m^2B_\theta^2(r^*)}{r^*}<0$ for $m\neq 0$.
\begin{prop}\label{infimum-A-out-3}
If there exists $r^*$ such that $2p'(r^*)+\frac{m^2B_\theta^2(r^*)}{r^*}<0$ for $m\neq 0$,  then for any fixed large  $k=k_0$ and any fixed $m\neq 0$, there exists constant $s_0>0$ depending on $\varepsilon$, $\delta$ and $k_0$ so that for $s\leq s_0$,  it holds that $\lambda=\inf E_{m,k}<0$. 
\end{prop}
\begin{proof}
	Since both $E_{m,k}$ and $\mathcal{J}$ are homogeneous degree $2$, it suffices to show that 
	\begin{equation*}
	\inf_{((\xi,\eta,\zeta),\widehat{ Q}_r))\in Y_{m,k}\times  H^1(r_0,r_w)}\frac{E_{m,k}(\xi,\eta,\zeta,\widehat{ Q}_r)}{\mathcal{J}(\xi,\eta,\zeta)}<0.
	\end{equation*}
	But since $\mathcal{J}$ is positive definite, one may reduce to constructing  $((\xi,\eta,\zeta),\widehat{ Q}_r)\in Y_{m,k}\times H^1(r_0,r_w)$ (see \eqref{defi-ymk}), such that $$E_{m,k}(\xi, \eta, \zeta, \widehat{ Q}_r)<0.$$

From the assumption in Proposition \ref{infimum-A-out-3}, we can choose a smooth function $\xi^*\in C_c^{\infty}(0,r_0)$ such that $$2\pi^2\int_0^{r_0}\Big[2p'+\frac{m^2B_{\theta}^2}{r}\Big]{\xi^*}^2dr< 0.$$
	Then, 
	we can assume that $\eta^*=\frac{rk(r\xi^*)'-2rk\xi^*}{m^2+k^2r^2}$ and $\zeta^*=\frac{r}{m}(k\eta^*-\frac{1}{r}(r\xi^*)')$, such that the first and second terms of $E_{m,k}(\xi^*,\eta^*,\zeta^*,\widehat{ Q}^*_r)$ in  \eqref{va-m-1-b} vanish,
	that is,
	\begin{equation*}
	\begin{split}
	&	2\pi^2\int_0^{r_0}(m^2+k^2r^2)\Big[\frac{B_\theta}{r}\eta^*+\frac{-kB_\theta(r\xi^*)'+2kB_{\theta}\xi^*}{m^2+k^2r^2}\Big]^2rdr=0,\\
	&
	2\pi^2\int_0^{r_0}\gamma p\Big[\frac{1}{r}(r\xi^*)'-k\eta^*+\frac{m\zeta^*}{r}\Big]^2rdr=0.
	\end{split}
	\end{equation*} Here, $\xi^*$, $\eta^*$ and $\zeta^*$ are smooth functions and belong to the space $Y_{m,k}$. 
	
Fix $k=k_0$ large enough and fix $m=m_0$, and choose $\widehat{ Q}^*_r=0$ so that
	\begin{equation*}
	\begin{split}
	&2\pi^2\int_0^{r_0}\frac{m^2B_\theta^2}{r(m^2+k^2r^2)}(\xi^*-r{\xi^*}')^2dr\leq -\frac{1}{2}\pi^2\int_0^{r_0}\Big[2p'+\frac{m^2B_{\theta}^2}{r}\Big]{\xi^*}^2dr,\\
	&2\pi^2\int_{r_0}^{r_w}\bigg[|\widehat{Q}^*_r|^2+\frac{1}{m^2+k^2r^2}|(r\widehat{Q}^*_r)'|^2\bigg]r
	dr=0,
	\end{split}
	\end{equation*}	
	and from $\xi^*\in C_c^{\infty}(0,r_0)$,  $\eta^*=\frac{rk(r\xi^*)'-2rk\xi^*}{m^2+k^2r^2}$ and $\zeta^*=\frac{r}{m}(k\eta^*-\frac{1}{r}(r\xi^*)')$, we know that
	\begin{equation*}
	\begin{split}
&2\pi^2\varepsilon\int_0^{r_0}\Big[\frac{2}{9}\Big(-2{\xi^*}'+\frac{\xi^*}{r}+\frac{m}{r}\zeta^*-k\eta^*\Big)^2+\frac{2}{9}\Big({\xi^*}'-\frac{2\xi^*}{r}-\frac{2m}{r}\zeta^*-k\eta^*\Big)^2\\
&\qquad+\frac{2}{9}\Big({\xi^*}'+\frac{\xi^*}{r}+\frac{m}{r}\zeta^*+2k\eta^*\Big)^2
+\Big(-{\zeta^*}'+\frac{\zeta^*}{r}+\frac{m}{r}\xi^*\Big)^2\\
&\qquad+({\eta^*}'+k\xi^*)^2+\Big(\frac{m}{r}\eta^*-k\zeta^*\Big)^2\Big]rdr
+2\pi^2\delta\int_0^{r_0}\Big({\xi^*}'+\frac{\xi^*}{r}+\frac{m}{r}\zeta^*-k\eta^*\Big)^2rdr\leq C.
	\end{split}
	\end{equation*}
	Therefore, we get  for $k=k_0$ and $m=m_0$  that 
	\begin{equation*}\label{negative-g}
	\begin{split}
	&E_{m,k}(\xi^*,\eta^*,\zeta^*,\widehat{ Q}^*_r)\\&=2\pi^2\int_0^{r_0}\tilde{\varepsilon}\Big[\frac{2}{9}\Big(-2{\xi^*}'+\frac{\xi^*}{r}+\frac{m}{r}\zeta^*-k\eta^*\Big)^2+\frac{2}{9}\Big({\xi^*}'-\frac{2\xi^*}{r}-\frac{2m}{r}\zeta^*-k\eta^*\Big)^2\\
	&\qquad+\frac{2}{9}\Big({\xi^*}'+\frac{\xi^*}{r}+\frac{m}{r}\zeta^*+2k\eta^*\Big)^2
	+\Big(-{\zeta^*}'+\frac{\zeta^*}{r}+\frac{m}{r}\xi^*\Big)^2\\
	&\qquad+({\eta^*}'+k\xi^*)^2+\Big(\frac{m}{r}\eta^*-k\zeta^*\Big)^2\Big]rdr
	+2\pi^2\int_0^{r_0}\tilde{\delta}\Big({\xi^*}'+\frac{\xi^*}{r}+\frac{m}{r}\zeta^*-k\eta^*\Big)^2rdr\\
	&\qquad+2\pi^2\int_0^{r_0}\frac{m^2B_\theta^2}{r(m^2+k^2r^2)}(\xi^*-r{\xi^*}')^2dr+2\pi^2\int_0^{r_0}\Big[2p'+\frac{m^2B_{\theta}^2}{r}\Big]{\xi^*}^2dr\\
	&\qquad+2\pi^2\int_{r_0}^{r_w}\bigg[|\widehat{Q}^*_r|^2+\frac{1}{m^2+k^2r^2}|(r\widehat{Q}^*_r)'|^2\bigg]r
	dr
	\\	&\leq s C+\frac{3}{2}\pi^2	\int_0^{r_0}\Big[2p'+\frac{m^2B_{\theta}^2}{r}\Big]{\xi^*}^2dr.
	\end{split}
	\end{equation*}		
Then there existes $s_0>0$ depending on $\varepsilon$, $\delta$ and $k_0$, so that $s\leq s_0$,	it holds that
	\begin{equation*}\label{negative-g2}
	\begin{split}
	\widetilde{E}(\xi^*)
	\leq \pi^2\int_0^{r_0}\Big[2p'+\frac{m^2B_{\theta}^2}{r}\Big]{\xi^*}^2dr<0,
	\end{split}
	\end{equation*} 
which implies the result.
\end{proof}
	We now prove that the minimizer constructed in the previous result satisfies Euler-Lagrange equations equivalent to \eqref{spectral-formulation} and
\eqref{euler-l-q} with suitable boundary conditions.		

\begin{prop}\label{infimum-A-out-2-viscosity}
	Let $((\xi,\eta,\zeta),\widehat{ Q}_{r})\in\mathcal{A}_2$ be the minimizer of $E_{m,k}$ constructed in Proposition \ref{infimum-A-out-viscosity}. Then $(\xi,\eta,\zeta)$ are smooth in $(0,r_0)$ and satisfy
	\begin{equation}\label{spectal formulation-viscosity}
	\begin{split}
&\left(
\begin{array}{ccc}
\frac{d}{dr}\frac{\gamma p+B_{\theta}^2}{r}\frac{d}{dr}r-\frac{m^2}{r^2}B_{\theta}^2
-r(\frac{B_{\theta}^2}{r^2})'&-\frac{d}{dr}k(\gamma p+B_{\theta}^2)-\frac{2kB_{\theta}^2}{r}&\frac{d}{dr}\frac{m}{r}\gamma p\\
\frac{k(\gamma p+B_{\theta}^2)}{r}\frac{d}{dr}r-\frac{2kB_{\theta}^2}{r}&
-k^2(\gamma p+B_{\theta}^2)-\frac{m^2}{r^2}B_{\theta}^2&\frac{mk}{r}\gamma p\\
-\frac{m\gamma p}{r^2}\frac{d}{dr}r&\frac{mk}{r}\gamma p&-\frac{m^2}{r^2}\gamma p
\end{array}
\right)
\left(
\begin{array}{lll}
\xi   \\
\eta \\
\zeta\\
\end{array}
\right)\\
&	+\left(
\begin{array}{ccc}
a_{11}&a_{12}&a_{13}\\
a_{21}&
a_{22}&a_{23}\\
a_{31}&a_{32}&a_{33}
\end{array}
\right)
\left(
\begin{array}{lll}
\xi   \\
\eta \\
\zeta\\
\end{array}
\right)	
=-\rho \lambda \left(
\begin{array}{lll}
\xi   \\
\eta \\
\zeta\\
\end{array}
\right),
\end{split}
\end{equation}
with 
\begin{equation*}
\begin{split}
&a_{11}=	(\frac{4\tilde{\varepsilon}}{3}+\tilde{\delta})\frac{d^2}{dr^2}+(\frac{4\tilde{\varepsilon}}{3}+\tilde{\delta})\frac{d}{dr}\frac{1}{r}-\tilde{\varepsilon}\Big(\frac{m^2}{r^2}+k^2\Big),\quad a_{12}=-(\frac{\tilde{\varepsilon}}{3}+\tilde{\delta})k\frac{d}{dr},\\
&a_{13}=(\frac{\tilde{\varepsilon}}{3}+\tilde{\delta})\frac{d}{dr}\frac{m}{r}
-\frac{2\tilde{\varepsilon}m}{r^2},\quad a_{21}=(\frac{\tilde{\varepsilon}}{3}+\tilde{\delta})k\frac{d}{dr}+(\frac{\tilde{\varepsilon}}{3}+\tilde{\delta})\frac{k}{r},\\
&a_{22}=	\tilde{\varepsilon}\frac{d^2}{dr^2}+\frac{\tilde{\varepsilon}}{r}\frac{d}{dr}
-\tilde{\varepsilon}\frac{m^2}{r^2}-(\frac{4\tilde{\varepsilon}}{3}+\tilde{\delta})k^2,\quad
a_{23}=(\frac{\tilde{\varepsilon}}{3}+\tilde{\delta})\frac{mk}{r},\\
&a_{31}=-(\frac{\tilde{\varepsilon}}{3}+\tilde{\delta})\frac{m}{r}\frac{d}{dr}-(\frac{7\tilde{\varepsilon}}{3}+\tilde{\delta})\frac{m}{r^2},\quad a_{32}=(\frac{\tilde{\varepsilon}}{3}+\tilde{\delta})\frac{mk}{r},\\
&a_{33}=\tilde{\varepsilon}\frac{d^2}{dr^2}+\tilde{\varepsilon}\frac{d}{dr}\frac{1}{r}-(\frac{4\tilde{\varepsilon}}{3}+\tilde{\delta})\frac{m^2}{r^2}-\tilde{\varepsilon}k^2.
\end{split}
\end{equation*}
	And the solution $\widehat{ Q}_{r}$ is smooth on $(r_0,r_w)$  and satisfies
	\begin{equation}\label{ode-curl-div-viscosity}
	\bigg[\frac{r}{m^2+k^2r^2}(r\widehat{Q}_r)'\bigg]'-\widehat{Q}_r=0,
	\end{equation} 
		with the other two components $\widehat{ Q}_{\theta}=-\frac{m}{m^2+k^2r^2}(r\widehat{ Q}_r)'$ and
	$\widehat{ Q}_{z}=-\frac{kr}{m^2+k^2r^2}(r\widehat{ Q}_r)'$. 
Moreover, the solution $(\xi,\eta,\zeta)$ and $\widehat{ Q}_{r}$ satisfy the interface boundary conditions 	
	\begin{equation}\label{inter-four-2-viscosity}
	\begin{split}
		&\Big[B_{\theta}^2\xi-B_{\theta}^2\xi' r+kB_\theta^2\eta r-\widehat{ B}_{\theta}\widehat{Q}_{\theta}r\Big]n\\
	&\quad-\tilde{\varepsilon}\Big(2\xi'r,-i\zeta'r+im\xi +i\zeta, i\eta'r+ik\xi r\Big)^T\\
	&\quad-(\tilde{\delta}-\frac{2}{3}\tilde{\varepsilon})\Big[\xi'r+\xi+m\zeta -k\eta r\Big]n=0, \quad \mbox{at} \quad r=r_0.
	\end{split}
	\end{equation}
\end{prop}
\begin{proof}
Fix $((\xi_0,\eta_0,\zeta_0), q_r)\in Y_{m,k}\times H_0^1(r_0,r_w)$ (see \eqref{defi-ymk}), and satisfy $ m\widehat{  B}_\theta \xi_0=rq_r$ on the boundary $r=r_0$ and $q_r=0$ on the boundary $r=r_w$.

 Define 
	$$j(t,\tau)=\mathcal{J}(\xi+t\xi_0+\tau\xi,\eta+t\eta_0+\tau\eta,\zeta+t\zeta_0+\tau\zeta)$$
	and note that $j(0,0)=1$. Moreover, $j$ is smooth, 
	\begin{equation*}
	\begin{split}
	&\frac{\partial j}{\partial t}(0,0)=2\pi^2\int_{0}^{r_0}2\rho(\xi_0\xi +\eta_0\eta+\zeta_0\zeta)rdr, \\
	&\frac{\partial j}{\partial \tau}(0,0)=2\pi^2\int_{0}^{r_0}2\rho(\xi^2 +\eta^2+\zeta^2)rdr=2.
	\end{split}
	\end{equation*}
	So, by the inverse function theorem, we can solve for $\tau=\tau(t)$ in a neighborhood of $0$ as a $C^1$ function of $t$ so that $\tau(0)=0$ and $j(t,\tau(t))=1$. We may differentiate the last equation to find 
	\begin{equation*}
	\frac{\partial j}{\partial t}(0,0)+\frac{\partial j}{\partial \tau}(0,0)\tau'(0)=0,
	\end{equation*}
	hence that 
	\begin{equation*}
	\tau'(0)=-\frac{1}{2}\frac{\partial j}{\partial t}(0,0)=-2\pi^2\int_{0}^{r_0}\rho(\xi_0\xi +\eta_0\eta+\zeta_0\zeta)rdr.
	\end{equation*}
	Since $((\xi, \eta,\zeta),\widehat{ Q}_r)$ is the minimizer over the set $\mathcal{A}_2$, we may make variations with respect to $(\xi_0,\eta_0,\zeta_0)$ to find that 
	\begin{equation*}
	0=\frac{d}{dt}\bigg|_{t=0}E(\xi+t\xi_0+\tau(t)\xi, \eta+t\eta_0+\tau(t)\eta, \zeta+t\zeta_0+\tau(t)\zeta,\widehat{ Q}_r+tq_r+\tau \widehat{ Q}_r),
	\end{equation*}
	which implies that
	\begin{equation*}
	\begin{split}
0&=4\pi^2 \int _0^{r_0}B_{\theta}^2\Big[\frac{1}{r}(r\xi)'-k\eta-\frac{2\xi}{r}\Big](r\xi_0)'dr+4\pi^2\int_0^{r_0}\gamma p\Big[\frac{1}{r}(r\xi)'-k\eta+\frac{m\zeta}{r}\Big]
(r\xi_0)'dr\\
&\quad+4\pi^2\int_0^{r_0}2B^2_\theta\Big[k\eta-\frac{1}{r}(r\xi)'+\frac{2\xi}{r}\Big]
\xi_0dr+4\pi^2\int_0^{r_0}\Big[2p'+\frac{m^2B_{\theta}^2}{r}\Big]\xi\xi_0dr\\
&\quad+4\pi^2 \int _0^{r_0}(m^2+k^2r^2)B_{\theta}^2\Big[\frac{1}{r}\eta+\frac{-k(r\xi)'+2k\xi}{m^2+k^2r^2}\Big]\eta_0d+4\pi^2\int_0^{r_0}\gamma p\Big[\frac{1}{r}(r\xi)'-k\eta\\
&\quad+\frac{m\zeta}{r}\Big]
(-k\eta_0r)dr+4\pi^2\int_0^{r_0}\gamma p\Big[\frac{1}{r}(r\xi)'-k\eta+\frac{m\zeta}{r}\Big]
m\zeta_0dr+4\pi^2 \int _0^{r_0}\Big\{\frac{2\tilde{\varepsilon}}{3}\Big(-\xi'\xi_0\\
&\quad-\xi\xi'_0+\frac{2\xi\xi_0}{r}\Big)+\Big(\frac{4\tilde{\varepsilon}}{3}+\tilde{\delta}\Big)\xi'\xi'_0r+\tilde{\varepsilon}\Big(\frac{m^2}{r^2}+k^2\Big)\xi\xi_0r+\tilde{\delta}\Big(\xi'\xi_0+\xi\xi'_0+\frac{\xi\xi_0}{r}\Big) \Big\}dr\\
&\quad+4\pi^2 \int _0^{r_0}\Big\{\Big(-\frac{2\tilde{\varepsilon}}{3}+\tilde{\delta}\Big)m\zeta \xi'_0+\Big(\frac{7\tilde{\varepsilon}}{3}+\tilde{\delta}\Big)\frac{m}{r}\zeta \xi_0-\tilde{\varepsilon}\zeta'm \xi_0+\Big(\frac{2\tilde{\varepsilon}}{3}-\tilde{\delta}\Big)k\eta\xi'_0r
\\
&\quad+\Big(\frac{2\tilde{\varepsilon}}{3}-\tilde{\delta}\Big)k\eta\xi_0+\tilde{\varepsilon}\eta'k\xi_0r
\Big\}dr +4\pi^2 \int _0^{r_0}\Big\{\Big(\tilde{\delta}
-\frac{2\tilde{\varepsilon}}{3}\Big)\xi'm\zeta
_0-\tilde{\varepsilon}m\xi\zeta'_0+\Big(\frac{7\tilde{\varepsilon}}{3}
+\tilde{\delta}\Big)\frac{m}{r}\xi\zeta_0\Big\}dr\\
&\quad-4\pi^2 \int _0^{r_0}\Big(\tilde{\delta}+\frac{\tilde{\varepsilon}}{3}\Big)k\eta m\zeta_0dr+4\pi^2 \int _0^{r_0}\Big\{\Big(\frac{4\tilde{\varepsilon}}{3}+\tilde{\delta}\Big)\frac{m^2}{r}\zeta\zeta_0+\tilde{\varepsilon}k^2\zeta\zeta_0r-\tilde{\varepsilon}\zeta'\zeta_0-\tilde{\varepsilon}\zeta\zeta'_0\\
&\quad+\tilde{\varepsilon}\frac{\zeta\zeta_0}{r}+\tilde{\varepsilon}\zeta'\zeta'_0r\Big\}dr
+4\pi^2\int_0^{r_0}\Big\{\Big(\frac{2\tilde{\varepsilon}}{3}-\tilde{\delta}\Big)k\xi'\eta_0r+\Big(\frac{2\tilde{\varepsilon}}{3}-\tilde{\delta}\Big)k\xi\eta_0+\tilde{\varepsilon}k\xi\eta'_0r\Big\}dr\\
&\quad+4\pi^2\int_0^{r_0}\Big\{\Big(\frac{4\tilde{\varepsilon}}{3}+\tilde{\delta}\Big)k^2\eta\eta_0r+\tilde{\varepsilon}\frac{m^2}{r}\eta\eta_0+\tilde{\varepsilon}\eta'\eta'_0r\Big\}dr -4\pi^2\int_0^{r_0}\Big(\frac{\tilde{\varepsilon}}{3}+\tilde{\delta}\Big)km\zeta\eta_0 dr
\\
&\quad+4\pi^2\int_{r_0}^{r_w} \bigg[\widehat{Q}_r q_r+\frac{1}{m^2+k^2r^2}(r\widehat{ Q}_r)'(rq_r)'\bigg]rdr+2\tau'(0)\lambda.
	\end{split}
	\end{equation*}
		Since $\xi_0$, $\eta_0$, $\zeta_0$ and $q_r$  are independent, we deduce
		\begin{equation}\label{xi-distri}
		\begin{split}
		& \int _0^{r_0}B_{\theta}^2\Big[\frac{1}{r}(r\xi)'-k\eta-\frac{2\xi}{r}\Big](r\xi_0)'dr+\int_0^{r_0}\gamma p\Big[\frac{1}{r}(r\xi)'-k\eta+\frac{m\zeta}{r}\Big]
		(r\xi_0)'dr\\
		&\quad +\int_0^{r_0}2B^2_\theta\Big[k\eta-\frac{1}{r}(r\xi)'+\frac{2\xi}{r}\Big]
		\xi_0dr+\int_0^{r_0}\Big[2p'+\frac{m^2B_{\theta}^2}{r}\Big]\xi\xi_0dr\\
		&\quad+ \int _0^{r_0}\Big\{\frac{2\tilde{\varepsilon}}{3}\Big(-\xi'\xi_0-\xi\xi'_0+\frac{2\xi\xi_0}{r}\Big)
		+\Big(\frac{4\tilde{\varepsilon}}{3}+\tilde{\delta}\Big)\xi'\xi'_0r+\tilde{\varepsilon}\Big(\frac{m^2}{r^2}+k^2\Big)\xi\xi_0r
	\\
	&\quad+\tilde{\delta}\Big(\xi'\xi_0+\xi\xi'_0+\frac{\xi\xi_0}{r}\Big) \Big\}dr+ \int _0^{r_0}\Big\{\Big(-\frac{2\tilde{\varepsilon}}{3}+\tilde{\delta}\Big)m\zeta \xi'_0+\Big(\frac{7\tilde{\varepsilon}}{3}+\tilde{\delta}\Big)\frac{m}{r}\zeta \xi_0\\
	&\quad-\tilde{\varepsilon}\zeta'm \xi_0+\Big(\frac{2\tilde{\varepsilon}}{3}-\tilde{\delta}\Big)k\eta\xi'_0r
		+\Big(\frac{2\tilde{\varepsilon}}{3}-\tilde{\delta}\Big)k\eta\xi_0+\tilde{\varepsilon}\eta'k\xi_0r
		\Big\}dr+\int_{r_0}^{r_w} \bigg[\widehat{Q}_r q_r\\
		&\quad+\frac{1}{m^2+k^2r^2}(r\widehat{ Q}_r)'(rq_r)'\bigg]rdr=\int_{0}^{r_0}\rho\lambda\xi_0\xi rdr,
		\end{split}
		\end{equation}	
		\begin{equation}\label{eta-distri}
		\begin{split}
	& \int _0^{r_0}(m^2+k^2r^2)B_{\theta}^2\Big[\frac{1}{r}\eta+\frac{-k(r\xi)'+2k\xi}{m^2+k^2r^2}\Big]\eta_0dr\\
	&\quad+\int_0^{r_0}\gamma p\Big[\frac{1}{r}(r\xi)'-k\eta+\frac{m\zeta}{r}\Big]
	(-k\eta_0r)dr+\int_0^{r_0}\Big\{\Big(\frac{2\tilde{\varepsilon}}{3}-\tilde{\delta}\Big)k\xi'\eta_0r\\
		&\quad+\Big(\frac{2\tilde{\varepsilon}}{3}-\tilde{\delta}\Big)k\xi\eta_0+\tilde{\varepsilon}k\xi\eta'_0r\Big\}dr+\int_0^{r_0}\Big\{\Big(\frac{4\tilde{\varepsilon}}{3}+\tilde{\delta}\Big)k^2\eta\eta_0r+\tilde{\varepsilon}\frac{m^2}{r}\eta\eta_0+\tilde{\varepsilon}\eta'\eta'_0r\Big\}dr\\
		&\quad-\int_0^{r_0}\Big(\frac{\tilde{\varepsilon}}{3}+\tilde{\delta}\Big)km\zeta\eta_0 dr=\int_{0}^{r_0}\rho\lambda\eta_0\eta rdr,
		\end{split}
		\end{equation}
		\begin{equation}\label{zeta-distri}
		\begin{split}
	&	\int_0^{r_0}\gamma p\Big[\frac{1}{r}(r\xi)'-k\eta+\frac{m\zeta}{r}\Big]
	m\zeta_0dr+ \int _0^{r_0}\bigg\{\Big(\tilde{\delta}-\frac{2\tilde{\varepsilon}}{3}\Big)\xi'm\zeta
		_0-\tilde{\varepsilon}m\xi\zeta'_0\\
		&\quad+\Big(\frac{7\tilde{\varepsilon}}{3}+\tilde{\delta}\Big)\frac{m}{r}\xi\zeta_0\bigg\}dr- \int _0^{r_0}\Big(\tilde{\delta}+\frac{\tilde{\varepsilon}}{3}\Big)k\eta m\zeta_0dr+ \int_0^{r_0}\bigg\{\Big(\frac{4\tilde{\varepsilon}}{3}+\tilde{\delta}\Big)\frac{m^2}{r}\zeta\zeta_0\\
		&\quad+\tilde{\varepsilon}k^2\zeta\zeta_0r-\tilde{\varepsilon}\zeta'\zeta_0-\tilde{\varepsilon}\zeta\zeta'_0+\tilde{\varepsilon}\frac{\zeta\zeta_0}{r}+\tilde{\varepsilon}\zeta'\zeta'_0r\bigg\}dr  =\int_{0}^{r_0}\rho\lambda\zeta_0\zeta rdr.
		\end{split}
		\end{equation}

		By making variation with $\xi_0, \eta_0$ and $\zeta_0$ compactly supported in $(0,r_0)$, and make variation $q_r$ compactly supported in $(r_0, r_w)$, one gets that $\xi$, $\eta$ and $\zeta$ satisfy \eqref{spectal formulation-viscosity} in a weak sense in $(0,r_0)$ and $\widehat{ Q}_{r}$ solves \eqref{ode-curl-div-viscosity} in a weak sense in $(r_0,r_w)$. 
		
		Now we  show that the interface boundary conditions 
		 \eqref{inter-four-2-viscosity} are satisfied. 
	Since from Proposition \ref{embeddding-m1}, we know that  $(\xi,\eta,\zeta) \in L^2(0,r_0)\times L^2(0,r_0)\times L^2(0,r_0)$, which together with $(\xi,\eta,\zeta)\in Y_{m,k}$ (see \eqref{defi-ymk}), gives that
		$(\xi,\eta,\zeta)\in H^1(\frac{r_0}{2},r_0)\times H^1(\frac{r_0}{2},r_0)\times H^1(\frac{r_0}{2},r_0)$.
			We make variations with respect to $\xi_0$, $\eta_0$ and $\zeta_0\in C_c^{\infty}(0,r_0]$, $q_r\in C_c^{\infty}[r_0,r_w)$. Integrating the terms in  \eqref{xi-distri}, \eqref{eta-distri} and \eqref{zeta-distri} with derivatives of $\xi_0$, $q_r$, $\eta_0$ and $\zeta_0$  by parts, using $\xi$, $ \eta$ and $ \zeta$ solve the system \eqref{spectal formulation-viscosity} on $(0,r_0)$ and $\widehat{ Q}_r$ solves \eqref{ode-curl-div-viscosity} on $(r_0,r_w)$, we get that 
		\begin{equation*}
		\begin{split}
	&\Big [B_{\theta}^2\Big(\frac{1}{r}(r\xi)'-k\eta-\frac{2\xi}{r}\Big)(r\xi_0)\Big]_{r=r_0}+\Big[\gamma p\Big(\frac{1}{r}(r\xi)'-k\eta+\frac{m\zeta}{r}\Big)
	(r\xi_0)\Big]_{r=r_0}\\
	&-\Big[\frac{r}{m^2+k^2r^2}(rQ_r)'r\widehat{ q}_r\Big]_{r=r_0}+\Big [\Big(\frac{4\tilde{\varepsilon}}{3}+\tilde{\delta}\Big)\xi'\xi_0r+\Big(\tilde{\delta}-\frac{2\tilde{\varepsilon}}{3}\Big)\xi\xi_0\\
	&\quad+\Big(\tilde{\delta}-\frac{2\tilde{\varepsilon}}{3}\Big)m\zeta \xi_0+\Big(\frac{2\tilde{\varepsilon}}{3}-\tilde{\delta}\Big)k\eta\xi_0r\Big]_{r=r_0}=0,
		\end{split}
		\end{equation*}
		\begin{equation*}
		\tilde{\varepsilon}(\eta'r+k\xi r)\eta_0=0,\,\,\, 
		\tilde{\varepsilon}(\zeta' r-\zeta-m\xi )\zeta_0=0.
		\end{equation*}
		Since $\xi_0$, $\eta_0$, $\zeta_0$ and $q_r$ may be chosen arbitrarily, and $q_r$ satisfies $m\widehat{  B}_\theta \xi_0=rq_r$ on the bounary $r=r_0$, using $B_{\theta}^2(r\xi)'=B_{\theta}^2\xi+B_{\theta}\xi' r$ and $\widehat{ Q}_{\theta}=-\frac{m}{m^2+k^2r^2}(r\widehat{ Q}_r)'$, we deduce the interface boundary conditions \eqref{inter-four-2-viscosity}. This proves the result.
\end{proof}

Now, we establish the continuity and monotonicity properties of the eigenvalue $\lambda(s)$.
\begin{prop}
	Let $\lambda: (0,\infty)\rightarrow \mathbb{R}$ be given by \eqref{e-3-viscosity}. Then the following hold:\\
	\mbox{(i)} $\lambda \in C_{loc}^{0,1}((0,\infty))$, and $\lambda\in C^0((0,\infty))$.\\
	\mbox{(ii)} There exists a positive constant $C_3=C_3(r_0,\mathbb{J}_z,\varepsilon,\delta,m,k)$ so that
	\begin{equation}\label{decomp-m}
	\lambda(s)\geq -C+sC_3.
	\end{equation}\\
	\mbox{(iii)} $\lambda(s) $ is strictly increasing.
\end{prop}
\begin{proof}
	Fix a compact interval $Q=[a,b] \subset\subset (0,\infty)$, and fix any  $((\xi_0,\eta_0,\zeta_0),\widehat{ Q}_{r0})\in \mathcal{A}_2$. We can decompose $E_{m,k}$ as follows
	\begin{equation}\label{decomp-2-m}
	E_{m,k}(\xi,\eta,\zeta,\widehat{ Q}_r;s)=E_{m,k}^0(\xi,\eta,\zeta,\widehat{ Q}_r)+sE_{m,k}^1(\xi,\eta,\zeta),
	\end{equation}
	with \begin{equation*}
	\begin{split}
	E_{m,k}^0(\xi,\eta,\zeta,\widehat{ Q}_r)&=2\pi^2\int_0^{r_0}\Big\{(m^2+k^2r^2)\Big[\frac{B_\theta}{r}\eta+\frac{-kB_\theta(r\xi)'+2kB_{\theta}\xi}{m^2+k^2r^2}\Big]^2\\
	&\quad+\gamma p\Big[\frac{1}{r}(r\xi)'-k\eta+\frac{m\zeta}{r}\Big]^2\Big\}rdr
	+2\pi^2\int_0^{r_0}\frac{m^2B_\theta^2}{r(m^2+k^2r^2)}(\xi-r\xi')^2dr\\
	&\quad+2\pi^2\int_0^{r_0}\Big[2p'+\frac{m^2B_{\theta}^2}{r}\Big]\xi^2dr
	+2\pi^2\int_{r_0}^{r_w}\bigg[|\widehat{Q}_r|^2+\frac{1}{m^2+k^2r^2}|(r\widehat{Q}_r)'|^2\bigg]rdr,
	\end{split}
	\end{equation*}
	\begin{equation}\label{e1-m}
	\begin{split}
	E_{m,k}^1(\xi,\eta,\zeta)&=2\pi^2\int_0^{r_0}\varepsilon\Big[\frac{2}{9}\Big(-2\xi'+\frac{\xi}{r}+\frac{m}{r}\zeta-k\eta\Big)^2\\
	&\quad+\frac{2}{9}\Big(\xi'-\frac{2\xi}{r}-\frac{2m}{r}\zeta-k\eta\Big)^2
	+\frac{2}{9}\Big(\xi'+\frac{\xi}{r}+\frac{m}{r}\zeta+2k\eta\Big)^2\\
	&\quad+\Big(-\zeta'+\frac{\zeta}{r}+\frac{m}{r}\xi\Big)^2+(\eta'+k\xi)^2+\Big(\frac{m}{r}\eta
	-k\zeta\Big)^2\Big]rdr
	\\
	&\quad+2\pi^2\int_0^{r_0}\delta\Big(\xi'+\frac{\xi}{r}+\frac{m}{r}\zeta-k\eta\Big)^2rdr>0.
	\end{split}
	\end{equation}
	The nonnegativity of $E_{m,k}^1$ implies that $E_{m,k}$ is non-decreasing in $s$ with $((\xi,\eta,\zeta),\widehat{ Q}_r)\in \mathcal{A}_2$ kept fixed. 
	By Proposition \ref{infimum-A-out-viscosity}, for each $s\in (0,\infty) $ we can find a suit $((\xi,\eta,\zeta),\widehat{ Q}_r)\in \mathcal{A}_2$ so that 
	\begin{equation*}
	E_{m,k}(\xi_s,\eta_s,\zeta_s,\widehat{ Q}_{rs};s)=\inf_{((\xi,\eta,\zeta),\widehat{ Q}_r)\in \mathcal{A}}E_{m,k}(\xi,\eta,\zeta,\widehat{ Q}_r;s)=\lambda(s).
	\end{equation*}
	From the nonnegativity of $E_{m,k}^1$, the minimality of $(\xi_s,\eta_s, \zeta_s,\widehat{ Q}_{rs})$, and for $0<s_0<r_0$
 \begin{equation}\label{E^1_m-k}
	\begin{split}
	E_{m,k}(\xi,\eta,\zeta,\widehat{ Q}_r)\geq \pi^2\int_{s_0}^{r_0}\gamma p\Big|\frac{1}{r}(r\xi)'\Big|^2rdr-3C\mathcal{J}\geq -3C\mathcal{J}
	\end{split}
	\end{equation} which can be established from Lemma 3.17 of Bian-Guo-Tice \cite{bian-guo-tice-inviscid}, using Lemma \ref{presure-boundary},
we have 
	\begin{equation*}
	E_{m,k}(\xi_0,\eta_0,\zeta_0,\widehat{ Q}_{r0};b)\geq E_{m,k}(\xi_0,\eta_0,\zeta_0,\widehat{ Q}_{r0};s)\geq E_{m,k}(\xi_s,\eta_s,\zeta_s,\widehat{ Q}_{rs};s)\geq sE_{m,k}^1(\xi_s,\eta_s,\zeta_s)-C,
	\end{equation*}
	for all $s\in Q$. This implies that there exists a constant $$0<K=K(a,b,\xi_0,\eta_0,\zeta_0,\widehat{ Q}_{r0},\pi, L, \mathbb{J}_z)<\infty$$ so that 
	\begin{equation}\label{bound-e1-m}
	\sup_{s\in Q}E_{m,k}^1(\xi_s,\eta_s,\zeta_s)\leq K.
	\end{equation}
	Let $s_i\in Q$ for $i=1,2$. Using the minimality of $((\xi_{s_1},\eta_{s_1}, \zeta_{s_1}),\widehat{ Q}_{rs_1}) $ compared to $((\xi_{s_2},\eta_{s_2}, \zeta_{s_2}),\widehat{ Q}_{rs_2}) $, we have 
	\begin{equation}\label{lam-s1-m}
	\lambda(s_1)=E_{m,k}(\xi_{s_1},\eta_{s_1},\zeta_{s_1},\widehat{ Q}_{rs_1};s_1)\leq E_{m,k}(\xi_{s_2},\eta_{s_2},\zeta_{s_2},\widehat{ Q}_{rs_2};s_1),
	\end{equation}
	which together with \eqref{decomp-2-m} gives that
	\begin{equation}\label{lam-s2-m}
	\begin{split}
	E_{m,k}(\xi_{s_2},\eta_{s_2},\zeta_{s_2},\widehat{ Q}_{rs_2};s_1)&\leq E_{m,k}(\xi_{s_2},\eta_{s_2},\zeta_{s_2},\widehat{ Q}_{rs_2};s_2)+|s_1-s_2|E_{m,k}^1(\xi_{s_2},\eta_{s_2},\zeta_{s_2})\\
	&\quad=\lambda(s_2)+|s_1-s_2|E_{m,k}^1(\xi_{s_2},\eta_{s_2},\zeta_{s_2}).
	\end{split}
	\end{equation}
	Combining \eqref{bound-e1-m}, \eqref{lam-s1-m} and \eqref{lam-s2-m}, we can get
	$\lambda(s_1)\leq \lambda(s_2)+K|s_1-s_2|,$
	which shows that (i) holds.

Now, let's prove (ii). Note that \eqref{E^1_m-k}
and the nonnegativity of $E_{m,k}^1$ imply that 
	$$\lambda(s)\geq s\inf_{((\xi,\eta,\zeta),\widehat{ Q}_r)\in \mathcal{A}_2}E_{m,k}^1(\xi,\eta,\zeta)-C,$$
	where we denote the constant $C_3=\inf_{((\xi,\eta,\zeta),\widehat{ Q}_r)\in \mathcal{A}_2}E_{m,k}^1(\xi,\eta,\zeta)$ and this canstant is positive. 

Finally, we prove (iii). Notice that if $0<s_1<s_2<\infty$, then
	the decomposition \eqref{decomp-2-m} ensures that
	\begin{equation*}
	\begin{split}
	\lambda(s_1)&=E_{m,k}(\xi_{s_1},\eta_{s_1},\zeta_{s_1},\widehat{ Q}_{rs_1};s_1)\leq E_{m,k}(\xi_{s_2},\eta_{s_2},\zeta_{s_2},\widehat{ Q}_{rs_2};s_1)\\
	&\leq E_{m,k}(\xi_{s_2},\eta_{s_2},\zeta_{s_2},\widehat{ Q}_{rs_2};s_2)=\lambda(s_2).
	\end{split}
	\end{equation*} So $\lambda $ is non-decreasing in $s$. Suppose by way of contradiction that
	$\lambda(s_1)=\lambda(s_2)$. Then, the above inequality implies that 
	$s_1E_{m,k}^1(\xi_{s_2},\eta_{s_2},\zeta_{s_2})=s_2E_{m,k}^1(\xi_{s_2},\eta_{s_2},\zeta_{s_2})$,
	which gives that
	$E_{m,k}^1(\xi_{s_2},\eta_{s_2},\zeta_{s_2})=0$. This in turn implies $\xi_{s_2}=\eta_{s_2}=\zeta_{s_2}=0$, which contradicts that $((\xi_{s_2},\eta_{s_2},\zeta_{s_2}),\widehat{ Q}_{rs_2})\in \mathcal{A}_2$. Therefore equality can not be achieved, and $\lambda$ is strictly increasing in $s$.
\end{proof} 
If there exists $r^*$ such that $2p'(r^*)+\frac{m^2B_\theta^2(r^*)}{r^*}<0$ for $m\neq 0$, by Proposition \ref{infimum-A-out-3},  we can get the existence of $s$.
\begin{rmk}\label{regularity-um1}
	Define the open set  $\mathcal{S}=\lambda^{-1}((-\infty,0)) \subset(0,\infty)$, then we calculate $\mu=\sqrt{-\lambda}>0$. The open set $\mathcal{S}$ is nonempty by Proposition \ref{infimum-A-out-3}. 
\end{rmk}
If there exists $r^*$ such that $2p'(r^*)+\frac{m^2B_\theta^2(r^*)}{r^*}<0$ for $m\neq 0$, by Proposition \ref{infimum-A-out-3}, we can prove the uniqueness of $s$ on the open set $\mathcal{S}$.
\begin{prop}\label{prop3.14}
	There exists a unique $s\in \mathcal{S}$ so that $\mu(s)=\sqrt{-\lambda}>0$ and 
	$s=\mu(s)$.
\end{prop}
\begin{proof}
	From Proposition \ref{infimum-A-out-3}, we know that there exist two constants $C_0>0$ and $C_1>0$ such that $\lambda(s)\leq -C_0+sC_1$. On the other hand, the lower bound \eqref{decomp-m} implies that $\lambda(s)\rightarrow +\infty$, as $s\rightarrow \infty$. Since $\lambda$ is continuous and strictly increasing, there exists $s^* \in (0,\infty)$ so that 
	$$\mathcal{S}=\lambda^{-1}((-\infty,0))=(0,s^*).$$ Since $\lambda<0$, on the set $\mathcal{S}$, we can define $\mu=\sqrt{-\lambda(s)}$. Now, we define the function $\Phi: (0,s^*)\rightarrow(0,\infty)$ according to $\Phi(s)=\frac{s}{\mu(s)}$. So $\Phi$ is continous and strictly increasing in $s$ from the continuity and monotonicity properties of $\lambda$. Moreover, we know that $\lim_{s\rightarrow 0}\Phi(s)=0$ and $\lim_{s\rightarrow s^*}\Phi(s)=+\infty$. By the intermediate value theorem, there exists $s\in (0,s^*)$ so that $\Phi(s)=1$, that is, $s=\mu(s)$. This $s$ is unique since $\Phi
	$ is strictly increasing. 
\end{proof}
Finally, we establish the regularity about $\xi$, $\eta$ and $\zeta$.
\begin{prop}\label{regularity}
	Let $(\xi,\eta,\zeta)$ be the solutions to \eqref{spectal formulation-viscosity}, then there exists $A_n$ such that 
	\begin{equation*}
	\|\xi^{(n)}r^{(2n-1)/2}\|_{L^2(0,r_0)}\leq A_n,\,\,
	\|\eta^{(n)}r^{(2n-1)/2}\|_{L^2(0,r_0)}\leq A_n,
\,\,
	\|\zeta^{(n)}r^{n}\|_{L^2(0,r_0)}\leq A_n.
	\end{equation*}	
	with $A_n$ depending on $r_0$, $k$, $\pi$ and the presure $p$.
\end{prop}
\begin{proof}
	From $(\xi,\eta,\zeta)\in Y_{m,k}$ and the compactness results in Proposition  \ref{embeddding-m1}, we can get $(\xi,\eta, \zeta)\in L^2(0,r_0)$,  $\Big(\sqrt{r}\xi',\sqrt{r}\eta', \sqrt{r}\zeta',  \frac{\eta}{\sqrt{r}},\Big(\frac{\xi}{r}+\frac{m\zeta}{r}\Big)\sqrt{r}\Big)\in L^2(0,r_0)$. 
	From the system \eqref{spectal formulation-viscosity}, we can get 
	\begin{equation*}
	\|\xi''r^{3/2}\|_{L^2(0,r_0)}\leq C\Big\|\Big(\xi,\eta,\zeta,\sqrt{r}\xi',\sqrt{r}\eta',\sqrt{r}\zeta',
	\Big(\frac{\xi}{r}+\frac{m\zeta}{r}\Big)\sqrt{r}\Big)\Big\|_{L^2(0,r_0)},
	\end{equation*}
	\begin{equation*}
	\|\eta''r^{3/2}\|_{L^2(0,r_0)}\leq C\Big\|\Big(\xi,\eta,\zeta,\sqrt{r}\xi',\sqrt{r}\eta',\frac{\eta}{\sqrt{r}}\Big)\Big\|_{L^2(0,r_0)},
	\end{equation*}
		\begin{equation*}
	\|\zeta''r^{2}\|_{L^2(0,r_0)}\leq C\Big\|\Big(\eta, \zeta,\sqrt{r}\xi',\sqrt{r}\zeta',
	\Big(\frac{\xi}{r}+\frac{m\zeta}{r}\Big)\sqrt{r}\Big)\Big\|_{L^2(0,r_0)}.
	\end{equation*} 
	By induction on $n$. Suppose for some $n\geq 1$, 
	\begin{equation*}
	\|\xi^{(n)}r^{(2n-1)/2}\|_{L^2(0,r_0)}\leq A_n,\,\,
	\|\eta^{(n)}r^{(2n-1)/2}\|_{L^2(0,r_0)}\leq A_n,\,\,
	\|\zeta^{(n)}r^{n}\|_{L^2(0,r_0)}\leq A_n.
	\end{equation*}
	By Remark \ref{regularity-um1}, then differentiating \eqref{spectal formulation-viscosity}, we get that there exists a constant $C$ depending on the various parameters so that 
	\begin{equation*}
	\|\xi^{(n+1)}r^{(2n+1)/2}\|_{L^2(0,r_0)}\leq CA_n\leq A_{n+1},
	\end{equation*}
	\begin{equation*}
	\|\eta^{(n+1)}r^{(2n+1)/2}\|_{L^2(0,r_0)}\leq CA_n\leq A_{n+1},
	\end{equation*}
	\begin{equation*}
	\|\zeta^{(n+1)}r^{n+1}\|_{L^2(0,r_0)}\leq CA_n\leq  A_{n+1}.
	\end{equation*} 
	Then, the bound holds for $n+1$, and so by induction the bound holds for all $n\geq 1$.
\end{proof}
\section{Lower bound of the dissipation rate $D_{m,k}$ in \eqref{dissipation-rate} }
In this section, we will prove the lower bound of the dissipation rate $D_{m,k}$ defined in  \eqref{dissipation-rate} and the existence of the biggest growing mode for any $(m,k)\in \mathbb{Z}\times\mathbb{Z}$. 	For notational simplicity, we use $o(1)$ and $O(1)$ to denote constants independent of $k$ and $m$ to be $1$.
	First, we introduce the following lemma which will be used in the proof of Proposition \ref{new-prop-bian-guo}.
	\begin{lem}\label{basic-cutoff}
For any function $(f_1,f_2,f_3)\in X_{k}$ or $ Y_{m,k}$ (see \eqref{defi-Xk} and \eqref{defi-ymk}), there holds that for $i=1,2,3$
\begin{equation}\label{first}
\begin{split}
f_i^2(r_0)\leq 2 \Big(\int_{\frac{r_0}{2}}^{r_0}|f_i|^2dr\Big)^{\frac{1}{2}}
\Big(\int_{\frac{r_0}{2}}^{r_0}|f_i'|^2dr+\int_{\frac{r_0}{2}}^{r_0}|f_i|^2dr\Big)^{\frac{1}{2}},
\end{split}
\end{equation}
\begin{equation}\label{second}
\begin{split}
\Big|f_i(r_0)\Big|\leq o(1)\Big(\int_{0}^{r_0}|f'_i|^2rdr\Big)^{\frac{1}{2}}
+O(1)\Big(\int_{0}^{r_0}\rho |f_i|^2rdr\Big)^{\frac{1}{2}}.
\end{split}
\end{equation}
	\end{lem}
\begin{proof}
We choose the smooth cutoff function $\Psi(r)$, satisfying $0\leq \Psi(r)\leq 1$ for $r\in[\frac{r_0}{2},r_0]$, $\Psi(r)=0$ for $r\in [0,\frac{r_0}{2}]$ and $\Psi(r_0)=1$.
Then we can prove that for $i=1,2,3$
\begin{equation*}
\begin{split}
f_i^2(r_0)&=2\int_0^{r_0}(\Psi f_i)(\Psi f_i)'dr\leq 2 \Big(\int_0^{r_0}|\Psi f_i|^2dr\Big)^{\frac{1}{2}}\Big(\int_0^{r_0}|(\Psi f_i)'|^2dr\Big)^{\frac{1}{2}}\\
&\leq 2 \Big(\int_{\frac{r_0}{2}}^{r_0}|f_i|^2dr\Big)^{\frac{1}{2}}
\Big(\int_{\frac{r_0}{2}}^{r_0}|f_i'|^2dr+\int_{\frac{r_0}{2}}^{r_0}|f_i|^2dr\Big)^{\frac{1}{2}}.
\end{split}
\end{equation*}
For $0<r_*<r_0$ we choose the smooth cutoff function $\Phi(r)$ with support in $(\frac{r_0}{2},\frac{r_*+r_0}{2})$, satisfying $0\leq \Phi(r)\leq 1$ for $r\in[\frac{r_0}{2},\frac{r_*+r_0}{2}]$, $\Phi(r)=0$ for $r\in [0,\frac{r_0}{2}]\bigcup[\frac{r_*+r_0}{2},r_0] $ and $\Phi(r_*)=1$.
Then we can show that for $i=1,2,3$
\begin{equation*}
\begin{split}
f_i^2(r_*)&=2\int_0^{r_*}(\Phi f_i)(\Phi f_i)'dr\leq 2 \Big(\int_0^{\frac{r_*+r_0}{2}}|\Phi f_i|^2dr\Big)^{\frac{1}{2}}\Big(\int_0^{\frac{r_*+r_0}{2}}|(\Phi f_i)'|^2dr\Big)^{\frac{1}{2}}
\\
&\leq 2 \Big(\int_{\frac{r_0}{2}}^{\frac{r_*+r_0}{2}}|f_i|^2dr\Big)^{\frac{1}{2}}
\Big(\int_{\frac{r_0}{2}}^{\frac{r_*+r_0}{2}}|f_i'|^2dr+\int_{\frac{r_0}{2}}^{\frac{r_*+r_0}{2}}|f_i|^2dr\Big)^{\frac{1}{2}}\\
&\lesssim \Big(\int_{\frac{r_0}{2}}^{\frac{r_*+r_0}{2}}\rho|f_i|^2rdr\Big)^{\frac{1}{2}}
\Big(\int_{\frac{r_0}{2}}^{\frac{r_*+r_0}{2}}|f_i'|^2rdr+\int_{\frac{r_0}{2}}^{\frac{r_*+r_0}{2}}\rho|f_i|^2rdr\Big)^{\frac{1}{2}}\\
&\leq o(1)\int_{0}^{r_0}|f_i'|^2rdr+O(1)\int_{0}^{r_0}\rho|f_i|^2rdr.
\end{split}
\end{equation*}
Hence, for $(r_0-r_*)^{\frac{1}{2}}=o(1)$, we get
\begin{equation*}
\begin{split}
|f_i(r_0)|&\leq |f_i(r_0)-f_i(r_*)|+|f_i(r_*)|\\
&\lesssim \Big(\int_{r_*}^{r_0}f_i'^2dr\Big)^{\frac{1}{2}}(r_0-r_*)^{\frac{1}{2}}
+o(1)\Big(\int_{0}^{r_0}|f_i'|^2rdr\Big)^{\frac{1}{2}}+O(1)\Big(\int_{0}^{r_0}\rho|f_i|^2rdr\Big)^{\frac{1}{2}}\\
&\leq o(1)\Big(\int_{0}^{r_0}|f_i'|^2rdr\Big)^{\frac{1}{2}}+O(1)\Big(\int_{0}^{r_0}\rho|f_i|^2rdr\Big)^{\frac{1}{2}}.
\end{split}
\end{equation*}
\end{proof}
Now we establish the lower bound of the dissipation rate $D_{m,k}$ in \eqref{dissipation-rate}.
\begin{prop}\label{new-prop-bian-guo}
Assume that $0\leq \rho \leq C$, $\mathcal{J}=\int_0^{r_0}\rho(\xi^2+\eta^2+\zeta^2)rdr=1$, when $m=0$, $(\xi,\eta,\zeta)\in X_{k}$, and when $m\neq 0$, $(\xi,\eta,\zeta)\in Y_{m,k}$. Then for any $|m|\neq 1$ and $k \in \mathbb{Z}$,
 \begin{equation}\label{m neq 1}
 D_{m,k}\gtrsim\|(\xi',\eta',\zeta')\sqrt{r}\|_{L^2}+k^2+m^2-O(1).
 \end{equation}
 For $|m|=1$ and any $k\in \mathbb{Z}$,
 \begin{equation}\label{m eq 1}
D_{m,k}\gtrsim\|(\xi',\eta',\zeta')\sqrt{r}\|_{L^2}^2+k^2+\int_0^{r_0}\frac{1}{r}\eta^2dr+ \int_0^{r_0}\Big(\frac{\xi}{r}+\frac{m\zeta}{r}\Big)^2rdr-O(1).
 \end{equation}
\end{prop}
\begin{proof}
We can estimate the dissipation rate $D_{m,k}$ as that	
\begin{equation}\label{bian-1}
\begin{split}
D_{m,k}
&
	\geq 2\pi^2\min(\varepsilon,\delta)\int_0^{r_0}\bigg[\frac{2}{9}\Big(-2\xi'+\frac{\xi}{r}+\frac{m}{r}\zeta-k\eta\Big)^2+\frac{2}{9}\Big(\xi'+\frac{\xi}{r}+\frac{m}{r}\zeta-k\eta\Big)^2\bigg]rdr\\
	&\quad+2\pi^2\min(\varepsilon,\delta)\int_0^{r_0}\bigg[\frac{2}{9}\Big(\xi'+\frac{\xi}{r}+\frac{m}{r}\zeta+2k\eta\Big)^2+\frac{2}{9}\Big(\xi'+\frac{\xi}{r}+\frac{m}{r}\zeta-k\eta\Big)^2\bigg]rdr\\
	&\quad+2\pi^2\min(\varepsilon,\delta)\int_0^{r_0}\bigg[\frac{2}{9}\Big(\xi'-\frac{2\xi}{r}-\frac{2m}{r}\zeta-k\eta\Big)^2+\frac{2}{9}\Big(\xi'+\frac{\xi}{r}+\frac{m}{r}\zeta-k\eta\Big)^2\bigg]rdr\\
	&\quad+2\pi^2\varepsilon\int_0^{r_0}\bigg[\Big(-\zeta'+\frac{\zeta}{r}+\frac{m}{r}\xi\Big)^2+(\eta'+k\xi)^2+\Big(\frac{m}{r}\eta-k\zeta\Big)^2\bigg]rdr\\
	&\quad+\frac{2\pi^2}{3}\delta\int_0^{r_0}\Big(\xi'+\frac{\xi}{r}+\frac{m}{r}\zeta-k\eta\Big)^2rdr\\
&\geq 2\pi^2\min(\varepsilon,\delta) \int_0^{r_0}{\xi'}^2rdr+2\pi^2\min(\varepsilon,\delta) \int_0^{r_0}k^2\eta^2rdr\\
&\quad+2\pi^2\min(\varepsilon,\delta) \int_0^{r_0}\Big(\frac{\xi}{r}+\frac{m}{r}\zeta\Big)^2rdr+2\varepsilon\pi^2\int_0^{r_0}(\eta'+k\xi)^2rdr\\
&\quad+2\varepsilon\pi^2\int_0^{r_0}\Big(\frac{m}{r}\eta-k\zeta\Big)^2rdr+2\pi^2\varepsilon\int_0^{r_0}\Big(-\zeta'+\frac{\zeta}{r}+\frac{m}{r}\xi\Big)^2rdr\\
&\quad+\frac{2\pi^2}{3}\delta\int_0^{r_0}\Big(\xi'+\frac{\xi}{r}+\frac{m}{r}\zeta-k\eta\Big)^2rdr,
\end{split}
\end{equation}	
	where we have used the facts that $a^2+b^2\geq \frac{1}{2}(a-b)^2,$ with $a=-2\xi'+\frac{\xi}{r}+\frac{m}{r}\zeta-k\eta$, $b=\xi'+\frac{\xi}{r}+\frac{m}{r}\zeta-k\eta$;
	$a=\xi'+\frac{\xi}{r}+\frac{m}{r}\zeta+2k\eta$, $b=\xi'+\frac{\xi}{r}+\frac{m}{r}\zeta-k\eta$;
	$a=\xi'-\frac{2\xi}{r}-\frac{2m}{r}\zeta-k\eta$, $b=\xi'+\frac{\xi}{r}+\frac{m}{r}\zeta-k\eta$.

Now we prove our result by the following three cases.

{\bf Case I: when $m=0$, the estimate \eqref{m neq 1}  holds.}

From \eqref{bian-1}, we can get
\begin{equation}\label{estimate-c-k}
\begin{split}
D_{0,k}&\geq 2\pi^2\min(\varepsilon,\delta) \int_0^{r_0}{\xi'}^2rdr+2\pi^2\min(\varepsilon,\delta) \int_0^{r_0}\frac{\xi^2}{r}dr\\
&\quad+2\pi^2\varepsilon\int_0^{r_0}\Big(-\zeta'+\frac{\zeta}{r}\Big)^2rdr+2\pi^2\min(\varepsilon,\delta) \int_0^{r_0}k^2\eta^2rdr\\
&\quad+2\varepsilon\pi^2\int_0^{r_0}k^2\zeta^2rdr+2\varepsilon\pi^2\int_0^{r_0}(\eta'+k\xi)^2rdr.
\end{split}
\end{equation}

Since the minimizer $(\xi,\eta,\zeta)\in X_k$,  Definition \ref{defi-X} implies that $\xi(0)=0$. From Proposition \ref{embeddding-b},  we know that $\int_{0}^{r_0}\eta'^2rdr$ and $\int_0^{r_0}\eta^2dr$ are bounded, which shows that $(\eta r)'\in L^2(0,r_0)$. So we get that $\eta r$ is well-defined at the origin $r=0$. Therefore, the boundary term 
$ \eta k\xi r|_{r=0}=0$.

	Then	we can prove from integrating by parts that
	\begin{equation}\label{boundary-term-g}
	\begin{split}
	&	\int_0^{r_0}(\eta'+k\xi)^2rdr=\int_0^{r_0}\eta'^2rdr+\int_0^{r_0}k^2\xi^2rdr
	+2\int_0^{r_0}\eta'k\xi rdr\\
	&=\int_0^{r_0}\eta'^2rdr+\int_0^{r_0}k^2\xi^2rdr+2\eta k\xi r|_{r=r_0}
	-2\int_0^{r_0}\eta k\xi' rdr-2\int_0^{r_0}\eta k\xi dr.
	\end{split}
	\end{equation}
	Notice that
	$$\Big|-2\int_0^{r_0}\eta k\xi' rdr\Big|\leq \int_0^{r_0}k^2\eta^2rdr+\int_0^{r_0}\xi'^2rdr,$$
	$$\Big|-2\int_0^{r_0}\eta k\xi dr\Big|\leq  \int_0^{r_0}\frac{\xi^2}{r}dr+\int_0^{r_0}k^2\eta^2rdr.$$
	Then we have \begin{equation*}
	\begin{split}
	\varepsilon\pi^2\int_0^{r_0}(\eta'+k\xi)^2rdr&\geq \varepsilon\pi^2\int_0^{r_0}\eta'^2rdr+\varepsilon\pi^2\int_0^{r_0}k^2\xi^2rdr+2\varepsilon\pi^2\eta k\xi r|_{r=r_0}\\
	&\quad-2\varepsilon\pi^2\int_0^{r_0}k^2\eta^2rdr-\varepsilon\pi^2\int_0^{r_0}\xi'^2rdr
	-\varepsilon\pi^2\int_0^{r_0}\frac{\xi^2}{r}dr,
	\end{split}
	\end{equation*}
	which gives that
	\begin{equation}\label{O(1)-2-j}
		\begin{split}
		\varepsilon\pi^2\int_0^{r_0}(\eta'+k\xi)^2rdr&\geq \varepsilon\pi^2\int_0^{r_0}\eta'^2rdr+\varepsilon\pi^2\int_0^{r_0}k^2\xi^2rdr\\
		&\quad+2\varepsilon\pi^2\eta k\xi r|_{r=r_0}
		-O(1)D_{0,k}.
		\end{split}
		\end{equation}

On the other hand, since the minimizer $(\xi,\eta,\zeta)\in X_{k}$,  Definition \ref{defi-X} implies that $\zeta(0)=0$. Therefore, we get
\begin{equation*}
\begin{split}
\varepsilon\pi^2\int_0^{r_0}\Big(-\zeta'+\frac{\zeta}{r}\Big)^2rdr=\varepsilon\pi^2\int_0^{r_0}\Big(\zeta'^2r+\frac{\zeta^2}{r}\Big)dr-\varepsilon\pi^2\zeta^2|_{r=r_0}.
\end{split}
\end{equation*}
which combining \eqref{estimate-c-k} with \eqref{O(1)-2-j}, we deduce 
\begin{equation*}\label{estimate-c-k-2}
\begin{split}
O(1)D_{0,k}&
\geq 2\pi^2\min(\varepsilon,\delta) \int_0^{r_0}{\xi'}^2rdr+2\pi^2\min(\varepsilon,\delta) \int_0^{r_0}\frac{\xi^2}{r}dr+2\pi^2\min(\varepsilon,\delta) \int_0^{r_0}k^2\eta^2rdr\\
&\quad+2\varepsilon\pi^2\int_0^{r_0}k^2\zeta^2rdr+2\varepsilon \pi^2\int_0^{r_0}{\eta'}^2rdr+2\varepsilon \pi^2\int_0^{r_0} k^2\xi^2rdr+2\varepsilon\pi^2\int_0^{r_0}\Big(\zeta'^2r+\frac{\zeta^2}{r}\Big)dr
\\
&\quad+4\varepsilon\pi^2\eta k\xi r|_{r=r_0}-2\varepsilon\pi^2\zeta^2|_{r=r_0}.
\end{split}
\end{equation*}
	The remaining thing is to deal with the boundary terms.
	
	By \eqref{second} in Lemma \ref{basic-cutoff}, we get
	\begin{equation}\label{boundary-term-1}
	\begin{split}
	\Big|-2\varepsilon\pi^2\zeta^2|_{r=r_0}\Big|
\leq 	o(1)\int_0^{r_0}\zeta'^2rdr
	+O(1)\mathcal{J}.
	\end{split}
	\end{equation}
Using \eqref{first}  in Lemma \ref{basic-cutoff} and Cauchy inequality, we have	\begin{equation}\label{boundary-term}
	\begin{split}
	\Big|4c\varepsilon\pi^2\eta& k\xi r|_{r=r_0}\Big|\leq 8cr_0k\varepsilon\pi^2 \Big(\int_{\frac{r_0}{2}}^{r_0}|\eta|^2dr\Big)^{\frac{1}{4}}
	\Big(\int_{\frac{r_0}{2}}^{r_0}|\eta'|^2dr+\int_{\frac{r_0}{2}}^{r_0}|\eta|^2dr\Big)^{\frac{1}{4}}\\
	&
	\qquad\times\Big(\int_{\frac{r_0}{2}}^{r_0}|\xi|^2dr\Big)^{\frac{1}{4}}
	\Big(\int_{\frac{r_0}{2}}^{r_0}|\xi'|^2dr+\int_{\frac{r_0}{2}}^{r_0}|\xi|^2dr\Big)^{\frac{1}{4}}\\
	&\leq 8cr_0\varepsilon\pi^2 \Big(\int_{\frac{r_0}{2}}^{r_0}|k\eta|^2dr\Big)^{\frac{1}{4}}\Big(\int_{\frac{r_0}{2}}^{r_0}|k\xi|^2dr\Big)^{\frac{1}{4}}
	\Big(\int_{\frac{r_0}{2}}^{r_0}|\eta'|^2dr\Big)^{\frac{1}{4}}\Big(\int_{\frac{r_0}{2}}^{r_0}|\xi'|^2dr\Big)^{\frac{1}{4}}\\
	&\qquad+8cr_0\varepsilon\pi^2 \Big(\int_{\frac{r_0}{2}}^{r_0}|k\eta|^2dr\Big)^{\frac{1}{2}}\Big(\int_{\frac{r_0}{2}}^{r_0}|\xi|^2dr\Big)^{\frac{1}{4}}\Big(\int_{\frac{r_0}{2}}^{r_0}|\xi'|^2dr\Big)^{\frac{1}{4}}\\
	&\qquad+8cr_0\varepsilon\pi^2 \Big(\int_{\frac{r_0}{2}}^{r_0}|\eta|^2dr\Big)^{\frac{1}{4}}\Big(\int_{\frac{r_0}{2}}^{r_0}|k\xi|^2dr\Big)^{\frac{1}{2}}\Big(\int_{\frac{r_0}{2}}^{r_0}|\eta'|^2dr\Big)^{\frac{1}{4}}\\
	&\qquad+8cr_0\varepsilon\pi^2 \Big(\int_{\frac{r_0}{2}}^{r_0}|k\eta|^2dr\Big)^{\frac{1}{2}}\Big(\int_{\frac{r_0}{2}}^{r_0}|\xi|^2dr\Big)^{\frac{1}{2}}\\
	&\leq o(1)\Big(\int_0^{r_0}\eta'^2rdr+\int_0^{r_0}k^2\xi^2rdr\Big)
	+O(1)\Big(\int_{\frac{r_0}{2}}^{r_0}|k\eta|^2rdr+\int_{\frac{r_0}{2}}^{r_0}\xi'^2rdr\Big)\\
	&\leq o(1)\Big(\int_0^{r_0}\eta'^2rdr+\int_0^{r_0}k^2\xi^2rdr\Big)+O(1)D_{0,k},
	\end{split}
	\end{equation}	where 
	we have used the facts that for any function $h$ and $k\neq 0$, it holds that $$\int_{\frac{r_0}{2}}^{r_0}h^2dr\leq \frac{2}{r_0}\int_0^{r_0}k^2h^2rdr.$$ 
Then by the steady density $0\leq \rho \leq C$ and $\mathcal{J}=\int_0^{r_0}\rho(\xi^2+\eta^2+\zeta^2)rdr=1$, we can show that 
\begin{equation*}
\begin{split}
O(1)D_{0,k}
&\geq 2\pi^2\min(\varepsilon,\delta) \int_0^{r_0}{\xi'}^2rdr+2\pi^2\min(\varepsilon,\delta) \int_0^{r_0}\frac{\xi^2}{r}dr+2\pi^2\min(\varepsilon,\delta) \int_0^{r_0}k^2\eta^2rdr\\
&\quad+2\varepsilon\pi^2\int_0^{r_0}k^2\zeta^2rdr+\varepsilon \pi^2\int_0^{r_0}{\eta'}^2rdr+\varepsilon \pi^2\int_0^{r_0} k^2\xi^2rdr+\varepsilon\pi^2\int_0^{r_0}\Big(\zeta'^2r+\frac{2\zeta^2}{r}\Big)dr\\
&\quad-O(1)\mathcal{J}\\
&
\geq \pi^2 \min(2\delta,\varepsilon) \int_0^{r_0}(\xi'^2+\eta'^2+\zeta'^2)rdr+\pi^2\min\bigg( \frac{2\delta}{C}, \frac{\varepsilon}{C} \bigg)k^2\int_0^{r_0}\rho(\xi^2+\eta^2+\zeta^2)rdr\\
&\quad-O(1)\\
&= \pi^2 \min(2\delta,\varepsilon) \int_0^{r_0}(\xi'^2+\eta'^2+\zeta'^2)rdr+\pi^2\min\bigg( \frac{2\delta}{C}, \frac{\varepsilon}{C} \bigg)k^2-O(1),
\end{split}
\end{equation*}
which implies that \eqref{m neq 1} holds. 

{ \bf  Case (II): when $|m|=1$,  the estimate \eqref{m eq 1}  holds.}

	From \eqref{bian-1},  we have
	\begin{equation}\label{b-g-i-second}
	\begin{split}
	D_{m,k}
	&\geq 2\pi^2\min(\varepsilon,\delta) \int_0^{r_0}{\xi'}^2rdr+2\pi^2\min(\varepsilon,\delta) \int_0^{r_0}k^2\eta^2rdr\\
	&\quad+2\pi^2\min(\varepsilon,\delta) \int_0^{r_0}\Big(\frac{\xi}{r}+\frac{m}{r}\zeta\Big)^2rdr
	+2c\varepsilon\pi^2\int_0^{r_0}(\eta'+k\xi)^2rdr\\
	&\quad+2\varepsilon\pi^2\int_0^{r_0}\Big(\frac{m}{r}\eta-k\zeta\Big)^2rdr+2\pi^2\varepsilon\int_0^{r_0}\Big(-\zeta'+\frac{\zeta}{r}+\frac{m}{r}\xi\Big)^2rdr\\
	&\quad+\frac{2\pi^2}{3}\delta\int_0^{r_0}\Big(\xi'+\frac{\xi}{r}+\frac{m}{r}\zeta-k\eta\Big)^2rdr,
	\end{split}
	\end{equation}
	with small constant $c\leq 1$  (to be determined). When $|m|=1$, using that $2(a^2+b^2)\geq (a-mb)^2$ with $a=\frac{\xi}{r}+\frac{m}{r}\zeta$ and $b=-\zeta'+\frac{\zeta}{r}+\frac{m}{r}\xi$, we can get that
	\begin{equation}\label{bian-2}
	D_{m,k}\geq \pi^2\min(\varepsilon,\delta)\int_0^{r_0}\zeta'^2rdr.
	\end{equation}
The term $\int_0^{r_0}\Big(\frac{m}{r}\eta-k\zeta\Big)^2rdr$ can be estimated as 
	\begin{equation*}
	\begin{split}
	\int_0^{r_0}\Big(\frac{m}{r}\eta-k\zeta\Big)^2rdr&
	= \int_0^{r_0}\frac{1}{r}\eta^2dr+ \int_0^{r_0}k^2\zeta^2rdr-\int_0^{r_0}2m\eta k(\zeta+m\xi)dr+\int_0^{r_0}2\eta(\eta'+k\xi) dr\\
	&\quad-\int_0^{r_0}2\eta\eta'dr\\
	&= \int_0^{r_0}\frac{1}{r}\eta^2dr+ \int_0^{r_0}k^2\zeta^2rdr-\int_0^{r_0}2m\eta k(\zeta+m\xi)dr+\int_0^{r_0}2\eta(\eta'+k\xi) dr\\
	&\quad-\eta^2|_{r=r_0}+\eta^2(0).
	\end{split}
	\end{equation*}
	Note that 
	$$\Big|-\int_0^{r_0}2m\eta k(\zeta+m\xi)dr\Big|\leq  \int_0^{r_0}k^2\eta^2rdr+\int_0^{r_0}\frac{(\zeta+m\xi)^2}{r}dr,$$
	$$\Big|\int_0^{r_0}2\eta(\eta'+k\xi) dr\Big|\leq \frac{1}{2}\int_0^{r_0}\frac{1}{r}\eta^2dr+2\int_0^{r_0}(\eta'+k\xi)^2rdr.$$
	When $|m|=1$, we have
	$$\int_0^{r_0}\frac{(\zeta+m\xi)^2}{r}dr=\int_0^{r_0}\frac{(\xi+m\zeta)^2}{r}dr.$$
	Then we get
	\begin{equation*}
	\begin{split}
	2\varepsilon\pi^2 \int_0^{r_0}\Big(\frac{m}{r}\eta-k\zeta\Big)^2rdr&\geq \varepsilon\pi^2\int_0^{r_0}\frac{1}{r}\eta^2dr+ 2\varepsilon\pi^2\int_0^{r_0}k^2\zeta^2rdr\\
	&\quad -2\varepsilon\pi^2\int_0^{r_0}k^2\eta^2rdr-2\varepsilon\pi^2\int_0^{r_0}\frac{(\zeta+m\xi)^2}{r}dr\\
	&\quad-4\varepsilon\pi^2\int_0^{r_0}(\eta'+k\xi)^2rdr-2\varepsilon\pi^2\eta^2|_{r=r_0},
	\end{split}
	\end{equation*}
	which implies that
	\begin{equation}\label{O(1)-1}
		\begin{split}
		&2\varepsilon\pi^2 \int_0^{r_0}\Big(\frac{m}{r}\eta-k\zeta\Big)^2rdr
		\\&\quad\geq \varepsilon\pi^2\int_0^{r_0}\frac{1}{r}\eta^2dr+ 2\varepsilon\pi^2\int_0^{r_0}k^2\zeta^2rdr-2\varepsilon\pi^2\eta^2|_{r=r_0}-O(1)D_{m,k}.
		\end{split}
		\end{equation}
	We now estimate the term $c\varepsilon\pi^2\int_0^{r_0}(\eta'+k\xi)^2rdr$.
	Since the minimizer $(\xi,\eta,\zeta)\in Y_{m,k}$,  Definition \ref{defi-y} implies that $\xi(0)=0$. From Proposition \ref{embeddding-m1},  we know that $\int_{0}^{r_0}\eta'^2rdr$ and $\int_0^{r_0}\eta^2dr$ are bounded, which gives that $(\eta r)'\in L^2(0,r_0)$. So we get that $\eta r$ is well-defined at the origin $r=0$. Therefore, we prove the boundary term 
	$\eta k\xi r|_{r=0}=0.$
Similarly as \eqref{boundary-term-g},	we have
	\begin{equation*}
	\begin{split}
	&	c\int_0^{r_0}(\eta'+k\xi)^2rdr=c\int_0^{r_0}\eta'^2rdr+c\int_0^{r_0}k^2\xi^2rdr
	+2c\int_0^{r_0}\eta'k\xi rdr\\
	&=c\int_0^{r_0}\eta'^2rdr+c\int_0^{r_0}k^2\xi^2rdr+2c\eta k\xi r|_{r=r_0}
	-2c\int_0^{r_0}\eta k\xi' rdr-2c\int_0^{r_0}\eta k\xi dr.
	\end{split}
	\end{equation*}
	Notice that
	$$\Big|-2c\int_0^{r_0}\eta k\xi' rdr\Big|\leq c\int_0^{r_0}k^2\eta^2rdr+c\int_0^{r_0}\xi'^2rdr,$$
	$$\Big|-2c\int_0^{r_0}\eta k\xi dr\Big|\leq  2c\int_0^{r_0}\frac{1}{r}\eta^2dr+\frac{c}{2}\int_0^{r_0}k^2\xi^2rdr.$$
So we deduce \begin{equation*}
	\begin{split}
	c\varepsilon\pi^2\int_0^{r_0}(\eta'+k\xi)^2rdr&\geq c\varepsilon\pi^2\int_0^{r_0}\eta'^2rdr+\frac{c}{2}\varepsilon\pi^2\int_0^{r_0}k^2\xi^2rdr+2c\varepsilon\pi^2\eta k\xi r|_{r_0}\\
	&\quad-c\varepsilon\pi^2\int_0^{r_0}k^2\eta^2rdr-c\varepsilon\pi^2\int_0^{r_0}\xi'^2rdr
	-2\varepsilon\pi^2c\int_0^{r_0}\frac{1}{r}\eta^2dr,
	\end{split}
	\end{equation*}
	which gives that
	\begin{equation}\label{O(1)-2}
		\begin{split}
		c\varepsilon\pi^2\int_0^{r_0}(\eta'+k\xi)^2rdr&\geq c\varepsilon\pi^2\int_0^{r_0}\eta'^2rdr+\frac{c}{2}\varepsilon\pi^2\int_0^{r_0}k^2\xi^2rdr+2c\varepsilon\pi^2\eta k\xi r|_{r=r_0}
		\\
		&\quad	-2\varepsilon\pi^2c\int_0^{r_0}\frac{1}{r}\eta^2dr-O(1)D_{m,k}.
		\end{split}
		\end{equation}
	
	Choosing the constant $c$ small enough such that $c\leq \frac{1}{8}$, by \eqref{b-g-i-second}, \eqref{O(1)-1} and \eqref{O(1)-2}, we can show that
		\begin{equation}\label{O(1)-3}
		\begin{split}
		O(1)D_{m,k}&\geq \varepsilon\pi^2\int_0^{r_0}\frac{1}{r}\eta^2dr+ 2\varepsilon\pi^2\int_0^{r_0}k^2\zeta^2rdr+2c\varepsilon\pi^2\int_0^{r_0}\eta'^2rdr\\
		&\quad+c\varepsilon\pi^2\int_0^{r_0}k^2\xi^2rdr+2\pi^2\min(\varepsilon,\delta) \int_0^{r_0}{\xi'}^2rdr\\
		&\quad+2\pi^2\min(\varepsilon,\delta) \int_0^{r_0}\Big(\frac{\xi}{r}+\frac{m\zeta}{r}\Big)^2rdr-2\varepsilon\pi^2\eta^2|_{r=r_0}\\
		&\quad+2\pi^2\min(\varepsilon,\delta) \int_0^{r_0}k^2\eta^2rdr+4c\varepsilon\pi^2\eta k\xi r|_{r=r_0}.
		\end{split}
		\end{equation}
	
	The remaining thing is to deal with the boundary terms.
	
	By Lemma \ref{basic-cutoff} and Cauchy inequality, similarly as \eqref{boundary-term-1} and \eqref{boundary-term}, we can prove
		\begin{equation*}
		\begin{split}
		\Big|-2\varepsilon\pi^2\eta^2|_{r=r_0}\Big|
		\leq o(1)\int_0^{r_0}\eta'^2rdr
		+O(1)\mathcal{J},
		\end{split}
		\end{equation*}	
		\begin{equation*}
		\begin{split}
		\Big|4c\varepsilon\pi^2\eta k\xi r|_{r=r_0}\Big|
		\leq o(1)\Big(\int_0^{r_0}\eta'^2rdr+\int_0^{r_0}k^2\xi^2rdr\Big)+O(1)D_{m,k}.
		\end{split}
		\end{equation*}		
		Combining the above estimates with \eqref{bian-2} and \eqref{O(1)-3}, by $0\leq \rho \leq C$ and $\int_0^{r_0}\rho(\xi^2+\eta^2+\zeta^2)rdr=1$, we have
		\begin{equation*}
		\begin{split}
		O(1)D_{m,k}
		&\geq \varepsilon\pi^2\int_0^{r_0}\frac{1}{r}\eta^2dr+ 2\varepsilon\pi^2\int_0^{r_0}k^2\zeta^2rdr+\frac{c}{2}\varepsilon\pi^2\int_0^{r_0}k^2\xi^2rdr+c\varepsilon\pi^2\int_0^{r_0}\eta'^2rdr\\
		&\quad+2\pi^2\min(\varepsilon,\delta) \int_0^{r_0}{\xi'}^2rdr+\pi^2\min(\varepsilon,\delta) \int_0^{r_0}{\zeta'}^2rdr+2\pi^2\min(\varepsilon,\delta) \int_0^{r_0}k^2\eta^2rdr\\
		&\quad+2\pi^2\min(\varepsilon,\delta) \int_0^{r_0}\Big(\frac{\xi}{r}+\frac{m\zeta}{r}\Big)^2rdr \\
		&\geq c\varepsilon\pi^2\int_0^{r_0}\eta'^2rdr+2\pi^2\min(\varepsilon,\delta) \int_0^{r_0}{\xi'}^2rdr+\pi^2\min(\varepsilon,\delta) \int_0^{r_0}{\zeta'}^2rdr\\
		&\quad+2\varepsilon\pi^2\int_0^{r_0}k^2\zeta^2rdr+\frac{c}{2}\varepsilon\pi^2\int_0^{r_0}k^2\xi^2rdr+2\pi^2\min(\varepsilon,\delta) \int_0^{r_0}k^2\eta^2rdr\\
		&\quad+\varepsilon\pi^2\int_0^{r_0}\frac{1}{r}\eta^2dr+2\pi^2\min(\varepsilon,\delta) \int_0^{r_0}\Big(\frac{\xi}{r}+\frac{m\zeta}{r}\Big)^2rdr\\
		&\geq
		c\varepsilon\pi^2\int_0^{r_0}\eta'^2rdr+2\pi^2\min(\varepsilon,\delta) \int_0^{r_0}{\xi'}^2rdr+\pi^2\min(\varepsilon,\delta) \int_0^{r_0}{\zeta'}^2rdr\\
		 &\quad+\min\Big(\frac{2\varepsilon\pi^2}{C},\frac{c\varepsilon\pi^2}{2C},\frac{2\pi^2\min(\varepsilon,\delta)}{C}\Big)k^2\int_0^{r_0}\rho(\xi^2+\eta^2+\zeta^2)rdr\\
		 &\quad+\varepsilon\pi^2\int_0^{r_0}\frac{1}{r}\eta^2dr+2\pi^2\min(\varepsilon,\delta) \int_0^{r_0}\Big(\frac{\xi}{r}+\frac{m\zeta}{r}\Big)^2rdr\\
		&=
	c\varepsilon\pi^2\int_0^{r_0}\eta'^2rdr+2\pi^2\min(\varepsilon,\delta) \int_0^{r_0}{\xi'}^2rdr+\pi^2\min(\varepsilon,\delta) \int_0^{r_0}{\zeta'}^2rdr\\
		&\quad+\min\Big(\frac{c\varepsilon\pi^2}{2C},\frac{2\pi^2\min(\varepsilon,\delta)}{C}\Big)k^2+\varepsilon\pi^2\int_0^{r_0}\frac{1}{r}\eta^2dr\\
		&\quad+2\pi^2\min(\varepsilon,\delta) \int_0^{r_0}\Big(\frac{\xi}{r}+\frac{m\zeta}{r}\Big)^2rdr,
		\end{split}
		\end{equation*}
		which gives that  \eqref{m eq 1} holds 

{\bf Case (III): when $|m|\geq 2$, the estimate \eqref{m neq 1} holds.}

	From \eqref{bian-1}, we can get
	\begin{equation}\label{bian-guo-j}
	\begin{split}
	D_{m,k}
	&\geq 2\pi^2\min(\varepsilon,\delta) \int_0^{r_0}{\xi'}^2rdr+2\pi^2\min(\varepsilon,\delta) \int_0^{r_0}k^2\eta^2rdr\\
	&\quad+2\pi^2\min(\varepsilon,\delta) \int_0^{r_0}\Big(\frac{\xi}{r}+\frac{m}{r}\zeta\Big)^2rdr+2\varepsilon\pi^2\int_0^{r_0}(\eta'+k\xi)^2rdr\\
	&\quad+2\varepsilon\pi^2\int_0^{r_0}\Big(\frac{m}{r}\eta-k\zeta\Big)^2rdr+2c\varepsilon\pi^2\int_0^{r_0}\Big(-\zeta'+\frac{\zeta}{r}+\frac{m}{r}\xi\Big)^2rdr\\
	&\quad+\frac{2\pi^2}{3}\delta\int_0^{r_0}\Big(\xi'+\frac{\xi}{r}+\frac{m}{r}\zeta-k\eta\Big)^2rdr.
	\end{split}
	\end{equation}	
	
	We temporarily write $g=-\zeta'+\frac{\zeta}{r}+\frac{m}{r}\xi$, then we have $\xi=\frac{1}{m}(r\zeta'-\zeta+rg)$. Moreover, we get
	\begin{equation*}
	\begin{split}
	\int_0^{r_0}\Big(\frac{\xi}{r}+\frac{m}{r}\zeta\Big)^2rdr
	&=\int_0^{r_0}\frac{\xi^2}{r}dr+\int_0^{r_0}\frac{m^2}{r}\zeta^2dr+2\int_0^{r_0}\frac{m\xi\zeta}{r}dr\\
	&=\int_0^{r_0}\frac{\xi^2}{r}dr+\int_0^{r_0}\frac{m^2}{r}\zeta^2dr
	+2\int_0^{r_0}\zeta \Big(g+\zeta'-\frac{\zeta}{r}\Big)dr\\
	&=\int_0^{r_0}\frac{\xi^2}{r}dr+\int_0^{r_0}\frac{m^2}{r}\zeta^2dr
	+2\int_0^{r_0}\zeta g dr+ 2\int_0^{r_0}\zeta\zeta'dr -2\int_0^{r_0}\frac{\zeta^2}{r}dr\\
	&=\int_0^{r_0}\frac{\xi^2}{r}dr+\int_0^{r_0}\frac{m^2-2}{r}\zeta^2dr
	+2\int_0^{r_0}\zeta g dr+ \zeta^2(r_0),
	\end{split}
	\end{equation*}	
	\begin{equation*}
	\begin{split}
	\int_0^{r_0}\Big(\frac{\xi}{r}+\frac{m}{r}\zeta\Big)^2rdr
	&=\frac{1}{m^2}\int_0^{r_0}\frac{(r\zeta'-\zeta+rg)^2}{r}dr+\int_0^{r_0}\frac{m^2-2}{r}\zeta^2dr
	+2\int_0^{r_0}\zeta g dr+ \zeta^2(r_0)\\
	&=\frac{1}{m^2}\int_0^{r_0}\Big(\zeta'^2r+\frac{\zeta^2}{r}+g^2r+2(r\zeta'-\zeta)g-2\zeta\zeta'\Big)dr+\int_0^{r_0}\frac{m^2-2}{r}\zeta^2dr\\
	&\quad
	+2\int_0^{r_0}\zeta g dr+\zeta^2(r_0)\\
	&=\frac{1}{m^2}\int_0^{r_0}\Big(\zeta'^2r+\frac{\zeta^2}{r}+g^2r+2(r\zeta'-\zeta)g\Big)dr+\int_0^{r_0}\frac{m^2-2}{r}\zeta^2dr\\
	&\quad
	+2\int_0^{r_0}\zeta g dr+\Big(1-\frac{1}{m^2}\Big)\zeta^2(r_0),
	\end{split}
	\end{equation*}	
	which implies that
	\begin{equation*}
	\begin{split}
	2\int_0^{r_0}\Big(\frac{\xi}{r}+\frac{m}{r}\zeta\Big)^2rdr
	&=\int_0^{r_0}\frac{\xi^2}{r}dr+\int_0^{r_0}\frac{m^2}{r}\zeta^2dr
	+\int_0^{r_0}\frac{m^2-4}{r}\zeta^2dr\\
	&\quad+\frac{1}{m^2}\int_0^{r_0}\Big(\zeta'^2r+\frac{\zeta^2}{r}+g^2r\Big)dr+\frac{2}{m^2}\int_0^{r_0}\zeta'grdr\\
	&\quad-\frac{2}{m^2}\int_0^{r_0}\zeta gdr
	+4\int_0^{r_0}\zeta g dr+\Big(2-\frac{1}{m^2}\Big)\zeta^2(r_0).
	\end{split}
	\end{equation*}	
	Notice that
	\begin{equation*}
	\Big|\frac{2}{m^2}\int_0^{r_0}\zeta'grdr\Big|\leq \frac{1}{2m^2}\int_0^{r_0}\zeta'^2rdr+\frac{2}{m^2}\int_0^{r_0}g^2rdr,
	\end{equation*}
	\begin{equation*}
	\Big|-\frac{2}{m^2}\int_0^{r_0}\zeta gdr\Big|\leq \frac{1}{2m^2}\int_0^{r_0}\frac{\zeta^2}{r}dr+\frac{2}{m^2}\int_0^{r_0}g^2rdr,
	\end{equation*}
	\begin{equation*}
	\Big| 4\int_0^{r_0}\zeta g dr\Big|\leq \frac{1}{2}\int_0^{r_0}\frac{m^2\zeta^2}{r}dr+\frac{8}{m^2}\int_0^{r_0}g^2rdr.
	\end{equation*}
	Hence, we have
	\begin{equation}\label{bian-guo-1}
	\begin{split}
	&2\pi^2\min(\varepsilon,\delta)\int_0^{r_0}\Big(\frac{\xi}{r}+\frac{m}{r}\zeta\Big)^2rdr
	\\&\geq \pi^2\min(\varepsilon,\delta)\int_0^{r_0}\frac{\xi^2}{r}dr+\frac{1}{2}\pi^2\min(\varepsilon,\delta)\int_0^{r_0}\frac{m^2}{r}\zeta^2dr
	\\
	&\quad+\pi^2\min(\varepsilon,\delta)\int_0^{r_0}\frac{m^2-4}{r}\zeta^2dr+\frac{1}{2m^2}\pi^2\min(\varepsilon,\delta)\int_0^{r_0}\Big(\zeta'^2r+\frac{\zeta^2}{r}\Big)dr\\
	&\quad-\frac{11}{m^2}\pi^2\min(\varepsilon,\delta)\int_0^{r_0}g^2rdr
	+\Big(2-\frac{1}{m^2}\Big)\pi^2\min(\varepsilon,\delta)\zeta^2(r_0).
	\end{split}
	\end{equation}	
	We can estimate the term $2\varepsilon\pi^2\int_0^{r_0}\Big(\frac{m}{r}\eta-k\zeta\Big)^2rdr$ as follows
	\begin{equation*}
	2\varepsilon\pi^2\int_0^{r_0}\Big(\frac{m}{r}\eta-k\zeta\Big)^2rdr= 2\varepsilon\pi^2\int_0^{r_0}\frac{m^2}{r}\eta^2dr+2\varepsilon\pi^2 \int_0^{r_0}k^2\zeta^2rdr-2\varepsilon\pi^2\int_0^{r_0}2m\eta k\zeta dr.
	\end{equation*}
	Notice that
		$$\Big|-\int_0^{r_0}2m\eta k\zeta dr\Big|\leq \int_0^{r_0} k^2\eta^2r dr+\int_0^{r_0}\frac{m^2}{r}\zeta^2dr,$$
	which implies that
\begin{equation}\label{O(1)-1-general-m}
		\begin{split}
		2\varepsilon\pi^2\int_0^{r_0}\Big(\frac{m}{r}\eta-k\zeta\Big)^2rdr&\geq 2\varepsilon\pi^2\int_0^{r_0}\frac{m^2}{r}\eta^2dr+ 2\varepsilon\pi^2\int_0^{r_0}k^2\zeta^2rdr\\
		&\quad
		-O(1)\Big(\int_0^{r_0} k^2\eta^2r dr+\int_0^{r_0}\frac{m^2}{r}\zeta^2dr\Big).
		\end{split}
		\end{equation}
	We now estimate the term $2\varepsilon\pi^2\int_0^{r_0}(\eta'+k\xi)^2rdr$.
	Similarly as \eqref{boundary-term-g}, we can get
	\begin{equation*}
	\begin{split}
	\int_0^{r_0}(\eta'+k\xi)^2rdr
	=\int_0^{r_0}\eta'^2rdr+\int_0^{r_0}k^2\xi^2rdr+2\eta k\xi r|_{r=r_0}
	-2\int_0^{r_0}\eta k\xi' rdr-2\int_0^{r_0}\eta k\xi dr.
	\end{split}
	\end{equation*}
	Note that
	$$\Big|-2\int_0^{r_0}\eta k\xi' rdr\Big|\leq \int_0^{r_0}k^2\eta^2rdr+\int_0^{r_0}\xi'^2rdr,$$
$$\Big|-2\int_0^{r_0}\eta k\xi dr\Big|\leq  \int_0^{r_0}k^2\eta^2dr+\int_0^{r_0}\frac{\xi^2}{r}dr.$$	
Hence we have \begin{equation}\label{O(1)-2-general-m}
		\begin{split}
		2\varepsilon\pi^2\int_0^{r_0}(\eta'+k\xi)^2rdr&\geq 2\varepsilon\pi^2\int_0^{r_0}\eta'^2rdr+2\varepsilon\pi^2\int_0^{r_0}k^2\xi^2rdr+4\varepsilon\pi^2\eta k\xi r|_{r=r_0}\\
		&\quad-O(1)\Big(\int_0^{r_0}k^2\eta^2rdr+\int_0^{r_0}\xi'^2rdr+\int_0^{r_0}\frac{\xi^2}{r}dr\Big).
		\end{split}
		\end{equation}
	
	On the other hand, since the minimizer $(\xi,\eta,\zeta)\in Y_{m,k}$,  Definition \ref{defi-y} implies that $\xi(0)=\zeta(0)=0$. Therefore, we get the boundary terms 
	$m \zeta \xi|_{r=0}=\zeta^2|_{r=0}=0,$ which gives that
	\begin{equation*}
	\begin{split}
	\int_0^{r_0}\Big(-\zeta'+\frac{\zeta}{r}+\frac{m}{r}\xi\Big)^2rdr&=\int_0^{r_0}\Big(\zeta'^2r+\frac{\zeta^2}{r}+\frac{m^2\xi^2}{r}\Big)dr-2\int_0^{r_0}\zeta'\zeta dr\\
	&\quad-2\int_0^{r_0}\zeta'm\xi dr+2\int_0^{r_0}\frac{m\xi\zeta}{r}dr\\
	&=\int_0^{r_0}\Big(\zeta'^2r+\frac{\zeta^2}{r}+\frac{m^2\xi^2}{r}\Big)dr-\zeta^2(r_0)\\
	&\quad-2m\zeta\xi|_{r=r_0}+2\int_0^{r_0}\zeta m\xi' dr+2\int_0^{r_0}\frac{m\xi\zeta}{r}dr.
	\end{split}
	\end{equation*}
	Since 
		\begin{equation*}
		\Big|2\int_0^{r_0}\zeta m\xi' dr\Big|\leq \int_0^{r_0}\xi'^2rdr+ \int_0^{r_0}\frac{m^2\zeta^2}{r}dr,
		\end{equation*}
		\begin{equation*}
		\Big|2\int_0^{r_0}\frac{m\xi\zeta}{r}dr\Big|\leq \int_0^{r_0}\frac{\xi^2}{r}dr+ \int_0^{r_0}\frac{m^2\zeta^2}{r}dr,
		\end{equation*}
		we can show that
		\begin{equation*}
		\begin{split}
		2\varepsilon\pi^2\int_0^{r_0}\Big(-\zeta'+\frac{\zeta}{r}+\frac{m}{r}\xi\Big)^2rdr
		&\geq 2\varepsilon\pi^2\int_0^{r_0}\Big(\zeta'^2r+\frac{\zeta^2}{r}+\frac{m^2\xi^2}{r}\Big)dr-2\varepsilon\pi^2\zeta^2(r_0)\\
		&\quad-4\varepsilon\pi^2m\zeta\xi|_{r=r_0}-O(1)\int_0^{r_0}\xi'^2rdr\\
		&\quad-O(1)\int_0^{r_0}\frac{\xi^2}{r}dr
		-O(1)\int_0^{r_0}\frac{m^2\zeta^2}{r}dr,
		\end{split}
		\end{equation*}
		which together with \eqref{bian-guo-j}, \eqref{bian-guo-1}, \eqref{O(1)-1-general-m} and \eqref{O(1)-2-general-m}, gives that for $|m|\geq 2$
		\begin{equation}\label{bian-guo-j2}
		\begin{split}
		O(1)D_{m,k} 
		&\geq 2\pi^2\min(\varepsilon,\delta) \int_0^{r_0}{\xi'}^2rdr+2\pi^2\min(\varepsilon,\delta) \int_0^{r_0}k^2\eta^2rdr\\
		&\quad+ \pi^2\min(\varepsilon,\delta)\int_0^{r_0}\frac{\xi^2}{r}dr+\frac{1}{2}\pi^2\min(\varepsilon,\delta)\int_0^{r_0}\frac{m^2}{r}\zeta^2dr
		\\
		&\quad+\pi^2\min(\varepsilon,\delta)\int_0^{r_0}\frac{m^2-4}{r}\zeta^2dr+\frac{\pi^2}{2m^2}\min(\varepsilon,\delta)\int_0^{r_0}\Big(\zeta'^2r+\frac{\zeta^2}{r}\Big)dr\\
		&\quad
		+2\varepsilon\pi^2\int_0^{r_0}\frac{m^2}{r}\eta^2dr+ 2\varepsilon\pi^2\int_0^{r_0}k^2\zeta^2rdr+2\varepsilon\pi^2\int_0^{r_0}\eta'^2rdr\\
		&\quad+2\varepsilon\pi^2\int_0^{r_0}k^2\xi^2rdr+4\varepsilon\pi^2\eta k\xi r|_{r=r_0}-4\varepsilon\pi^2m\zeta\xi|_{r=r_0}\\
		&\quad+2\varepsilon\pi^2\int_0^{r_0}\Big(\zeta'^2r+\frac{\zeta^2}{r}+\frac{m^2\xi^2}{r}\Big)dr-2\varepsilon\pi^2\zeta^2(r_0).
		\end{split}
		\end{equation}
		The remaining thing is to deal with the boundary terms $4\varepsilon\pi^2\eta k\xi r|_{r=r_0} $, $-4c\varepsilon\pi^2m\zeta\xi|_{r=r_0}$ and $2\varepsilon\pi^2\zeta^2(r_0)$.
		Similarly as \eqref{boundary-term}, by Cauchy inequality and \eqref{first} in Lemma \ref{basic-cutoff}, it follows from \eqref{bian-guo-j} and \eqref{bian-guo-1} that
		\begin{equation*}
		\begin{split}
		\Big|4\varepsilon\pi^2\eta k\xi r|_{r=r_0}\Big|
		\leq o(1)\Big( \int_0^{r_0}\eta'^2rdr+\int_0^{r_0}k^2\xi^2rdr\Big)
		+O(1)D_{m,k},
		\end{split}
		\end{equation*}
		\begin{equation*}
		\begin{split}
		\Big|4\varepsilon\pi^2m\zeta \xi |_{r=r_0}\Big|
		\leq o(1)\Big(\int_0^{r_0}\zeta'^2rdr+\int_0^{r_0}\frac{m^2\xi^2}{r}dr\Big)
		+O(1)D_{m,k},
		\end{split}
		\end{equation*}
		where
		we have used the facts that for any function $h$ and $|m|\geq 2$, it holds that $$\int_{\frac{r_0}{2}}^{r_0}h^2dr\leq r_0\int_0^{r_0}\frac{m^2h^2}{r}dr, \quad\int_{\frac{r_0}{2}}^{r_0}h^2dr\leq \frac{2}{r_0}\int_0^{r_0}h^2rdr.$$
		By \eqref{second} in Lemma \ref{basic-cutoff}, the same estimate in \eqref{boundary-term-1}, implies that
	\begin{equation*}
	\Big|-2\varepsilon\pi^2\zeta^2|_{r=r_0}\Big|
	\leq 	o(1)\int_0^{r_0}\zeta'^2rdr
	+O(1)\mathcal{J}.
	\end{equation*}
	Then it follows from \eqref{bian-guo-j2}, $0\leq \rho\leq C$ and  $\int_0^{r_0}\rho(\xi^2+\eta^2+\zeta^2)rdr=1$ that
		\begin{equation*}
		\begin{split}
		O(1)D_{m,k} 
		&
		\geq 2\pi^2\min(\varepsilon,\delta) \int_0^{r_0}{\xi'}^2rdr+\varepsilon\pi^2\int_0^{r_0}\eta'^2rdr+\varepsilon\pi^2\int_0^{r_0}\zeta'^2rdr-O(1)\mathcal{J}\\
		&\quad+2\pi^2\min(\varepsilon,\delta) \int_0^{r_0}k^2\eta^2rdr+ 2\varepsilon\pi^2\int_0^{r_0}k^2\zeta^2rdr+\varepsilon\pi^2\int_0^{r_0}k^2\xi^2rdr\\
		&\quad+\frac{1}{2}\pi^2\min(\varepsilon,\delta)\int_0^{r_0}\frac{m^2}{r}\zeta^2dr+2\varepsilon\pi^2\int_0^{r_0}\frac{m^2}{r}\eta^2dr+\varepsilon\pi^2\int_0^{r_0}\frac{m^2\xi^2}{r}dr\\
		&\geq 2\pi^2\min(\varepsilon,\delta) \int_0^{r_0}{\xi'}^2rdr+\varepsilon\pi^2\int_0^{r_0}\eta'^2rdr+\varepsilon\pi^2\int_0^{r_0}\zeta'^2rdr-O(1)\\
		&\quad+\min\Big(\frac{2\pi^2\delta}{C},  \frac{\pi^2\varepsilon}{C}\Big)k^2
		+\frac{\pi^2\min(\varepsilon,\delta)}{2r_0^2C}m^2,
		\end{split}
		\end{equation*}
		which implies that \eqref{m neq 1} holds.
	Therefore, the result is proved from case (I), case  (II) and case (III).
\end{proof}

Next, we show the existence of the biggest growing mode for any $m$ and $k$.

\begin{prop}\label{growing-mode-b2}
There exists the biggest growing mode $\mu$ for any $(m, k)\in \mathbb{Z}\times \mathbb{Z}$.
\end{prop}

\begin{proof}	
The proof is divided into the following steps.
			
	{\bf  Step 1}:  
Assume that when $m=0$, $\inf_{(\xi,\eta,\zeta)\in \mathcal{A}_1} E_{0,k}(\xi,\eta,\zeta)<0$, 
and when $m\neq 0$, $\inf_{((\xi,\eta,\zeta),\widehat{ Q}_r)\in \mathcal{A}_2} E_{m,k}(\xi,\eta,\zeta, \widehat{ Q}_r;s_{m,k})<0$ for any fixed large $k$, any fixed $m$ and small enough $s_{m,k}$. 
	
	{\bf  Step 2}: 
By Proposition \ref{infimum-A} and Proposition \ref{infimum-A-out-viscosity}, then a minimizer exists, which ensures that, for any fixed $m$ and $k$, we can define the function $\lambda: (0,\infty)\rightarrow \mathbb{R}$ by 
\begin{equation*}
\lambda(0,k):=\inf_{(\xi,\eta,\zeta)\in \mathcal{A}_1}E_{0,k}(\xi,\eta,\zeta;s_k), \,\, \mbox{when}\, \,m=0,
\end{equation*}
	\begin{equation*}
	\lambda(m,k):=\inf_{((\xi,\eta,\zeta),\widehat{ Q}_r)\in \mathcal{A}_2}E_{m,k}(\xi,\eta,\zeta,\widehat{ Q}_r;s_{m,k}),\,\, \mbox{when}\,\, m\neq 0.
	\end{equation*}

	{\bf Step 3}:  Let $\lambda(m,k)=-\mu^2(m,k)$ for any $m$ and $k$, then from Step 1 and Step 2, we get $\mu(m,k)>0$ for any fixed  $m$ and $k$.  By Proposition \ref{prop3.7} and Proposition \ref{prop3.14}, it holds that there exists a unique $s_{m,k}$ such that
	$\mu(m,k)=\sqrt{-\lambda(m,k)}=s_{m,k}$,
	for any fixed $m$ and $k$.
		
	{\bf Step 4}: We will use a contradiction  to show that 
	\begin{equation}\label{decay-zero-m=2}
\mu(m,k)=s_{m,k}\rightarrow 0, \quad \mbox{as} \quad m \rightarrow \infty \quad\mbox{or} \quad k\rightarrow\infty.
	\end{equation}
 Suppose 
	\begin{equation}\label{sk-bound-general-m}
	\limsup_{m \rightarrow \infty\, \mbox{or}\,k\rightarrow\infty} s_{m,k}>0.
	\end{equation}
	It is convenient to decompose $E_{m,k}$ according to that 
	when $m=0$, \begin{equation*}
	E_{0,k}(\xi,\eta,\zeta; s_k)=\widetilde{E}_{0,k}^0+\widetilde{E}_{0,k}^1+\widetilde{E}_{0,k}^2=\widetilde{E}_{0,k}^0+\widetilde{E}_{0,k}^1+s_{0,k} D_{0,k},
	\end{equation*}
	with 
	\begin{equation*}
	\widetilde{E}_{0,k}^0=2\pi^2\int_0^{r_0}\Big[\frac{2p'}{r}+\frac{4\gamma p B_{\theta}^2}{r^2(\gamma p+B_{\theta}^2)}\Big]\xi^2rdr,
	\end{equation*}
	\begin{equation*}
	\begin{split}
	\widetilde{E}_{0,k}^1=2\pi^2\int_0^{r_0}(\gamma p+B_{\theta}^2)\Big[k\eta-\frac{1}{r}\Big((r\xi)'-\frac{2B_{\theta}^2}{\gamma p +B_{\theta}^2}\xi\Big)\Big]^2rdr,
	\end{split}
	\end{equation*}
	and	when $m\neq 0$,
	\begin{equation*}
	E_{m,k}(\xi,\eta,\zeta,\widehat{ Q}_r; s_{m,k})=\widetilde{E}_{m,k}^0+\widetilde{E}_{m,k}^1+\widetilde{E}_{m,k}^2=\widetilde{E}_{m,k}^0+\widetilde{E}_{m,k}^1+s_{m,k} D_{m,k},
	\end{equation*}
	with 
	\begin{equation*}
	\widetilde{E}_{m,k}^0=2\pi^2\int_0^{r_0}\Big[2p'+\frac{m^2B_{\theta}^2}{r}\Big]\xi^2dr,
	\end{equation*}
	\begin{equation*}
	\begin{split}
	\widetilde{E}_{m,k}^1&=2\pi^2\int_0^{r_0}\Big\{(m^2+k^2r^2)\Big[\frac{B_\theta}{r}\eta+\frac{-kB_\theta(r\xi)'+2kB_{\theta}\xi}{m^2+k^2r^2}\Big]^2+\gamma p\Big[\frac{1}{r}(r\xi)'-k\eta+\frac{m\zeta}{r}\Big]^2\Big\}rdr
	\\
	&\quad+2\pi^2\int_0^{r_0}\frac{m^2B_\theta^2}{r(m^2+k^2r^2)}(\xi-r\xi')^2dr
	+2\pi^2\int_{r_0}^{r_w}\bigg[|\widehat{Q}_r|^2+\frac{1}{m^2+k^2r^2}|(r\widehat{Q}_r)'|^2\bigg]r
	dr.
	\end{split}
	\end{equation*}
	Similarly as \eqref{bian-1}, we can get that for any $(\xi,\eta,\zeta)\in \mathcal{A}_1$,
	\begin{equation}\label{con-mono-use}
	\begin{split}
	&E_{0,k}(\xi,\eta,\zeta;s_k)
	\\
	&
	\geq 2\pi^2s_{0,k}\min(\varepsilon,\delta)\int_0^{r_0}{\xi'}^2rdr+2\pi^2s_{0,k}\min(\varepsilon,\delta)\int_0^{r_0}\frac{\xi^2}{r}dr+4\pi^2\int_0^{r_0}p'\xi^2dr,
	\end{split}
	\end{equation}	
	and 	for any $((\xi,\eta,\zeta),\widehat{  Q}_r)\in \mathcal{A}_2$,
	\begin{equation}
	\begin{split}
	E_{m,k}(\xi,\eta,\zeta,\widehat{ Q}_r;s_{m,k})
	\geq 2\pi^2s_{m,k}\min(\varepsilon,\delta)\int_0^{r_0}{\xi'}^2rdr+4\pi^2\int_0^{r_0}p'\xi^2dr.
	\end{split}
	\end{equation}	

	By Lemma \ref{xi-bound}, choosing $0<r_1<\frac{r_0}{3}$ small enough  such that
	$Cr_1\leq \frac{1}{4}$, similarly as \eqref{estimate-p'}, we deduce
	\begin{equation*}
	\begin{split}
	\Big|\int_0^{r_1}p'\xi^2dr\Big|
	\leq C\mathcal{J}r_1
	+\frac{1}{2}\pi^2s_k\min(\varepsilon,\delta)\int_0^{r_0}{\xi'}^2rdr.
	\end{split}
	\end{equation*}	

On the other hand, from the Definition of $\mathcal{J}$ and Lemma \ref{xi-bound-m-general},  choosing  $\frac{r_0}{3}<r_2<r_0$ close enough to $r_0$ such that
$C(r_0-r_2)\leq \frac{1}{4}$, similarly as \eqref{estimates-p'-3} and \eqref{xi-L^2-boundary}, we can show that
\begin{equation*}
\Big|\int_{r_1}^{r_2}p'\xi^2dr\Big|\leq C\int_{r_1}^{r_2} \xi^2dr\leq C\mathcal{J},
\end{equation*}
\begin{equation*}
\begin{split}
\int_{r_2}^{r_0}\xi^2(r)dr
\leq C\mathcal{J}(r_0-r_2)
+\frac{1}{2}\pi^2s_{m,k}\min(\varepsilon,\delta)\int_0^{r_0}{\xi'}^2rdr.
\end{split}
\end{equation*}	
Hence, we can get 
\begin{equation*}
\Big|\int_{0}^{r_0}p'\xi^2dr\Big|\leq C(\mathcal{J})+C,
\end{equation*}
which  together with Step 1 gives that for any $m$ and $k$, $|\widetilde{E}_{m,k}^0|\leq C(\mathcal{J})+C$.		
	Since $E_{0,k}(\xi,\eta,\zeta; s_{0,k})<0$ and $E_{m,k}(\xi,\eta,\zeta,\widehat{ Q}_r; s_{m,k})<0$,  and for any $m$ and $k$, it holds that $\widetilde{E}_{m,k}^1\geq 0$,
	we get that
	$0\leq s_{m,k}D_{m,k}\leq C_0+C$ for any $m$ and $k$,
	which gives that
	\begin{equation*}
	\begin{split}
	s_{m,k}\leq \frac{C_0+C}{D_{m,k}},
	\end{split}
	\end{equation*}
	for any $m$ and $k$.
Hence, we can get from Propostion \ref{new-prop-bian-guo} that when $|m|\neq 1$,
	$s_{m,k}\leq \frac{C_0+C}{\|(\xi',\eta',\zeta')\sqrt{r}\|_{L^2}^2+k^2+m^2-O(1)},$ and when $|m|=1$, $s_{m,k}\leq \frac{C_0+C}{\|(\xi',\eta',\zeta')\sqrt{r}\|_{L^2}^2+k^2-O(1)}.$
	This shows that $$\limsup_{m \rightarrow \infty\, \mbox{or}\,k\rightarrow\infty} s_{m,k}=0,$$
	which contradict with \eqref{sk-bound-general-m}.  This ensures \eqref{decay-zero-m=2}. 
The existence of growing mode $s_{m,k}>0$ requires that $m$ and $k$ are finite. By Proposition \ref{infimum-A} and Proposition \ref{infimum-A-out-viscosity}, for any fixed $m$ and $k$, the corresponding  energy achieves  its negative infimum with a growing mode, therefore, we can choose the biggest growing mode for these finitely many $m$ and $k$. This proves the result. 
\end{proof}
\section{Growth of solutions to the linearized problem}
In this section, we will prove estimates for the growth in time of arbitrary solutions to \eqref{spectral-formulation-orig} in terms of the biggest growing mode $\Lambda$ for any $m$ and $k$. 
In fact, from Proposition  \ref{growing-mode-b2},  we know that there exists the biggest growing mode for any  $m$ and $k$, so we can define the biggest growing mode as follows
\begin{equation}\label{biggest-mode}
\Lambda^2:=\sup_{m,k\in \mathbb{Z}}(-\lambda_{m,k}) .
\end{equation}

\subsection{Estimates  about $\xi$, $\eta$ and $\zeta$}
In this subsection, let us introduce the following two lemmas describing the basic estimates about $\xi$, $\eta$ and $\zeta$.
\begin{lem}\label{estimate3-lem}
Assume $(\xi,\eta,\zeta)$ is the $H^2$ solution of the system 
 \begin{equation} \label{spectal-m=1} 
 \begin{split}
 &\left(
 \begin{array}{ccc}
 \frac{d}{dr}\frac{\gamma p+B_{\theta}^2}{r}\frac{d}{dr}r-\frac{m^2}{r^2}B_{\theta}^2
 -r(\frac{B_{\theta}^2}{r^2})'&-\frac{d}{dr}k(\gamma p+B_{\theta}^2)-\frac{2kB_{\theta}^2}{r}&\frac{d}{dr}\frac{m}{r}\gamma p\\
 \frac{k(\gamma p+B_{\theta}^2)}{r}\frac{d}{dr}r-\frac{2kB_{\theta}^2}{r}&
 -k^2(\gamma p+B_{\theta}^2)-\frac{m^2}{r^2}B_{\theta}^2&\frac{mk}{r}\gamma p\\
 -\frac{m\gamma p}{r^2}\frac{d}{dr}r&\frac{mk}{r}\gamma p&-\frac{m^2}{r^2}\gamma p
 \end{array}
 \right)
 \left(
 \begin{array}{lll}
 \xi_t   \\
 \eta_t \\
 \zeta_t\\
 \end{array}
 \right)+\\
 &	\left(
 \begin{array}{ccc}
 d_{11}&d_{12}&d_{13}\\
 d_{21}&
 d_{22}&d_{23}\\
 d_{31}&d_{32}&d_{33}
 \end{array}
 \right)
 \left(
 \begin{array}{lll}
 \xi_{tt}   \\
 \eta_{tt} \\
 \zeta_{tt}\\
 \end{array}
 \right)	
 =\rho  \left(
 \begin{array}{lll}
 \xi_{ttt}   \\
 \eta_{ttt} \\
 \zeta_{ttt}\\
 \end{array}
 \right),
 \end{split}
 \end{equation}
 with 
 \begin{equation}
 \begin{split}
 &d_{11}=	(\frac{4\varepsilon}{3}+\delta)\frac{d^2}{dr^2}+(\frac{4\varepsilon}{3}+\delta)\frac{d}{dr}\frac{1}{r}-\varepsilon\Big(\frac{m^2}{r^2}+k^2\Big),\quad d_{12}=-(\frac{\varepsilon}{3}+\delta)k\frac{d}{dr},\\
 &d_{13}=(\frac{\varepsilon}{3}+\delta)\frac{d}{dr}\frac{m}{r}
 -\frac{2\varepsilon m}{r^2},\quad d_{21}=(\frac{\varepsilon}{3}+\delta)k\frac{d}{dr}+(\frac{\varepsilon}{3}+\delta)\frac{k}{r},\\
 &d_{22}=\varepsilon\frac{d^2}{dr^2}+\frac{\varepsilon}{r}\frac{d}{dr}
 -\varepsilon \frac{m^2}{r^2}-(\frac{4\varepsilon}{3}+\delta)k^2,\quad
 d_{23}=(\frac{\varepsilon}{3}+\delta)\frac{mk}{r},\\
 &d_{31}=-(\frac{\varepsilon}{3}+\delta)\frac{m}{r}\frac{d}{dr}-(\frac{7\varepsilon}{3}+\delta)\frac{m}{r^2},\quad d_{32}=(\frac{\varepsilon}{3}+\delta)\frac{mk}{r},\\
 &d_{33}=\varepsilon\frac{d^2}{dr^2}+\varepsilon\frac{d}{dr}\frac{1}{r}-(\frac{4\varepsilon}{3}+\delta)\frac{m^2}{r^2}-\varepsilon k^2,
 \end{split}
 \end{equation}
  and the function $\widehat{ Q}_r$ is the $H^2$ solution of the equation
 \begin{equation}\label{spectral-proof-2}
 \bigg[\frac{r}{m^2+k^2r^2}(r\widehat{Q}_{rt})'\bigg]'-\widehat{Q}_{rt}=0, 
 \end{equation}  
  with the other two components $\widehat{ Q}_{\theta t}=-\frac{m}{m^2+k^2r^2}(r\widehat{ Q}_{rt})'$ and
$\widehat{ Q}_{z t}=-\frac{kr}{m^2+k^2r^2}(r\widehat{ Q}_{rt})'$, 
along with the boundary conditions 
 \begin{equation}\label{spectral-proof-3b}
 \widehat{ Q}_{rtt}=0,  \quad \mbox{at} \quad r=r_w, 
 \end{equation}	
\begin{equation}\label{spectral-proof-3}
  m\widehat{  B}_\theta \xi_{tt}=r\widehat{ Q}_{rtt}, \quad \mbox{at} \quad r=r_0, 
 \end{equation}
 \begin{equation}\label{spectral-proof-4}
 \begin{split}
 &\Big[B_{\theta}^2\xi_t-B_{\theta}^2\xi_t' r+kB_\theta^2\eta_t r-\widehat{ B}_{\theta}\widehat{Q}_{\theta t}r\Big]n\\
 &\quad-\varepsilon\Big(2\xi'_{tt}r,-i\zeta'_{tt}r+im\xi_{tt} +i\zeta_{tt}, i\eta'_{tt}r+ik\xi_{tt} r\Big)^T\\
 &\quad-(\delta-\frac{2}{3}\varepsilon)\Big[\xi'_{tt}r+\xi_{tt}+m\zeta_{tt} -k\eta_{tt} r\Big]n=0, \quad \mbox{at} \quad r=r_0.
 \end{split}
 \end{equation}
 Then it holds that 
\begin{equation}\label{esimate3-growth}
\begin{split}
&\frac{1}{2}\frac{d}{dt}\int_0^{r_0} \rho (|\xi_{tt}|^2+|\eta_{tt}|^2+|\zeta_{tt}|^2)rdr+\int_0^{r_0}\varepsilon\Big[\frac{2}{9}\Big(-2\xi'_{tt}+\frac{\xi_{tt}}{r}+\frac{m}{r}\zeta_{tt}-k\eta_{tt}\Big)^2\\
&\qquad+\frac{2}{9}\Big(\xi'_{tt}-\frac{2\xi_{tt}}{r}-\frac{2m}{r}\zeta_{tt}-k\eta_{tt}\Big)^2+\frac{2}{9}\Big(\xi'_{tt}+\frac{\xi_{tt}}{r}+\frac{m}{r}\zeta_{tt}+2k\eta_{tt}\Big)^2\\
&\qquad+\Big(-\zeta'_{tt}+\frac{\zeta_{tt}}{r}+\frac{m}{r}\xi_{tt}\Big)^2+(\eta'_{tt}+k\xi_{tt})^2+\Big(\frac{m}{r}\eta_{tt}-k\zeta_{tt}\Big)^2\Big]rdr\\
&\qquad+\int_0^{r_0}\delta\Big(\xi'_{tt}+\frac{\xi_{tt}}{r}+\frac{m}{r}\zeta_{tt}-k\eta_{tt}\Big)^2rdr\\
&\quad=-\frac{1}{2}\frac{d}{dt}\bigg\{\int_0^{r_0}\Big\{(m^2+k^2r^2)\Big[\frac{B_\theta}{r}\eta_t+\frac{-kB_\theta(r\xi_t)'+2kB_{\theta}\xi_t}{m^2+k^2r^2}\Big]^2\\
&\qquad+\gamma p\Big[\frac{1}{r}(r\xi_t)'-k\eta_t+\frac{m\zeta_t}{r}\Big]^2\Big\}rdr+\int_0^{r_0}\frac{m^2B_\theta^2}{r(m^2+k^2r^2)}(\xi_t-r\xi'_t)^2dr\\
&\qquad+\int_0^{r_0}\Big[2p'+\frac{m^2B_{\theta}^2}{r}\Big]\xi_t^2dr+\int_{r_0}^{r_w}\Big[|\widehat{Q}_{rt}|^2+\frac{1}{m^2+k^2r^2}|(r\widehat{Q}_{rt})'|^2\Big]rdr\bigg\}.
\end{split}
\end{equation}
\end{lem}
\begin{proof}
Multiplying the first equation of \eqref{spectal-m=1} by $r\xi_{tt}$, the second one by $r\eta_{tt}$ and the third one by $r\zeta_{tt}$, integrating by parts over $(0,r_0)$, we get 
\begin{equation*}
\begin{split}
&[\xi_{tt} (\gamma p+B_{\theta}^2)(r\xi_t)']_{r=0}^{r=r_0}-[k(\gamma p+B_{\theta}^2)\xi_{tt}\eta_t r]_{r=0}^{r=r_0}+[m\gamma p\xi_{tt} \zeta_t  ]_{r=0}^{r=r_0}\\
&-\frac{1}{2}\frac{d}{dt}\int_0^{r_0}\frac{\gamma p+B_{\theta}^2}{r}[(r\xi_t)']^2 dr+\frac{d}{dt}\int_0^{r_0}k(\gamma p+B_{\theta}^2)\eta_t(r\xi_t)' dr-\frac{d}{dt}\int_0^{r_0}\frac{m}{r}\gamma p\zeta_t (r\xi_t)' dr\\
&-\frac{1}{2}\frac{d}{dt}\int_0^{r_0}\frac{m^2}{r}B_{\theta}^2\xi_t^2 dr-\frac{1}{2}\frac{d}{dt}\int_0^{r_0}\left(\frac{B_{\theta}^2}{r^2}\right)'(r\xi_t)^2 dr
-\frac{d}{dt} \int_0^{r_0}2kB_{\theta}^2\xi_t\eta_t dr\\
&-\frac{1}{2}\frac{d}{dt}\int_0^{r_0}k^2(\gamma p+B_{\theta}^2)\eta_t^2 r dr-\frac{1}{2}\frac{d}{dt}\int_0^{r_0}\frac{m^2}{r}B_{\theta}^2\eta_t^2 dr+\frac{d}{dt}\int_0^{r_0}mk\gamma p\eta_t \zeta_t dr\\
&-\frac{1}{2}\frac{d}{dt}\int_0^{r_0}\frac{m^2}{r}\gamma p\zeta_t^2 dr+\Big[\Big(\frac{4\varepsilon}{3}+\delta\Big)\xi_{tt}' \xi_{tt} r\Big]_{r=0}^{r=r_0}+[\varepsilon\eta_{tt}'\eta_{tt} r]_{r=0}^{r=r_0}+\varepsilon[\zeta_{tt}'\zeta_{tt} r-\zeta_{tt}^2]_{r=0}^{r=r_0}\\
&-\int_0^{r_0}\Big(\frac{4\varepsilon}{3}+\delta\Big)\xi_{tt}'^2rdr-\int_0^{r_0}\Big(\frac{4\varepsilon}{3}+\delta\Big)\xi'_{tt}\xi_{tt} dr+\int_0^{r_0}\Big(\frac{4\varepsilon}{3}+\delta\Big)\xi_{tt}'\xi_{tt} dr\\
&-\int_0^{r_0}\Big(\frac{4\varepsilon}{3}+\delta\Big)\frac{\xi_{tt}^2}{r}dr
-\int_0^{r_0}\varepsilon k^2\xi_{tt}^2rdr-\int_0^{r_0}\Big(\frac{\varepsilon}{3}+\delta\Big)
k\eta_{tt}'\xi_{tt} rdr+\int_0^{r_0}\Big(\frac{\varepsilon}{3}+\delta\Big)
k\xi'_{tt}\eta_{tt} rdr\\
&+\int_0^{r_0}\Big(\frac{\varepsilon}{3}+\delta\Big)
k\xi_{tt}\eta_{tt} dr-\int_0^{r_0}\varepsilon\eta'^2_{tt}rdr-\int_0^{r_0}\Big(\frac{4\varepsilon}{3}+\delta\Big)k^2\eta_{tt}^2 r dr-\int_0^{r_0}\varepsilon\zeta'^2_{tt}rdr\\
&+2\int_0^{r_0}\varepsilon\zeta_{tt}'\zeta_{tt} dr-\int_0^{r_0}\varepsilon\frac{\zeta_{tt}^2}{r}dr-\int_0^{r_0}\varepsilon k^2\zeta_{tt}^2rdr-\int_0^{r_0}\varepsilon \frac{m^2}{r}\xi_{tt}^2 dr+\int_0^{r_0}\Big(\frac{\varepsilon}{3}+\delta\Big)(\frac{m}{r}\zeta_{tt})'\xi_{tt} r dr\\
&-2\int_0^{r_0}\varepsilon \frac{m}{r}\xi_{tt}\zeta_{tt} dr-\int_0^{r_0}\varepsilon \frac{m^2}{r}\eta_{tt}^2dr+\int_0^{r_0}\Big(\frac{\varepsilon}{3}
+\delta\Big)km\zeta_{tt}\eta_{tt}  dr-\int_0^{r_0}\Big(\frac{\varepsilon}{3}+\delta\Big)m\zeta_{tt}\xi_{tt}' dr\\
&-\int_0^{r_0}\Big(\frac{7\varepsilon}{3}+\delta\Big)\frac{m}{r}\zeta_{tt}\xi_{tt} dr-\int_0^{r_0}\Big(\frac{4\varepsilon}{3}+\delta\Big)\frac{m^2}{r}\zeta_{tt}^2 dr=\frac{1}{2}\frac{d}{dt}\int_0^{r_0}\rho (\xi_{tt}^2+\eta_{tt}^2+\zeta
_{tt}^2)rdr.
\end{split}
\end{equation*}
Using the following identities
\begin{equation*}
\begin{split}
\int_0^{r_0}\left(\frac{B_{\theta}^2}{r^2}\right)'(r\xi_t)^2 dr&=\int_0^{r_0}2\xi_t^2 B_{\theta} B'_{\theta}dr- \int_0^{r_0}\frac{2B_{\theta}^2\xi_t^2}{r}dr\\
&=2B_{\theta}^2\xi_t^2|_{r=0}^{r=r_0}-\int_0^{r_0}\Big(2B_{\theta}B'_{\theta}\xi_t^2 +4B_{\theta}^2\xi'_t\xi_t+\frac{2B_{\theta}^2\xi_t^2}{r}\Big)dr,
\end{split}
\end{equation*}
\begin{equation*}
-\int_0^{r_0}\Big(\frac{4\varepsilon}{3}+\delta\Big)\xi_{tt}'\xi_{tt} dr+\int_0^{r_0}\Big(\frac{4\varepsilon}{3}+\delta\Big)\xi'_{tt}\xi_{tt} dr=-\int_0^{r_0}2\delta\xi'_{tt}\xi_{tt} dr+\int_0^{r_0}\frac{4\varepsilon}{3}\xi'_{tt}\xi_{tt} dr
+[\delta\xi_{tt}^2]_{r=0}^{r=r_0}-\Big[\frac{2}{3}\varepsilon\xi_{tt}^2\Big]_{r=0}^{r=r_0},
\end{equation*}
\begin{equation*}
\begin{split}
&-\int_0^{r_0}\Big(\frac{\varepsilon}{3}+\delta\Big)
k\eta_{tt}'\xi_{tt} rdr+\int_0^{r_0}\Big(\frac{\varepsilon}{3}+\delta\Big)
k\xi'_{tt}\eta_{tt} rdr+\int_0^{r_0}\Big(\frac{\varepsilon}{3}+\delta\Big)
k\xi_{tt} \eta_{tt} dr\\
&\quad=2\delta\int_0^{r_0}k\xi_{tt}\eta_{tt} dr+2\delta\int_0^{r_0}k\xi'_{tt}\eta_{tt} r dr-[\delta k\xi_{tt}\eta_{tt} r]_{r=0}^{r=r_0}+\Big[\frac{2}{3}\varepsilon k \xi_{tt}\eta_{tt} r\Big]_{r=0}^{r=r_0}+\Big[\varepsilon k \xi_{tt}\eta_{tt} r\Big]_{r=0}^{r=r_0}\\
&\qquad-2\varepsilon\int_0^{r_0}
k\eta'_{tt}\xi_{tt} rdr-\int_0^{r_0}\frac{4\varepsilon}{3}k\xi'_{tt}\eta_{tt} rdr-\int_0^{r_0}\frac{4\varepsilon}{3}k\xi_{tt}\eta_{tt} dr,
\end{split}
\end{equation*}
\begin{equation*}
\begin{split}
\int_0^{r_0}\Big(\frac{\varepsilon}{3}+\delta\Big)(\frac{m}{r}\zeta_{tt})'\xi_{tt} rdr=\Big(\frac{\varepsilon}{3}+\delta\Big)[m\zeta_{tt} \xi_{tt} ]_{r=0}^{r=r_0}-\int_0^{r_0}\Big(\frac{\varepsilon}{3}+\delta\Big)\frac{m}{r}\zeta_{tt}\xi_{tt} dr-\int_0^{r_0}\Big(\frac{\varepsilon}{3}+\delta\Big)m\zeta_{tt}\xi_{tt}' dr,
\end{split}
\end{equation*}
\begin{equation*}
\begin{split}
-\int_0^{r_0}\frac{2\varepsilon}{3}m\zeta_{tt}\xi'_{tt} dr+2\int_0^{r_0}\varepsilon m\xi_{tt} \zeta_{tt}'dr-	2\int_0^{r_0}\varepsilon m\xi_{tt} \zeta'_{tt}dr
&=-2\varepsilon[m\zeta_{tt} \xi_{tt} ]_{r=0}^{r=r_0}+2\int_0^{r_0}\varepsilon m\xi_{tt} \zeta_{tt}'dr\\
&\quad+\int_0^{r_0}\frac{4\varepsilon}{3}m\zeta_{tt}\xi_{tt}' dr,
\end{split}
\end{equation*}
we deduce
\begin{equation}\label{spectral-proof-5-b}
\begin{split}
&[\xi_{tt} (\gamma p+B_{\theta}^2)(r\xi_t)']_{r=0}^{r=r_0}-[k(\gamma p+B_{\theta}^2)\xi_{tt}\eta_t r]_{r=0}^{r=r_0}+[m\gamma p\xi_{tt} \zeta_t  ]_{r=0}^{r=r_0}\\
&\quad-\left[2B_{\theta}^2\xi_t\xi_{tt}\right]_{r=0}^{r=r_0}+\Big[\Big(\frac{4\varepsilon}{3}+\delta\Big)\xi'_{tt} \xi_{tt} r\Big]_{r=0}^{r=r_0}+[\varepsilon\eta'_{tt}\eta_{tt} r+\varepsilon k\xi_{tt}\eta_{tt} r]_{r=0}^{r=r_0}\\
&\quad+[\varepsilon\zeta'_{tt}\zeta_{tt} r-\varepsilon\zeta_{tt}^2-\varepsilon m\xi_{tt} \zeta_{tt} ]_{r=0}^{r=r_0}-[\delta k\xi_{tt}\eta_{tt} r]_{r=0}^{r=r_0}+\Big[\frac{2}{3}\varepsilon k \xi_{tt}\eta_{tt} r\Big]_{r=0}^{r=r_0}\\
&\quad+[\delta\xi_{tt}^2]_{r=0}^{r=r_0}-\Big[\frac{2}{3}\varepsilon\xi_{tt}^2\Big]_{r=0}^{r=r_0}+\Big(\delta-\frac{2\varepsilon}{3}\Big)[m\xi_{tt}\zeta_{tt} ]_{r=0}^{r=r_0}\\
&\quad-\frac{1}{2}\frac{d}{dt}\bigg\{
\int_0^{r_0}\Big\{(m^2+k^2r^2)\Big[\frac{B_\theta}{r}\eta_t+\frac{-kB_\theta(r\xi_t)'+2kB_{\theta}\xi_t}{m^2+k^2r^2}\Big]^2\\
&\quad+\gamma p\Big[\frac{1}{r}(r\xi_t)'-k\eta_t+\frac{m\zeta_t}{r}\Big]^2\Big\}rdr+\int_0^{r_0}\frac{m^2B_\theta^2}{r(m^2+k^2r^2)}(\xi_t-r\xi'_t)^2dr\\
&\quad+\int_0^{r_0}\Big[2p'+\frac{m^2B_{\theta}^2}{r}\Big]\xi_t^2dr\bigg\}-\int_0^{r_0}\varepsilon\Big[\frac{2}{9}\Big(-2\xi'_{tt}+\frac{\xi_{tt}}{r}+\frac{m}{r}\zeta_{tt}-k\eta_{tt}\Big)^2\\
&\quad+\frac{2}{9}\Big(\xi'_{tt}-\frac{2\xi_{tt}}{r}-\frac{2m}{r}\zeta_{tt}-k\eta_{tt}\Big)^2+\frac{2}{9}\Big(\xi'_{tt}+\frac{\xi_{tt}}{r}+\frac{m}{r}\zeta_{tt}+2k\eta_{tt}\Big)^2\\
&\quad+\Big(-\zeta_{tt}'+\frac{\zeta_{tt}}{r}+\frac{m}{r}\xi_{tt}\Big)^2+(\eta_{tt}'+k\xi_{tt})^2+\Big(\frac{m}{r}\eta_{tt}-k\zeta_{tt}\Big)^2\Big]rdr
\\
&\quad-\int_0^{r_0}\delta\Big(\xi'_{tt}+\frac{\xi_{tt}}{r}+\frac{m}{r}\zeta_{tt}-k\eta_{tt}\Big)^2rdr\\
&=\frac{1}{2}\frac{d}{dt}\int_0^{r_0}\rho (\xi_{tt}^2+\eta_{tt}^2+\zeta
_{tt}^2)rdr.
\end{split}
\end{equation}
On the other hand, multiplying the first equation of \eqref{spectral-proof-2} by $r \widehat{ Q}_{rtt}$, integrating over $(r_0,r_w)$, one can show 
by \eqref{spectral-proof-3b}  that
\begin{equation}\label{spectral-proof-6-b}
\frac{1}{2}\frac{d}{dt}\int_{r_0}^{r_w}\Big[|\widehat{Q}_{rt}|^2+\frac{1}{m^2+k^2r^2}|(r\widehat{Q}_{rt})'|^2\Big]rdr
+[\frac{r}{m^2+k^2r^2}(r\widehat{Q}_{rt})'r\widehat{ Q}_{rtt}]|_{r=r_0}=0.
\end{equation}
Applying the boundary conditions \eqref{spectral-proof-3}-\eqref{spectral-proof-4}, $(\xi,\eta,\zeta)\in Y_{m,k}$ and the Definition \ref{defi-y} of function space $Y_{m,k}$,  from \eqref{spectral-proof-5-b} and \eqref{spectral-proof-6-b}, one can establish that \eqref{esimate3-growth} holds.
Hence, the result is proved.
\end{proof}
\begin{lem} \label{estimate4-lem}
Assume $\|(\xi_t,\eta_t,\zeta_t\|_1$ and  $\|(\xi_t,\eta_t,\zeta_t)\|_2$ are bounded,  then it holds that 
	\begin{equation}\label{estimate4-growth}
	\begin{split}
	&-\bigg\{\int_0^{r_0}\Big\{(m^2+k^2r^2)\Big[\frac{B_\theta}{r}\eta_t+\frac{-kB_\theta(r\xi_t)'+2kB_{\theta}\xi_t}{m^2+k^2r^2}\Big]^2\\
	&\qquad+\gamma p\Big[\frac{1}{r}(r\xi_t)'-k\eta_t+\frac{m\zeta_t}{r}\Big]^2\Big\}rdr+\int_0^{r_0}\frac{m^2B_\theta^2}{r(m^2+k^2r^2)}(\xi_t-r\xi'_t)^2dr\\
	&\qquad+\int_0^{r_0}\Big[2p'+\frac{m^2B_{\theta}^2}{r}\Big]\xi_t^2dr+\int_{r_0}^{r_w}\bigg[|\widehat{Q}_{rt}|^2+\frac{1}{m^2+k^2r^2}|(r\widehat{Q}_{rt})'|^2\bigg]r
	dr\bigg\}\\
	&\quad \leq \Lambda^2\int_0^{r_0} \rho (|\xi_{t}|^2+|\eta_{t}|^2+|\zeta_{t}|^2)rdr+\Lambda\int_0^{r_0}\varepsilon\Big[\frac{2}{9}\Big(-2\xi'_{t}+\frac{\xi_{t}}{r}+\frac{m}{r}\zeta_{t}-k\eta_{t}\Big)^2\\
	&\qquad+\frac{2}{9}\Big(\xi'_{t}-\frac{2\xi_{t}}{r}-\frac{2m}{r}\zeta_{t}-k\eta_{t}\Big)^2+\frac{2}{9}\Big(\xi'_{t}+\frac{\xi_{t}}{r}+\frac{m}{r}\zeta_{t}+2k\eta_{t}\Big)^2\\
	&\qquad+\Big(-\zeta'_{t}+\frac{\zeta_{t}}{r}+\frac{m}{r}\xi_{t}\Big)^2+(\eta'_{t}
	+k\xi_{t})^2+\Big(\frac{m}{r}\eta_{t}-k\zeta_{t}\Big)^2\Big]rdr\\
	&\qquad+\Lambda\int_0^{r_0}\delta\Big(\xi'_{t}+\frac{\xi_{t}}{r}+\frac{m}{r}\zeta_{t}-k\eta_{t}\Big)^2rdr.
	\end{split}
	\end{equation}
\end{lem}
\begin{proof}
From the energy \eqref{va-m-1-b}, it follows that
\begin{equation*}
\begin{split}
&-\bigg\{2\pi^2\int_0^{r_0}\Big\{(m^2+k^2r^2)\Big[\frac{B_\theta}{r}\eta_t+\frac{-kB_\theta(r\xi_t)'+2kB_{\theta}\xi_t}{m^2+k^2r^2}\Big]^2\\
&\quad+\gamma p\Big[\frac{1}{r}(r\xi_t)'-k\eta_t+\frac{m\zeta_t}{r}\Big]^2\Big\}rdr+2\pi^2\int_0^{r_0}\frac{m^2B_\theta^2}{r(m^2+k^2r^2)}(\xi_t-r\xi'_t)^2dr\\
&\quad+2\pi^2 	\int_0^{r_0}\Big[2p'+\frac{m^2B_{\theta}^2}{r}\Big]\xi_t^2dr+2\pi^2\int_{r_0}^{r_w}\bigg[|\widehat{Q}_{rt}|^2+\frac{1}{m^2+k^2r^2}|(r\widehat{Q}_{rt})'|^2\bigg]r
	dr\bigg\}\\
&=-E(\xi_t,\eta_t,\zeta_t,\widehat{ Q}_{rt})+2\pi^2\int_0^{r_0}\varepsilon\mu \Big[\frac{2}{9}\Big(-2\xi'_t+\frac{\xi_t}{r}+\frac{m}{r}\zeta_t-k\eta_t\Big)^2\\
&\quad+\frac{2}{9}\Big(\xi'_t-\frac{2\xi_t}{r}-\frac{2m}{r}\zeta_t-k\eta_t\Big)^2+\frac{2}{9}\Big(\xi'_t+\frac{\xi_t}{r}+\frac{m}{r}\zeta_t+2k\eta_t\Big)^2+\Big(-\zeta'_t
	+\frac{\zeta_t}{r}+\frac{m}{r}\xi_t\Big)^2\\
&\quad+(\eta'_t+k\xi_t)^2+\Big(\frac{m}{r}\eta_t-k\zeta_t\Big)^2\Big]rdr
+2\pi^2\int_0^{r_0}\delta\mu\Big(\xi'_t+\frac{\xi_t}{r}+\frac{m}{r}\zeta_t-k\eta_t\Big)^2rdr.
	\end{split}
	\end{equation*}
	Hence, it follows from \eqref{biggest-mode} that
	\begin{equation*}
	\begin{split}
	&-\bigg\{2\pi^2\int_0^{r_0}\Big\{(m^2+k^2r^2)\Big[\frac{B_\theta}{r}\eta_t+\frac{-kB_\theta(r\xi_t)'+2kB_{\theta}\xi_t}{m^2+k^2r^2}\Big]^2\\
	&\quad+\gamma p\Big[\frac{1}{r}(r\xi_t)'-k\eta_t+\frac{m\zeta_t}{r}\Big]^2\Big\}rdr+2\pi^2\int_0^{r_0}\frac{m^2B_\theta^2}{r(m^2+k^2r^2)}(\xi_t-r\xi'_t)^2dr\\
	&\quad+2\pi^2 	\int_0^{r_0}\Big[2p'+\frac{m^2B_{\theta}^2}{r}\Big]\xi_t^2dr+2\pi^2\int_{r_0}^{r_w}\bigg[|\widehat{Q}_{rt}|^2+\frac{1}{m^2+k^2r^2}|(r\widehat{Q}_{rt})'|^2\bigg]r
	dr\bigg\}\\
	&\leq 2\Lambda^2\pi^2\int_0^{r_0} \rho (|\xi_{t}|^2+|\eta_{t}|^2+|\zeta_{t}|^2)rdr+2\Lambda\pi^2\int_0^{r_0}\varepsilon \Big[\frac{2}{9}\Big(-2\xi'_t+\frac{\xi_t}{r}+\frac{m}{r}\zeta_t-k\eta_t\Big)^2\\
	&\quad+\frac{2}{9}\Big(\xi'_t-\frac{2\xi_t}{r}-\frac{2m}{r}\zeta_t-k\eta_t\Big)^2+\frac{2}{9}\Big(\xi'_t+\frac{\xi_t}{r}+\frac{m}{r}\zeta_t+2k\eta_t\Big)^2+\Big(-\zeta'_t+\frac{\zeta_t}{r}+\frac{m}{r}\xi_t\Big)^2\\
	&\quad+(\eta'_t+k\xi_t)^2+\Big(\frac{m}{r}\eta_t-k\zeta_t\Big)^2\Big]rdr
	+2\Lambda\pi^2\int_0^{r_0}\delta\Big(\xi'_t+\frac{\xi_t}{r}+\frac{m}{r}\zeta_t-k\eta_t\Big)^2rdr.
	\end{split}
	\end{equation*}
\end{proof}
\subsection{Exponential growth  about $\xi$, $\eta$ and $\zeta$}
In this subsection, we will use Lemma \ref{estimate3-lem} and Lemma \ref{estimate4-lem} to prove the exponential growth about $\xi$, $\eta$ and $\zeta$.
\begin{thm}\label{thm-growth}  
Let $(\xi,\eta,\zeta)$ be a $H^2$ solution to the system \eqref{spectal-m=1}  and $\widehat{ Q}_r$ be a $H^2$ solution to the equation \eqref{spectral-proof-2} with the corresponding boundary conditions, then it holds that 
	\begin{equation}
	\begin{split}
	&\|(\xi_t,\eta_t,\zeta_t)(t)\|_1^2+\|(\xi_t,\eta_t,\zeta_t)(t)\|_2^2+\|\partial_{tt}(\xi,\eta,\zeta)(t)\|_1^2\\
	&\quad\leq Ce^{2\Lambda t} \bigg(\|(\xi_t,\eta_t,\zeta_t)(0)\|_1^2+\|(\xi_t,\eta_t,\zeta_t)(0)\|_2^2
	+\|\partial_{tt}(\xi_,\eta,\zeta)(0)\|_1^2+\|\widehat{ Q}_{rt}(0)\|_{H^1}^2\bigg).
	\end{split}
	\end{equation}
\end{thm}
\begin{proof}
Integrating \eqref{esimate3-growth} of Lemma \ref{estimate3-lem} in time over $[0,t]$, we get
\begin{equation}
\begin{split}
&\frac{1}{2}\int_0^{r_0} \rho (|\xi_{tt}|^2+|\eta_{tt}|^2+|\zeta_{tt}|^2)rdr+\int_0^t\int_0^{r_0}\varepsilon\Big[\frac{2}{9}\Big(-2\xi'_{tt}+\frac{\xi_{tt}}{r}+\frac{m}{r}\zeta_{tt}-k\eta_{tt}\Big)^2\\
&\qquad+\frac{2}{9}\Big(\xi'_{tt}-\frac{2\xi_{tt}}{r}-\frac{2m}{r}\zeta_{tt}-k\eta_{tt}\Big)^2+\frac{2}{9}\Big(\xi'_{tt}+\frac{\xi_{tt}}{r}+\frac{m}{r}\zeta_{tt}+2k\eta_{tt}\Big)^2\\
&\qquad+\Big(-\zeta'_{tt}+\frac{\zeta_{tt}}{r}+\frac{m}{r}\xi_{tt}\Big)^2+(\eta'_{tt}+k\xi_{tt})^2+(\frac{m}{r}\eta_{tt}-k\zeta_{tt})^2\Big]rdrds\\
&\qquad+\int_0^t\int_0^{r_0}\delta\Big(\xi'_{tt}+\frac{\xi_{tt}}{r}+\frac{m}{r}\zeta_{tt}-k\eta_{tt}\Big)^2rdrds\\
&\quad=N_0-\frac{1}{2}\bigg\{\int_0^{r_0}\Big\{(m^2+k^2r^2)\Big[\frac{B_\theta}{r}\eta_t+\frac{-kB_\theta(r\xi_t)'+2kB_{\theta}\xi_t}{m^2+k^2r^2}\Big]^2\\
&\qquad+\gamma p\Big[\frac{1}{r}(r\xi_t)'-k\eta_t+\frac{m\zeta_t}{r}\Big]^2\Big\}rdr+\int_0^{r_0}\frac{m^2B_\theta^2}{r(m^2+k^2r^2)}(\xi_t-r\xi'_t)^2dr\\
&\qquad+\int_0^{r_0}\Big[2p'+\frac{m^2B_{\theta}^2}{r}\Big]\xi_t^2dr+\int_{r_0}^{r_w}\Big[|\widehat{Q}_{rt}|^2+\frac{1}{m^2+k^2r^2}|(r\widehat{Q}_{rt})'|^2\Big]rdr\bigg\},
\end{split}
\end{equation}
with 
\begin{equation}\label{N-0}
\begin{split}
N_0&=\frac{1}{2}\int_0^{r_0} \rho (|\xi_{tt}(0)|^2+|\eta_{tt}(0|^2+|\zeta_{tt}(0)|^2)rdr+\frac{1}{2}\bigg\{\int_0^{r_0}\Big\{(m^2+k^2r^2)\Big[\frac{B_\theta}{r}\eta_t(0)\\
&\quad+\frac{-kB_\theta(r\xi_t)'(0)+2kB_{\theta}\xi_t(0)}{m^2+k^2r^2}\Big]^2+\gamma p\Big[\frac{1}{r}(r\xi_t)'(0)-k\eta_t(0)+\frac{m\zeta_t(0)}{r}\Big]^2\Big\}rdr\\
&\quad+\int_0^{r_0}\frac{m^2B_\theta^2}{r(m^2+k^2r^2)}(\xi_t(0)-r\xi'_t(0))^2dr+\int_0^{r_0}\Big[2p'+\frac{m^2B_{\theta}^2}{r}\Big]\xi_t^2(0)dr\\
&\quad+\int_{r_0}^{r_w}\Big[|\widehat{Q}_{rt}(0)|^2+\frac{1}{m^2+k^2r^2}|(r\widehat{Q}_{rt})'(0)|^2\Big]rdr\bigg\}.
\end{split}
\end{equation}
Applying Lemma \ref{estimate4-lem}, we show that
\begin{equation}\label{cor-lem4.4}
\begin{split}
&\frac{1}{2}\int_0^{r_0} \rho (|\xi_{tt}|^2+|\eta_{tt}|^2+|\zeta_{tt}|^2)rdr+\int_0^t\int_0^{r_0}\varepsilon\Big[\frac{2}{9}\Big(-2\xi'_{tt}+\frac{\xi_{tt}}{r}+\frac{m}{r}\zeta_{tt}-k\eta_{tt}\Big)^2\\
&\qquad+\frac{2}{9}\Big(\xi'_{tt}-\frac{2\xi_{tt}}{r}-\frac{2m}{r}\zeta_{tt}-k\eta_{tt}\Big)^2+\frac{2}{9}\Big(\xi'_{tt}+\frac{\xi_{tt}}{r}+\frac{m}{r}\zeta_{tt}+2k\eta_{tt}\Big)^2\\
&\qquad+\Big(-\zeta'_{tt}+\frac{\zeta_{tt}}{r}+\frac{m}{r}\xi_{tt}\Big)^2+(\eta'_{tt}+k\xi_{tt})^2
+(\frac{m}{r}\eta_{tt}-k\zeta_{tt})^2\Big]rdrds\\
&\qquad+\int_0^t\int_0^{r_0}\delta\Big(\xi'_{tt}+\frac{\xi_{tt}}{r}+\frac{m}{r}\zeta_{tt}-k\eta_{tt}\Big)^2rdrds\\
&\quad\leq N_0+ \frac{1}{2}\Lambda^2\int_0^{r_0} \rho (|\xi_{t}|^2+|\eta_{t}|^2+|\zeta_{t}|^2)rdr\\
&\qquad+\frac{1}{2}\Lambda\int_0^{r_0}\varepsilon\Big[\frac{2}{9}\Big(-2\xi'_{t}+\frac{\xi_{t}}{r}+\frac{m}{r}\zeta_{t}-k\eta_{t}\Big)^2\\
&\qquad+\frac{2}{9}\Big(\xi'_{t}-\frac{2\xi_{t}}{r}-\frac{2m}{r}\zeta_{t}-k\eta_{t}\Big)^2+\frac{2}{9}\Big(\xi'_{t}+\frac{\xi_{t}}{r}+\frac{m}{r}\zeta_{t}+2k\eta_{t}\Big)^2\\
&\qquad+\Big(-\zeta'_{t}+\frac{\zeta_{t}}{r}+\frac{m}{r}\xi_{t}\Big)^2+(\eta'_{t}+k\xi_{t})^2
+(\frac{m}{r}\eta_{t}-k\zeta_{t})^2\Big]rdr\\
&\qquad+\frac{1}{2}\Lambda\int_0^{r_0}\delta\Big(\xi'_{t}+\frac{\xi_{t}}{r}+\frac{m}{r}\zeta_{t}-k\eta_{t}\Big)^2rdr.
\end{split}
\end{equation}
Using the definitions of the norms $\|\cdot \|_1$ and $\|\cdot\|_2$ given by \eqref{norm-1}-\eqref{norm-2} in the introduction, we prove from  \eqref{cor-lem4.4} that
\begin{equation}\label{growth-1}
\begin{split}
\frac{1}{2}\|\partial_{tt}(\xi,\eta,\zeta)\|_1^2+\int_0^t\|(\partial_{tt}(\xi,\eta,\zeta
))\|_2^2ds\leq N_0+\frac{\Lambda^2}{2}\|(\xi_t,\eta_t,\zeta_t)\|_1^2+\frac{\Lambda}{2}\|(\xi_t,\eta_t,\zeta_t)\|_2^2.
\end{split}
\end{equation}
Integrating in time and using Cauchy inequality, we obtain
\begin{equation}\label{growth-2}
\begin{split}
\Lambda\|(\xi_t,\eta_t,\zeta_t)\|_2^2&=\Lambda\|(\xi_t,\eta_t,\zeta_t)(0)\|_2^2+\Lambda\int_0^t2<(\xi_t,\eta_t,\zeta_t), \partial_{tt}(\xi,\eta,\zeta)>_2ds\\&\leq \Lambda\|(\xi_t,\eta_t,\zeta_t)(0)\|_2^2+\int_0^t\|\partial_{tt}(\xi,\eta,\zeta
)\|_2^2ds+\Lambda^2\int_0^t\|(\xi_t,\eta_t,\zeta_t
)\|_2^2ds,
\end{split}
\end{equation}
\begin{equation}\label{growth-3}
\Lambda\partial_t\|(\xi_t,\eta_t,\zeta_t)\|_1^2=2\Lambda<\partial_{tt}(\xi,\eta,\zeta),\partial_{t}(\xi,\eta,\zeta)>_1\leq \Lambda^2\|(\xi_t,\eta_t,\zeta_t)\|_1^2+\|\partial_{tt}(\xi,\eta,\zeta)\|_1^2.
\end{equation}
Combining estimates \eqref{growth-1}, \eqref{growth-2}  and \eqref{growth-3}, we deduce that 
\begin{equation}
\begin{split}
&\partial_t\|(\xi_t,\eta_t,\zeta_t)\|_1^2+\|(\xi_t,\eta_t,\zeta_t)\|_2^2\\
&\leq  \|(\xi_t,\eta_t,\zeta_t)(0)\|_2^2+\frac{2}{\Lambda}\int_0^t\|\partial_{tt}(\xi,\eta,\zeta
)\|_2^2ds+\Lambda\int_0^t\|(\xi_t,\eta_t,\zeta_t\|_2^2ds+\Lambda\|(\xi_t,\eta_t,\zeta_t)\|_1^2
\\
&\quad+\frac{1}{\Lambda}\|\partial_{tt}(\xi,\eta,\zeta)\|_1^2-\frac{1}{\Lambda}\int_0^t\|\partial_{tt}(\xi,\eta,\zeta
)\|_2^2ds\\
&\leq \frac{2N_0}{\Lambda}+\Lambda\|(\xi_t,\eta_t,\zeta_t)\|_1^2+\|(\xi_t,\eta_t,\zeta_t)\|_2^2+\|(\xi_t,\eta_t,\zeta_t)(0)\|_2^2\\
&\quad+\Lambda\int_0^t\|(\xi_t,\eta_t,\zeta_t\|_2^2ds
+\Lambda\|(\xi_t,\eta_t,\zeta_t)\|_1^2-\frac{1}{\Lambda}\int_0^t\|\partial_{tt}(\xi,\eta,\zeta
)\|_2^2ds\\
&\leq \frac{2N_0}{\Lambda}+2\|(\xi_t,\eta_t,\zeta_t)(0)\|_2^2+2\Lambda\|(\xi_t,\eta_t,\zeta_t)\|_1^2+2\Lambda\int_0^t\|(\xi_t,\eta_t,\zeta_t)\|_2^2ds,
\end{split}
\end{equation} 
where we have used the facts
 \begin{equation}
\|(\xi_t,\eta_t,\zeta_t)\|_2^2\leq \|(\xi_t,\eta_t,\zeta_t)(0)\|_2^2+\frac{1}{\Lambda}\int_0^t\|\partial_{tt}(\xi,\eta,\zeta
)\|_2^2ds+\Lambda\int_0^t\|(\xi_t,\eta_t,\zeta_t)\|_2^2ds.
\end{equation}
It follows from the Gronwall's inequality that 
\begin{equation}\label{growth-4}
\begin{split}
&\|(\xi_t,\eta_t,\zeta_t)\|_1^2+\int_0^t\|(\xi_t,\eta_t,\zeta_t)\|_2^2ds\\
&\quad\leq e^{2\Lambda t}\|(\xi_t,\eta_t,\zeta_t)(0)\|_1^2+\Big(\frac{N_0}{\Lambda^2}
+\frac{\|(\xi_t,\eta_t,\zeta_t)(0)\|_2^2}{\Lambda}\Big)(e^{2\Lambda t}-1).
\end{split}
\end{equation} 
From \eqref{growth-1}, we know that
\begin{equation}
\begin{split}
\frac{1}{\Lambda}\|\partial_{tt}(\xi,\eta,\zeta)\|_1^2+\frac{2}{\Lambda}\int_0^t\|(\partial_{tt}(\xi,\eta,\zeta
))\|_2^2ds\leq \frac{2N_0}{\Lambda}+\Lambda\|(\xi_t,\eta_t,\zeta_t)\|_1^2+\|(\xi_t,\eta_t,\zeta_t)\|_2^2.
\end{split}
\end{equation}
Therefore, 
\begin{equation}
\begin{split}
&\frac{1}{\Lambda}\|\partial_{tt}(\xi,\eta,\zeta)\|_1^2+\|(\xi_t,\eta_t,\zeta_t)\|_2^2\\
&\quad\leq -\frac{2}{\Lambda}\int_0^t\|(\partial_{tt}(\xi,\eta,\zeta
))\|_2^2ds+ \frac{2N_0}{\Lambda}+\Lambda\|(\xi_t,\eta_t,\zeta_t)\|_1^2+2\|(\xi_t,\eta_t,\zeta_t)\|_2^2,
\end{split}
\end{equation}
which together with \eqref{growth-2} gives that
\begin{equation}
\begin{split}
&\frac{1}{\Lambda}\|\partial_{tt}(\xi,\eta,\zeta)\|_1^2+\|(\xi_t,\eta_t,\zeta_t)\|_2^2\\
&\quad\leq -\frac{2}{\Lambda}\int_0^t\|(\partial_{tt}(\xi,\eta,\zeta
))\|_2^2ds+ \frac{2N_0}{\Lambda}+\Lambda\|(\xi_t,\eta_t,\zeta_t)\|_1^2\\
&\qquad+2\|(\xi_t,\eta_t,\zeta_t)(0)\|_2^2+\frac{2}{\Lambda}\int_0^t\|\partial_{tt}(\xi,\eta,\zeta
)\|_2^2ds+2\Lambda\int_0^t\|(\xi_t,\eta_t,\zeta_t
)\|_2^2ds\\
&\quad= \frac{2N_0}{\Lambda}+\Lambda\|(\xi_t,\eta_t,\zeta_t)\|_1^2+2\|(\xi_t,\eta_t,\zeta_t)(0)\|_2^2+2\Lambda\int_0^t\|(\xi_t,\eta_t,\zeta_t
)\|_2^2ds\\
&\quad\leq e^{2\Lambda t}\Big(2\Lambda\|(\xi_t,\eta_t,\zeta_t)(0)\|_1^2+\frac{2N_0}{\Lambda}+2\|(\xi_t,\eta_t,\zeta_t)(0)\|_2^2\Big).
\end{split}
\end{equation}
From the definitions of $N_0$ in \eqref{N-0}, it follows that
\begin{equation}
\begin{split}
N_0\leq C\Big(\|(\xi_t,\eta_t,\zeta_t)(0)\|_1^2+\|(\xi_t,\eta_t,\zeta_t)(0)\|_2^2+\|\partial_{tt}(\xi,\eta,\zeta)(0)\|_1^2+\|\widehat{ Q}_{rt}(0)\|_{H^1}^2\Big),
\end{split}
\end{equation}
which implies the result.
\end{proof}
\subsection{Exponential growth  about original notation $g$}
In this subsection, we will prove the exponential growth about solution $g$. First, let us introduce the basic estimates about the solution $g$.
	\begin{lem}\label{vec-xi-grow-lem-first}
	Assume $g$ is a $H^2$ solution  to the system \eqref{linear-perturbation-and-boundary-viscosity} with the corresponding jump and boundary conditions, then we can get
		\begin{equation}\label{vec-xi-t-est}
		\begin{split}
		&\int_0^t\int_{\overline{\Omega}} \Big[\varepsilon\Big|\nabla g_{t}+\nabla g_{t}^T-\frac{2}{3}\mbox{div} g_{t} \, \mathbb{I}\Big|^2+2\delta |\mbox{div} g_{t}|^2\Big]dxds+\int_{\overline{\Omega}} \Big[|Q|^2+\gamma p|\nabla\cdot g |^2\Big]dx\\
		&\quad+\int_{\overline{\Omega}} \Big[(\nabla\times B)\cdot (g^*\times Q)+\nabla\cdot g(g^*\cdot \nabla p)\Big]dx+\int_{\overline{\Omega}^v}|\widehat{Q}|^2dx+\|\sqrt{\rho}g_{t}\|^2_{L^2}\\
		&=\int_{\overline{\Omega}} \Big[|Q_0|^2+\gamma p|\nabla\cdot g_0 |^2\Big]dx
		+\int_{\overline{\Omega}} \Big[(\nabla\times B)\cdot (g_0^*\times Q_0)+\nabla\cdot g_0(g^*_0\cdot \nabla p)\Big]dx\\
		&\quad+\int_{\overline{\Omega}^v}|\widehat{Q}_0|^2dx+\|\sqrt{\rho}g_{0t}\|^2_{L^2},
		\end{split}
		\end{equation}	
		with $Q=\nabla\times (g\times B)$.
	\end{lem}
	\begin{proof}
		Taking the time derivation about the system \eqref{linear-perturbation-and-boundary-viscosity}$_2$, then multiplying the resulted equation by $g_{t}$, similarly as the proof of Lemma \ref{joint},  we can show 
		\begin{equation*}
		\begin{split}
		& \frac{1}{2}\frac{d}{dt}\|\sqrt{\rho}g_{t}\|^2_{L^2}=-\frac{1}{2}\frac{d}{dt}\int_{\overline{\Omega}^v}|\widehat{Q}|^2dx-\frac{1}{2}\frac{d}{dt}\int_{\overline{\Omega}} \Big[|Q|^2+\gamma p|\nabla\cdot g |^2\Big]dx\\
		&\quad-\frac{1}{2}\frac{d}{dt}\int_{\overline{\Omega}} \Big[(\nabla\times B)\cdot (g^*\times Q)+\nabla\cdot g(g^*\cdot \nabla p)\Big]dx\\
		&\quad-\frac{1}{2}\int_{\overline{\Omega}} \Big[\varepsilon\Big|\nabla g_{t}+\nabla g_{t}^T-\frac{2}{3}\mbox{div} g_{t} \, \mathbb{I}\Big|^2+2\delta |\mbox{div} g_{t}|^2\Big]dx.
		\end{split}
		\end{equation*}
		Integrating the above equality about time, we have \eqref{vec-xi-t-est}.
\end{proof}		
	\begin{lem}\label{vec-xi-grow-lem}
	Assume $g$ is a $H^2$ solution to the  system  \eqref{linear-perturbation-and-boundary-viscosity} with the corresponding jump and boundary conditions, we can get
		\begin{equation}\label{vec-xi-grow}
		\begin{split}
		& \frac{1}{2}\frac{d}{dt}\|\sqrt{\rho}g_{tt}\|^2_{L^2}
		=-\frac{1}{2}\frac{d}{dt}\int_{\overline{\Omega}^v}|\widehat{Q}_t|^2dx-\frac{1}{2}\frac{d}{dt}\int_{\overline{\Omega}} \Big[|Q_t|^2+\gamma p|\nabla\cdot g_t |^2\Big]dx\\
		&\quad-\frac{1}{2}\frac{d}{dt}\int_{\overline{\Omega}} \Big[(\nabla\times B)\cdot (g^*_t\times Q_t)+\nabla\cdot g_t(g^*_t\cdot \nabla p)\Big]dx\\
		&\quad-\frac{1}{2}\int_{\overline{\Omega}} \Big[\varepsilon\Big|\nabla g_{tt}+\nabla g_{tt}^T-\frac{2}{3}\mbox{div} g_{tt} \, \mathbb{I}\Big|^2+2\delta |\mbox{div} g_{tt}|^2\Big]dx,
		\end{split}
		\end{equation}
		with $Q_t=\nabla\times (g_t\times B)$.
	\end{lem}
	\begin{proof}
		Taking the time derivation about the system \eqref{linear-perturbation-and-boundary-viscosity}$_2$, then multiplying the resulted equation by $g_{tt}$,  similarly as the proof of Lemma \ref{joint},  we can show \eqref{vec-xi-grow}.
	\end{proof}
	\begin{lem}\label{vec-xi-grow-energy}
	Assume that $\|g\|_1$ and $\|g\|_2$ are bounded with their definitions in \eqref{norm-1-xi}-\eqref{norm-2-xi}, and	assume $g$ is a $H^2$ solution to the system \eqref{linear-perturbation-and-boundary-viscosity}, 
we can get
		\begin{equation}\label{explain}
		\begin{split}
		&-\int_{\overline{\Omega}} \Big[|Q|^2+\gamma p|\nabla\cdot g |^2\Big]dx-\int_{\overline{\Omega}} \Big[(\nabla\times B)\cdot (g^*\times Q)+\nabla\cdot g(g^*\cdot \nabla p)\Big]dx-\int_{\overline{\Omega}^v}|\widehat{Q}|^2dx\\
		&\quad\leq \Lambda^2 \|\sqrt{\rho}g\|^2_{L^2}+\Lambda\int_{\overline{\Omega}} \Big[\frac{\varepsilon}{2}\Big|\nabla g+\nabla g^T-\frac{2}{3}\mbox{div}  g \, \mathbb{I}\Big|^2+\delta |\mbox{div} g|^2\Big]dx,
		\end{split}
		\end{equation}	
		with $Q=\nabla\times (g\times B)$.
	\end{lem}
	\begin{proof}
		Notice that $$g_r(r,\theta,z)=\sum_{m,k \in Z}\mathcal F{g_r}(r,m,k)e^{im\theta+ikz},\quad g_\theta(r,\theta,z)=\sum_{m,k \in Z}\mathcal F{g_\theta}(r,m,k)e^{im\theta+ikz},$$ $$g_z(r,\theta,z)=\sum_{m,k \in Z}\mathcal F{g_z}(r,m,k)e^{im\theta+ikz}, \quad\widehat{ Q}_r(r,\theta,z)=\sum_{m,k \in Z}\mathcal F{\widehat{ Q}_r}(r,m,k)e^{im\theta+ikz},$$
		$$\widehat{ Q}_\theta(r,\theta,z)=\sum_{m,k \in Z}\mathcal F{\widehat{ Q}_\theta}(r,m,k)e^{im\theta+ikz},\quad
		\widehat{ Q}_z(r,\theta,z)=\sum_{m,k \in Z}\mathcal F{\widehat{ Q}_z}(r,m,k)e^{im\theta+ikz}.$$	
			Define the energy $\mathcal{E}$ and the dissipation term $D$ as follows
		\begin{equation*}
		\begin{split}
		\mathcal{E}= \int_{\overline{\Omega}} \Big[|Q|^2+\gamma p|\nabla\cdot g |^2\Big]dx+\int_{\overline{\Omega}} \Big[(\nabla\times B)\cdot (g^*\times Q)+\nabla\cdot g(g^*\cdot \nabla p)\Big]dx+\int_{\overline{\Omega}^v}|\widehat{Q}|^2dx,
		\end{split}
		\end{equation*}
			\begin{equation*}
		\begin{split}
		D= \mu\int_{\overline{\Omega}} \Big[\frac{\varepsilon}{2}\Big|\nabla g+\nabla g^T-\frac{2}{3}\mbox{div}  g \, \mathbb{I}\Big|^2+\delta |\mbox{div} g|^2\Big]dx.
		\end{split}
		\end{equation*}
		 Inserting the above Fourier expansions of $g_r(r,\theta,z) $, $g_\theta(r,\theta,z)$, $g_z(r,\theta,z)$,  $\widehat{ Q}_r(r,\theta,z)$, $\widehat{ Q}_\theta(r,\theta,z)$ and $\widehat{ Q}_z(r,\theta,z)$ into the above energy $\mathcal{E}$ and dissipation term $D$, we have
$$\mathcal{E}=\sum_{m,k \in Z}\mathcal{E}_{m,k}(\xi,\eta,\zeta,\widehat{ Q}_r)=\sum_{m,k \in Z}\Big(E_{m,k}(\xi,\eta,\zeta,\widehat{ Q}_r)-D_{m,k}(\xi,\eta,\zeta)\Big),$$ $$D=\sum_{m,k \in Z}D_{m,k}(\xi,\eta,\zeta),$$  with $E_{m,k}(\xi,\eta,\zeta,\widehat{ Q}_r)$ defined in  \eqref{va-m-1-b} when $m\neq 0$ and defined in \eqref{sausage-in-v} when $m=0$, and $D_{m,k}(\xi,\eta,\zeta)$ defined in \eqref{dissipation-rate}. 

In fact, in cylindrical coordinates, we know that
\begin{equation}\label{Q}
Q=\nabla\times (g\times B)=-e_\theta(\partial_rg_r B_\theta+g_r\partial_rB_\theta+\partial_zg_zB_\theta)+e_r\frac{\partial_\theta g_r B_\theta}{r}+e_z\frac{\partial_\theta g_zB_\theta}{r},
\end{equation}
\begin{equation}\label{xi-curl-q}
(\nabla\times B)\cdot (g^*\times Q)=-\Big(B_\theta'+\frac{B_\theta}{r}\Big)g_r(\partial_rg_r B_\theta+g_r B'_{\theta}+\partial_zg_zB_\theta)-\Big(B_\theta'+\frac{B_\theta}{r}\Big)\frac{g^*_\theta\partial_\theta g_r B_\theta}{r},
\end{equation}
\begin{equation}\label{presure-div}
(g^*\cdot\nabla  p)(\nabla \cdot g )=g_r  p'\Big(\frac1r(rg_r)'+\frac{\partial_\theta g_\theta}{r}+\partial_zg_z\Big), 
\end{equation}
\begin{equation}\label{nabla-xi}
\begin{split}
\nabla g=
\left(\begin{array}{ccc}
\partial_rg_r&\partial_{r}g_{\theta}&\partial_rg_z\\
\frac{\partial_{\theta}g_r}{r}- \frac{g_\theta}{r}&\frac{\partial_{\theta}g_\theta}{r}+\frac{g_r}{r}&\frac{\partial_{\theta}g_z}{r}\\
\partial_zg_r&\partial_zg_\theta&\partial_zg_z
\end{array}
\right),
\end{split}
\end{equation}
\begin{equation}\label{div-xi}
\begin{split}
\mbox{div}  g \, \mathbb{I}=\left(\begin{array}{ccc}
\xi'+\frac{\xi}{r}+\frac{\partial_{\theta}g_\theta}{r}+\partial_zg_z&0&0\\
0&\xi'+\frac{\xi}{r}+\frac{\partial_{\theta}g_\theta}{r}+\partial_zg_z&0\\
0&0&\xi'+\frac{\xi}{r}+\frac{\partial_{\theta}g_\theta}{r}+\partial_zg_z
\end{array}
\right).
\end{split}
\end{equation}
Inserting the Fourier expansions about $g_r$, $g_\theta$ and $g_z$ into \eqref{Q}--\eqref{div-xi}, we can get
\begin{equation}\label{Q-f}
\begin{split}
Q&=\sum_{m,k \in Z}e^{im\theta+ikz}\bigg\{-e_\theta\Big(\mathcal{F}\partial_rg_r B_\theta+\mathcal{F}g_r\partial_rB_\theta+ik\mathcal{F}g_zB_\theta\Big)\\
&\quad+e_r\frac{im \mathcal{F}g_r B_\theta}{r}+e_z\frac{im\mathcal{F}g_zB_\theta}{r}\bigg\},
\end{split}
\end{equation}
\begin{equation}\label{xi-curl-q-f}
\begin{split}
(\nabla\times B)\cdot (g^*\times Q)&=-\Big(B_\theta'+\frac{B_\theta}{r}\Big)\sum_{m,k \in Z}\mathcal{F}g_r\sum_{m,k \in Z}\Big(\mathcal{F}\partial_rg_r B_\theta+\mathcal{F}g_r B'_{\theta}+ik\mathcal{F}g_zB_\theta\Big)\\
&\quad-\Big(B_\theta'+\frac{B_\theta}{r}\Big)\sum_{m,k \in Z}\mathcal{F}g_\theta^*\sum_{m,k \in Z}\frac{im\mathcal{F}g_r B_\theta}{r},
\end{split}
\end{equation}
\begin{equation}\label{presure-div-f}
(g^*\cdot\nabla  p)(\nabla \cdot g )=\sum_{m,k \in Z}\mathcal{F}g_r  p'\sum_{m,k \in Z}\Big(\frac1r(r\mathcal{F}g_r)'+\frac{im\mathcal{F}g_\theta}{r}+ik\mathcal{F}g_z\Big), 
\end{equation}
\begin{equation}\label{nabla-xi-f}
\begin{split}
\nabla g&=
\sum_{m,k \in Z}e^{im\theta+ikz}\left(\begin{array}{ccc}
\mathcal{F}\partial_rg_r&\mathcal{F}\partial_{r}g_{\theta}&\mathcal{F}\partial_rg_z\\
\frac{im\mathcal{F}g_r}{r}- \frac{\mathcal{F}g_\theta}{r}&\frac{im\mathcal{F}g_\theta}{r}+\frac{\mathcal{F}g_r}{r}&\frac{im\mathcal{F}g_z}{r}\\
ik\mathcal{F}g_r&ik\mathcal{F}g_\theta&ik\mathcal{F}g_z
\end{array}
\right),
\end{split}
\end{equation}
\begin{equation}\label{div-xi-f}
\begin{split}
\mbox{div}  g \, \mathbb{I}=\sum_{m,k \in Z}e^{im\theta+ikz}\left(\begin{array}{ccc}
a&0&0\\
0&a&0\\
0&0&a
\end{array}
\right),
\end{split}
\end{equation}
with $a=\mathcal{F}\xi'+\frac{\mathcal{F}\xi}{r}+\frac{im\mathcal{F}g_\theta}{r}
+ik\mathcal{F}g_z$.

On the other hand,
\begin{equation}\label{div-Q}
\nabla\cdot \widehat{ Q}=\widehat{ Q}_r'+\frac{\widehat{ Q}_r}{r}+\frac{\partial_{\theta}\widehat{ Q}_\theta}{r}+\partial_z\widehat{ Q}_z=0,
\end{equation}
\begin{equation}\label{curl-q}
\nabla\times \widehat{ Q}=e_z\Big(\partial_r\widehat{ Q}_\theta-\frac{\partial_{\theta}\widehat{Q}_r}{r}+\frac{\widehat{ Q}_\theta}{r}\Big)+e_\theta(\partial_z\widehat{ Q}_r-\partial_r\widehat{ Q}_z)+e_r\Big(\frac{\partial_{\theta}\widehat{Q}_z}{r}-\partial_z\widehat{Q}_\theta\Big).
\end{equation}
Inserting the Fourier expansions about $\widehat{ Q}_r$, $\widehat{Q}_\theta$ and $\widehat{Q}_z$ into \eqref{div-Q} and \eqref{curl-q},  we obtain
\begin{equation}\label{div-Q-f}
\nabla\cdot \widehat{ Q}=\sum_{m,k \in Z}e^{im\theta+ikz}\bigg(\mathcal{F}\widehat{ Q}_r'+\frac{\mathcal{F}\widehat{ Q}_r}{r}+\frac{im\mathcal{F}\widehat{ Q}_\theta}{r}+ik\mathcal{F}\widehat{ Q}_z\bigg)=0,
\end{equation}
\begin{equation}\label{curl-q-f}
\begin{split}
\nabla\times \widehat{ Q}&=e_z\sum_{m,k \in Z}e^{im\theta+ikz}\Big(\mathcal{F}\partial_r\widehat{ Q}_\theta-\frac{im\mathcal{F}\widehat{Q}_r}{r}+\frac{\mathcal{F}\widehat{ Q}_\theta}{r}\Big)\\
&\quad+e_\theta\sum_{m,k \in Z}e^{im\theta+ikz}(ik\mathcal{F}\widehat{ Q}_r-\mathcal{F}\partial_r\widehat{ Q}_z)\\
&\quad+e_r\sum_{m,k \in Z}e^{im\theta+ikz}\Big(\frac{im\mathcal{F}\widehat{Q}_z}{r}
-ik\mathcal{F}\widehat{Q}_\theta\Big)=0.
\end{split}
\end{equation}
From \eqref{div-Q-f}, it follows that
\begin{equation}
\mathcal{F}\widehat{ Q}_z=\frac{i(r\mathcal{F}\widehat{ Q}_r)'-m\mathcal{F}\widehat{ Q}_\theta}{kr},
\end{equation}
which together with $\frac{im\mathcal{F}\widehat{Q}_z}{r}
-ik\mathcal{F}\widehat{Q}_\theta=0$, gives that 
\begin{equation}
\mathcal{F}\widehat{ Q}_{\theta}=\frac{im}{m^2+k^2r^2}(r\mathcal{F}\widehat{ Q}_r)', \quad \mathcal{F}\widehat{ Q}_{z}=\frac{ikr}{m^2+k^2r^2}(r\mathcal{F}\widehat{ Q}_r)'.
\end{equation}
Hence, we can show that
\begin{equation}
\begin{split}
\int_{\overline{\Omega}^v}|\widehat{ Q}|^2dx&=\int_0^{2\pi}\int_0^{2\pi }\int_{r_0}^{r_w}\bigg[|\widehat{ Q}_r|^2+|\widehat{ Q}_\theta|^2+|\widehat{ Q}_z|^2\bigg]rdrd\theta dz\\
&=4\pi^2\sum_{m,k \in Z}\int_{r_0}^{r_w}\bigg[|\mathcal{F}\widehat{ Q}_r|^2+|\mathcal{F}\widehat{ Q}_\theta|^2+|\mathcal{F}\widehat{ Q}_z|^2\bigg]rdr
\\
&=4\pi^2\sum_{m,k \in Z}\int_{r_0}^{r_w}\bigg[|\mathcal{F}\widehat{ Q}_r|^2r+\frac{1}{m^2+k^2r^2}|(r\mathcal{F}\widehat{Q}_r)'|^2\bigg]r
dr.
\end{split}
\end{equation}
Inserting \eqref{Q-f}-\eqref{div-xi-f} into the energy $E$ and dissipation term $D$, and at the same time, letting $$\xi=\mathcal{F}g_r,\quad \eta=-i\mathcal{F}g_z, \quad \zeta=i\mathcal{F}g_\theta,\quad i\widehat{ Q}_r=\mathcal{F}\widehat{Q}_r,$$ 
we can show that 
$$\mathcal{E}=\sum_{m,k \in Z}\mathcal{E}_{m,k}(\xi,\eta,\zeta,\widehat{Q}_r), \quad  D=\sum_{m,k \in Z}D_{m,k}(\xi,\eta,\zeta).$$ Here we have used the facts that $\xi$, $\eta$ and $\zeta$ are real-valued functions.

 From  Proposition \ref{growing-mode-b2}, we know that there exists the biggest growing mode for any $(m,k)\in \mathbb{Z}\times \mathbb{Z}$. Denoting the biggest growing mode as $\Lambda$, which can be found in \eqref{biggest-mode},  then we get that  $$-E_{m,k}(\xi,\eta,\zeta,\widehat{ Q}_r)\leq \Lambda^2J_{m,k} (\xi, \eta,\zeta).$$
Therefore, letting $\xi=\mathcal{F}g_r$, $\eta=-i\mathcal{F}g_z$,  $\zeta=i\mathcal{F}g_\theta$, $ i\widehat{ Q}_r=\mathcal{F}\widehat{Q}_r,$ using \eqref{Q-f}--\eqref{presure-div-f},
 we can show by Lemma \ref{estimate4-lem}  that
	\begin{equation*}
		\begin{split}
		&-\int_{\overline{\Omega}} \Big[|Q|^2+\gamma p|\nabla\cdot g |^2\Big]dx-\int_{
		\overline{\Omega}} \Big[(\nabla\times B)\cdot (g^*\times Q)+\nabla\cdot g(g^*\cdot \nabla p)\Big]dx-\int_{\overline{\Omega}^v}|\widehat{Q}|^2dx\\
		&\quad =-4\pi^2\sum_{m,k \in Z}\int_0^{r_0}\Big\{(m^2+k^2r^2)\Big[\frac{B_\theta}{r}\eta+\frac{-kB_\theta(r\xi)'+2kB_{\theta}\xi}{m^2+k^2r^2}\Big]^2+\gamma p\Big[\frac{1}{r}(r\xi)'-k\eta+\frac{m\zeta}{r}\Big]^2\Big\}rdr\\
		&\qquad-4\pi^2\sum_{m,k \in Z}\int_0^{r_0}\frac{m^2B_\theta^2}{r(m^2+k^2r^2)}(\xi-r\xi')^2dr-2\pi 	L\sum_{m,k \in Z}\int_0^{r_0}\Big[2p'+\frac{m^2B_{\theta}^2}{r}\Big]\xi^2dr\\
		&\qquad-4\pi^2\sum_{m,k \in Z}\int_{r_0}^{r_w}\bigg[|\widehat{Q}_r|^2+\frac{1}{m^2+k^2r^2}|(r\widehat{Q}_r)'|^2\bigg]rdr\\
		&\quad \leq 4\Lambda^2\pi^2\sum_{m,k \in Z}\int_0^{r_0} \rho (|\xi|^2+|\eta|^2+|\zeta|^2)rdr+4\Lambda\pi^2\sum_{m,k \in Z}\int_0^{r_0}\varepsilon\Big[\frac{2}{9}\Big(-2\xi'+\frac{\xi}{r}+\frac{m}{r}\zeta-k\eta\Big)^2\\
		&\qquad+\frac{2}{9}\Big(\xi'-\frac{2\xi}{r}-\frac{2m}{r}\zeta-k\eta\Big)^2+\frac{2}{9}\Big(\xi'+\frac{\xi}{r}+\frac{m}{r}\zeta+2k\eta\Big)^2+\Big(-\zeta'+\frac{\zeta}{r}+\frac{m}{r}\xi\Big)^2\\
		&\qquad+(\eta'
		+k\xi)^2+\Big(\frac{m}{r}\eta-k\zeta\Big)^2\Big]rdr+4\Lambda\pi^2\sum_{m,k \in Z}\int_0^{r_0}\delta\Big(\xi'+\frac{\xi}{r}+\frac{m}{r}\zeta-k\eta\Big)^2rdr
		\\
		&\quad= \Lambda^2 \|\sqrt{\rho}g\|^2_{L^2}+\Lambda\int_{\overline{\Omega}} \Big[\frac{\varepsilon}{2}\Big|\nabla g+\nabla g^T-\frac{2}{3}\mbox{div}
		g \, \mathbb{I}\Big|^2+\delta |\mbox{div} g|^2\Big]dx.
	\end{split}
		\end{equation*}	
	\end{proof}	
	Similar to the proof of Theorem \ref{thm-growth}, by Lemma \ref{vec-xi-grow-lem} and \ref{vec-xi-grow-energy}, we can get the following result.
	\begin{thm}\label{thm-growth-general}
	Assume $g$ is a $H^2$ solution  to the system \eqref{linear-perturbation-and-boundary-viscosity} with the corresponding boundary conditions, then it holds that 
		\begin{equation}
		\begin{split}
		&\|g_t\|_1^2+\|g_t\|_2^2+\|\partial_{tt}g(t)\|_1^2\\
		&\quad\leq Ce^{2\Lambda t} \Big(\|g_t(0)\|_1^2+\|g_t(0)\|_2^2
		+\|\partial_{tt}g(0)\|_1^2+\|\widehat{ Q}_{t}(0)\|_{L^2}^2\Big).
		\end{split}
		\end{equation}
	\end{thm}
	\begin{proof}
		Integrating \eqref{vec-xi-grow} of Lemma \ref{vec-xi-grow-lem} about time over $[0,t]$, we have
		\begin{equation}
		\begin{split}
		& \frac{1}{2}\|\sqrt{\rho}g_{tt}\|^2_{L^2}=L_0-\frac{1}{2}\int_{\overline{\Omega}^v}|\widehat{Q}_t|^2dx-\frac{1}{2}\int_{\overline{\Omega}} \Big[|Q_t|^2+\gamma p|\nabla\cdot g_t |^2\Big]dx\\
		&\quad-\frac{1}{2}\int_{\overline{\Omega}} \Big[(\nabla\times B)\cdot (g^*_t\times Q_t)+\nabla\cdot g_t(g^*_t\cdot \nabla p)\Big]dx\\
		&\quad-\int_0^t\int_{\overline{\Omega}} \Big[\frac{\varepsilon}{2}\Big|\nabla g_{tt}+\nabla g_{tt}^T-\frac{2}{3}\mbox{div}  g_{tt} \, \mathbb{I}\Big|^2+\delta |\mbox{div} g_{tt}|^2\Big]dxds,
		\end{split}
		\end{equation}
		with
		\begin{equation}\label{L-0}
		\begin{split}
		L_0&=\frac{1}{2}\|\sqrt{\rho}g_{tt}(0)\|^2_{L^2}+\frac{1}{2}\int_{\overline{\Omega}^v}|\widehat{Q}_t|^2(0)dx+\frac{1}{2}\int_{\overline{\Omega}} \Big[|Q_t|^2(0)+\gamma p|\nabla\cdot g_t |^2(0)\Big]dx\\
		&\quad+\frac{1}{2}\int_{\overline{\Omega}} \Big[(\nabla\times B)\cdot (g^*_t\times Q_t)(0)+\nabla\cdot g_t(g^*_t\cdot \nabla p)(0)\Big]dx.
		\end{split}
		\end{equation} 
		Using Lemma \ref{vec-xi-grow-energy}, we deduce 
		\begin{equation}\label{}
		\begin{split}
		& \frac{1}{2}\|\sqrt{\rho}g_{tt}\|^2_{L^2}+\int_0^t\int_{\overline{\Omega}} \Big[\frac{\varepsilon}{2}\Big|\nabla g_{tt}+\nabla g_{tt}^T-\frac{2}{3}\mbox{div}  g_{tt} \, \mathbb{I}\Big|^2+\delta |\mbox{div} g_{tt}|^2\Big]dxds\\
		&=L_0-\frac{1}{2}\int_{\overline{\Omega}^v}|\widehat{Q}_t|^2dx-\frac{1}{2}\int_{\overline{\Omega}} \Big[|Q_t|^2+\gamma p|\nabla\cdot g_t |^2\Big]dx\\
		&\quad-\frac{1}{2}\int_{\overline{\Omega}} \Big[(\nabla\times B)\cdot (g^*_t\times Q_t)+\nabla\cdot g_t(g^*_t\cdot \nabla p)\Big]dx\\
		&\leq L_0+\frac{\Lambda^2}{2} \|\sqrt{\rho}g_t\|^2_{L^2}+\frac{\Lambda}{2}\int_{\overline{\Omega}} \Big[\frac{\varepsilon}{2}\Big|\nabla g_t+\nabla g_t^T-\frac{2}{3}\mbox{div}  g_t \, \mathbb{I}\Big|^2+\delta |\mbox{div} g_t|^2\Big]dx.
		\end{split}
		\end{equation}
		Using the definitions of the norms $\|\cdot \|_1$ and $\|\cdot\|_2$ given by \eqref{norm-1-xi}-\eqref{norm-2-xi} in the introducntion, we obtain
		\begin{equation}\label{growth-1-g}
		\begin{split}
		\frac{1}{2}\|\partial_{tt}g\|_1^2+\int_0^t\|\partial_{tt}g\|_2^2ds\leq L_0+\frac{\Lambda^2}{2}\|g_t\|_1^2+\frac{\Lambda}{2}\|g_t\|_2^2.
		\end{split}
		\end{equation}
		Integrating in time and using Cauchy's inequality, we get
		\begin{equation}\label{growth-2-g}
		\begin{split}
		\Lambda\|g_t\|_2^2&=\Lambda\|g_t(0)\|_2^2+\Lambda\int_0^t2<g_t, \partial_{tt}g>_2ds\\&\leq \Lambda\|g_t(0)\|_2^2+\int_0^t\|\partial_{tt}g\|_2^2ds
		+\Lambda^2\int_0^t\|g_t\|_2^2ds,
		\end{split}
		\end{equation}
		\begin{equation}\label{growth-3-g}
		\Lambda\partial_t\|g_t\|_1^2=2\Lambda<\partial_{tt}g,\partial_{t}g>_1\leq \Lambda^2\|g_t\|_1^2+\|\partial_{tt}g\|_1^2.
		\end{equation}
		Combining estimates \eqref{growth-1-g}, \eqref{growth-2-g}  and \eqref{growth-3-g}, we obtain
		\begin{equation}
		\begin{split}
		\partial_t\|g_t\|_1^2+\|g_t\|_2^2&\leq  \|g_t(0)\|_2^2+\frac{2}{\Lambda}\int_0^t\|\partial_{tt}g\|_2^2ds
		+\Lambda\int_0^t\|g_t\|_2^2ds+\Lambda\|g_t\|_1^2
		\\
		&\quad+\frac{1}{\Lambda}\|\partial_{tt}g\|_1^2-\frac{1}{\Lambda}\int_0^t\|\partial_{tt}g\|_2^2ds\\
		&\leq \frac{2L_0}{\Lambda}+\Lambda\|g_t\|_1^2+\|g_t\|_2^2+\|g_t(0)\|_2^2
		+\Lambda\int_0^t\|g_t\|_2^2ds
		+\Lambda\|g_t\|_1^2\\
		&\quad-\frac{1}{\Lambda}\int_0^t\|\partial_{tt}g\|_2^2ds\\
		&\leq \frac{2L_0}{\Lambda}+2\|g_t(0)\|_2^2+2\Lambda\|g_t\|_1^2
		+2\Lambda\int_0^t\|g_t\|_2^2ds,
		\end{split}
		\end{equation} 
		where we have used the facts
		\begin{equation}
		\|g_t\|_2^2\leq \|g_t(0)\|_2^2+\frac{1}{\Lambda}\int_0^t\|\partial_{tt}g\|_2^2ds
		+\Lambda\int_0^t\|g_t\|_2^2ds.
		\end{equation}
		It follows from the Gronwall's inequality that 
		\begin{equation}
		\begin{split}
		\|g_t\|_1^2+\int_0^t\|g_t\|_2^2ds\leq e^{2\Lambda t}\|g_t(0)\|_1^2+\Big(\frac{L_0}{\Lambda^2}
		+\frac{\|g_t(0)\|_2^2}{\Lambda}\Big)(e^{2\Lambda t}-1).
		\end{split}
		\end{equation} 
		From \eqref{growth-1-g}, it follows that
		\begin{equation}
		\begin{split}
		\frac{1}{\Lambda}\|\partial_{tt}g\|_1^2+\frac{2}{\Lambda}\int_0^t\|\partial_{tt}g\|_2^2ds\leq \frac{2L_0}{\Lambda}+\Lambda\|g_t\|_1^2+\|g_t\|_2^2.
		\end{split}
		\end{equation}
		Therefore, 
		\begin{equation}
		\begin{split}
		\frac{1}{\Lambda}\|\partial_{tt}g\|_1^2+\|g_t\|_2^2\leq -\frac{2}{\Lambda}\int_0^t\|\partial_{tt}g\|_2^2ds+ \frac{2L_0}{\Lambda}+\Lambda\|g_t\|_1^2+2\|g_t\|_2^2,
		\end{split}
		\end{equation}
		which together with \eqref{growth-2-g} gives that
		\begin{equation}
		\begin{split}
		\frac{1}{\Lambda}\|\partial_{tt}g\|_1^2+\|g_t\|_2^2&\leq -\frac{2}{\Lambda}\int_0^t\|\partial_{tt}g\|_2^2ds+ \frac{2L_0}{\Lambda}+\Lambda\|g_t\|_1^2+2\|g_t(0)\|_2^2\\
		&\quad+\frac{2}{\Lambda}\int_0^t\|\partial_{tt}g\|_2^2ds+2\Lambda\int_0^t\|g_t\|_2^2ds\\
		&= \frac{2L_0}{\Lambda}+\Lambda\|g_t\|_1^2+2\|g_t(0)\|_2^2
		+2\Lambda\int_0^t\|g_t\|_2^2ds\\
		&\leq e^{2\Lambda t}\Big(2\Lambda\|g_t(0)\|_1^2+\frac{2L_0}{\Lambda}+2\|g_t(0)\|_2^2\Big).
		\end{split}
		\end{equation}
		From the definition of $L_0$  in \eqref{L-0}, we can get 
		\begin{equation}
		\begin{split}
		L_0\leq C\Big(\|g_t(0)\|_1^2+\|g_t(0)\|_2^2+\|\partial_{tt}g(0)\|_1^2+\|\widehat{ Q}_{t}(0)\|_{L^2}^2\Big),
		\end{split}
		\end{equation}
		which implies the result.
	\end{proof}
\section*{Acknowledgments}
{D. Bian is supported by an NSFC Grant (11871005).  Y. Guo  is supported  by an NSF Grant (DMS \#1810868).  
	I. Tice is supported by an NSF CAREER Grant (DMS \#1653161).

\end{document}